\documentclass[a4paper,reqno,12pt]{amsart}
\usepackage[margin=1.2in]{geometry}
\usepackage{indentfirst}
\usepackage{amsfonts}
\usepackage{amsthm}
\usepackage{amssymb}
\usepackage{amsmath}
\usepackage{thmtools}
\usepackage{tikz}
\usepackage{tikz-cd}
\usepackage{xypic}
\usepackage{upgreek}

\usepackage[disable]{todonotes}
\usepackage{enumitem}
\usetikzlibrary{decorations.markings}
\usetikzlibrary{patterns}
\usetikzlibrary{matrix,arrows,positioning}
\usepackage{color}
\usetikzlibrary{decorations.pathmorphing}
\usetikzlibrary{calc}
\usetikzlibrary{decorations.markings}
\usepackage{setspace}
\usepackage{hyperref}

\newtheorem{Lem}{Lemma}[section]
\newtheorem{Prop}[Lem]{Proposition}
\newtheorem*{Def}{Definition}

\theoremstyle{plain}
\newtheorem{Thm}[Lem]{Theorem}
\newtheorem{Cor}[Lem]{Corollary}

\theoremstyle{definition}
\declaretheorem[numbered=no,name=Example,qed={\lower-0.3ex\hbox{$\triangleleft$}}]{Ex}
\newtheorem*{Rem}{Remark}
\newtheorem*{Rems}{Remarks}

\newcommand{\Hom}{\text{\textnormal{Hom}}}

\newcommand{\End}{\text{\textnormal{End}}}
\newcommand{\Aut}{\text{\textnormal{Aut}}}
\newcommand{\pt}{\text{\textnormal{\textnormal{pt}}}}

\newcommand{\ev}{\text{\textnormal{ev}}}

\newcommand{\tr}{\text{\textnormal{tr}}}

\mathchardef\mhyphen="2D

\newcommand{\Obj}{\text{\textnormal{Obj}}}
\newcommand{\git}{/\!\!/}
\newcommand{\refl}{\text{\textnormal{ref}}}

\newcommand{\BBun}{\mathcal{B}\text{\textnormal{un}}}

\newcommand{\Map}{\text{\textnormal{Map}}}
\newcommand{\MMap}{\mathfrak{M}\text{\textnormal{ap}}}
\newcommand{\ori}{\text{\textnormal{or}}}

\newcommand{\op}{\text{\textnormal{op}}}
\newcommand{\coop}{\text{\textnormal{coop}}}
\newcommand{\pr}{\text{\textnormal{\textnormal{pr}}}}

\newcommand{\Ind}{\text{\textnormal{\textnormal{Ind}}}}
\newcommand{\Par}{\text{\textnormal{Par}}}
\newcommand{\Cob}{\mathsf{Cob}}
\newcommand{\Vect}{\mathsf{Vect}}
\newcommand{\Grpd}{\mathsf{Grpd}}
\newcommand{\Grp}{\mathsf{Grp}}
\newcommand{\id}{\text{\textnormal{id}}}
\newcommand{\Fund}{\text{\textnormal{Fund}}}
\newcommand{\TFT}{\mathsf{TFT}}

\definecolor{Myblue}{rgb}{0,0,0.6}

\makeatletter
\providecommand*{\twoheadrightarrowfill@}{%
  \arrowfill@\relbar\relbar\twoheadrightarrow
}
\providecommand*{\xtwoheadrightarrow}[2][]{%
  \ext@arrow 0579\twoheadrightarrowfill@{#1}{#2}%
}
\makeatother

\makeatletter
\newcommand{\xRightarrow}[2][]{\ext@arrow 0359\Rightarrowfill@{#1}{#2}}
\makeatother

\makeatletter
\providecommand\@dotsep{5}
\renewcommand{\listoftodos}[1][\@todonotes@todolistname]{%
  \@starttoc{tdo}{#1}}
\makeatother

\begin{document}

\title[Orientation twisted homotopy field theory]{Orientation twisted homotopy field theories and twisted unoriented Dijkgraaf--Witten theory}

\author[M.\,B. Young]{Matthew B. Young}
\address{Max Planck Institute for Mathematics\\
Vivatsgasse 7\\
53111 Bonn, Germany}
\email{myoung@mpim-bonn.mpg.de}

\date{\today}


\begin{abstract}
Given a finite $\mathbb{Z}_2$-graded group $\hat{\mathsf{G}}$ with ungraded subgroup $\mathsf{G}$ and a twisted cocycle $\hat{\lambda} \in Z^n(B \hat{\mathsf{G}}; \mathsf{U}(1)_{\pi})$ which restricts to $\lambda \in Z^n(B \mathsf{G}; \mathsf{U}(1))$, we construct a lift of $\lambda$-twisted $\mathsf{G}$-Dijkgraaf--Witten theory to an unoriented topological quantum field theory. Our construction uses a new class of homotopy field theories, which we call orientation twisted. We also introduce an orientation twisted variant of the orbifold procedure, which produces an unoriented topological field theory from an orientation twisted $\mathsf{G}$-equivariant topological field theory.
\end{abstract}

\maketitle


\setcounter{footnote}{0}

\section*{Introduction}
\addtocontents{toc}{\protect\setcounter{tocdepth}{1}}

Fix $n \geq 1$. Let $\mathsf{G}$ be a finite group and let $\lambda \in Z^n(B \mathsf{G}; \mathsf{U}(1))$ be an $n$-cocycle on the classifying groupoid $B \mathsf{G}$. The associated Dijkgraaf--Witten theory $\mathcal{Z}_{\mathsf{G}}^{\lambda}$ is an $n$-dimensional oriented topological quantum field theory. In this paper we regard $\mathcal{Z}_{\mathsf{G}}^{\lambda}$ as a once-extended theory, so that it can be evaluated on oriented manifolds of codimension $0$, $1$ and $2$. The theory $\mathcal{Z}_{\mathsf{G}}^{\lambda}$ was first introduced in \cite{dijkgraaf1990} as a finite version of Chern--Simons theory \cite{witten1989} and has since been developed from a number of different perspectives \cite{freed1993}, \cite{freed1994}, \cite{morton2015}, \cite{heuts2014}, \cite{sharma2017}. One reason for the enduring interest in Dijkgraaf--Witten theory is that it is a topological field theory which is both interesting and amenable to direct study. It is therefore beneficial to develop techniques in this finite setting before approaching more complicated topological field theories, such as Chern--Simons theory. On the other hand, Dijkgraaf--Witten theory itself has found applications in many areas of mathematics and physics, such as low dimensional topology, representation theory, tensor category theory and, more recently, condensed matter physics.

The main result of the present paper, which is stated as Theorem \ref{thm:unoriDW}, is a geometric construction of a class of unoriented lifts of Dijkgraaf--Witten theory. More precisely, for each Real structure $\hat{\mathsf{G}}$ on $\mathsf{G}$, that is, a short exact sequence of finite groups
\[
1 \rightarrow \mathsf{G} \rightarrow \hat{\mathsf{G}} \xrightarrow[]{\pi} \mathbb{Z}_2 \rightarrow 1,
\]
and lift of $\lambda$ to a $\pi$-twisted $n$-cocycle $\hat{\lambda} \in Z^n(B \hat{\mathsf{G}}; \mathsf{U}(1)_{\pi})$, we construct an $n$-dimensional unoriented topological quantum field theory $\mathcal{Z}_{\hat{\mathsf{G}}}^{\hat{\lambda}}$ whose oriented restriction is $\mathcal{Z}_{\mathsf{G}}^{\lambda}$. The theories $\mathcal{Z}_{\hat{\mathsf{G}}}^{\hat{\lambda}}$ recover as special cases the previously known unoriented lifts of Dijkgraaf--Witten theory. When $\hat{\mathsf{G}}=\mathsf{G} \times \mathbb{Z}_2$ and $\hat{\lambda} = 1$, we recover the unoriented lift determined by the stack $\BBun_{\mathsf{G}}(-)$ of principal $\mathsf{G}$-bundles (viewed as a functor out of the unoriented cobordism category) \cite{freed1993}, while for $n=2$, the trivial Real structure $\hat{\mathsf{G}} = \mathsf{G} \times \mathbb{Z}_2$ and particular $\hat{\lambda}$ we recover the state sum theories of \cite{turaev2007}, although in a different realization. On the other hand, both physical and abstract mathematical arguments, the latter involving the cobordism hypothesis \cite{lurie2009} or methods from stable homotopy theory, have been used to assert the existence of an unoriented topological field theory attached to a pair $(\hat{\mathsf{G}}, \hat{\lambda})$ as above \cite{freed2014}, \cite{kapustin2014}, \cite{freed2016}. Theorem \ref{thm:unoriDW} gives a concrete and direct geometric realization of these theories.

Our construction of $\mathcal{Z}_{\hat{\mathsf{G}}}^{\hat{\lambda}}$ is as a composition
\[
\mathcal{Z}_{\hat{\mathsf{G}}}^{\hat{\lambda}}: \Cob_{\langle n, n-1,n-2 \rangle} \xrightarrow[]{\mathcal{A}_{\hat{\mathsf{G}}}^{\hat{\lambda}}} 2\Vect_{\mathbb{C}}(\Grpd) \xrightarrow[]{\textnormal{Par}} 2\Vect_{\mathbb{C}}.
\]
Here $\Cob_{\langle n, n-1,n-2 \rangle}$ is the $n$-dimensional unoriented cobordism bicategory, $2\Vect_{\mathbb{C}}$ is the bicategory of Kapranov--Voevodsky $2$-vector spaces, $2\Vect_{\mathbb{C}}(\Grpd)$ is the bicategory of $2$-vector bundles on spans of groupoids, as constructed in \cite{morton2015}, \cite{schweigert2019}, \cite{haugseng2018}, and $\textnormal{Par}$ is the linearization functor of \cite{morton2015}, \cite{schweigert2018b}. The main problem, whose solution we outline below, is therefore to define $\mathcal{A}_{\hat{\mathsf{G}}}^{\hat{\lambda}}$, which we regard as a classical unoriented topological $\mathsf{G}$-gauge theory with Lagrangian $\hat{\lambda}$. As expected from the Lagrangian point of view, we prove that the values of $\mathcal{A}_{\hat{\mathsf{G}}}^{\hat{\lambda}}$ can be computed from the data $(\hat{\mathsf{G}}, \hat{\lambda})$ via a suitable pushforward procedure; see Theorem \ref{thm:closedPartition}. For example, for a closed unoriented (and possibly nonorientable) $(n-2)$-manifold $X$, the underlying groupoid of $\mathcal{A}_{\hat{\mathsf{G}}}^{\hat{\lambda}}(X)$ is the groupoid $\BBun_{\hat{\mathsf{G}}}^{\ori}(X)$ of principal $\hat{\mathsf{G}}$-bundles on $X$ with an $\ori_X$-twisted reduction of structure group to $\mathsf{G}$, where $\ori_X \rightarrow X$ is the orientation double cover. The additional cohomological data on $\BBun_{\hat{\mathsf{G}}}^{\ori}(X)$ is constructed from a flat $\mathsf{U}(1)$-gerbe on $\BBun_{\hat{\mathsf{G}}}^{\ori}(X)$, which is in turn obtained as the image of $\hat{\lambda}$ under an orientation twisted transgression map
\[
\uptau_{X}^{\ori}: Z^n(B \hat{\mathsf{G}}; \mathsf{U}(1)_{\pi}) \rightarrow Z^{2}(\BBun_{\hat{\mathsf{G}}}^{\ori}(X); \mathsf{U}(1)).
\]
This transgression map is introduced in Section \ref{sec:unoriTrans}.

In order to construct $\mathcal{A}_{\hat{\mathsf{G}}}^{\hat{\lambda}}$, we first introduce the notion of orientation twisted homotopy field theory, which is of independent interest. Such a theory depends on a chosen topological double cover $\Pi: T \rightarrow \hat{T}$. The relevant cobordism bicategory $\hat{T} \mhyphen \Cob^{\Pi}_{\langle n, n-1,n-2\rangle}$ is one of unoriented cobordisms with maps to $\hat{T}$ together with lifts to $\mathbb{Z}_2$-equivariant maps from the orientation double cover of the cobordism to $T$. Our definition of $\hat{T} \mhyphen \Cob^{\Pi}_{\langle n, n-1,n-2\rangle}$ is motivated by Atiyah's oriented cobordism categories with coefficients in a double cover \cite{atiyah1961} and by mathematical approaches to unoriented Wess--Zumino--Witten theory \cite{schreiber2007} and string theory with orientifolds \cite{distler2011b}. In the non-extended setting, various specializations of orientation twisted field theories have been studied under different names; see \cite{tagami2012}, \cite{sweet2013}, \cite{kapustin2017}. A key result of the present paper is Theorem \ref{thm:unoriPrim}, which associates to an $n$-cocycle $\hat{\lambda} \in Z^n(\hat{T}; \mathsf{U}(1)_{T})$, whose coefficients are twisted by the double cover $T \rightarrow \hat{T}$, an $n$-dimensional orientation twisted homotopy field theory $\mathcal{P}_{\hat{T}}^{\hat{\lambda}}$. Geometrically, $\mathcal{P}_{\hat{T}}^{\hat{\lambda}}$ can be understood as parallel transport along the Jandl $(n-1)$-gerbe $\hat{\lambda}$. The oriented theory which underlies $\mathcal{P}_{\hat{T}}^{\hat{\lambda}}$ is that determined by the pullback $\lambda \in Z^n(T; \mathsf{U}(1))$ of $\hat{\lambda}$, as constructed by Turaev \cite{turaev2010} and Turner \cite{turner2004} in the non-extended setting and by M\"uller--Woike \cite{muller2018} in the once-extended setting.

The connection between Dijkgraaf--Witten theory and homotopy field theory arises when the double cover $T \rightarrow \hat{T}$ is taken to be the (geometric realization of the) map of classifying groupoids $B \mathsf{G} \rightarrow B \hat{\mathsf{G}}$. In this setting, homotopy field theories are called $\mathsf{G}$-equivariant field theories. Our main result about $\mathsf{G}$-equivariant theories is Theorem \ref{thm:unoriOrbifold}, which is a generalization of the oriented orbifold construction of Schweigert--Woike \cite{schweigert2019}, \cite{schweigert2018}. The theorem allows us to functorially associate to each orientation twisted $\mathsf{G}$-equivariant topological quantum field theory an unoriented topological field theory with target $2\Vect_{\mathbb{C}}(\Grpd)$. The desired theory $\mathcal{A}_{\hat{\mathsf{G}}}^{\hat{\lambda}}$ can then be defined to be the orientation twisted orbifold of the $\mathsf{G}$-equivariant theory $\mathcal{P}_{B \hat{\mathsf{G}}}^{\hat{\lambda}}$.

There are a number of reasons to be interested in unoriented topological field theory in general and unoriented Dijkgraaf--Witten theory in particular. Aside from purely topological applications, unoriented topological field theories are central to the classification of symmetry protected topological phases of matter with time reversal symmetry \cite{kapustin2014}, \cite{freed2016}, \cite{barkeshli2016}, \cite{kapustin2017}, \cite{freed2018}. In this context, Dijkgraaf--Witten theories play an important role due their relation to Kitaev's quantum double model \cite{kitaev2003}, \cite{barkeshli2016}. As a rather different example, the fully-extended variant of $\mathcal{Z}_{\hat{\mathsf{G}}}^{\hat{\lambda}}$ can be seen as a physical realization of the Real representation theory of the higher categorical group determined by $(\mathsf{G}, \lambda)$ \cite{mbyoung2018c}.

The results of this paper suggest a number of follow-up problems. Perhaps the most interesting, which will be the subject of a forthcoming paper, is to study the algebraic structures encoded by $\mathcal{Z}_{\hat{\mathsf{G}}}^{\hat{\lambda}}$ in dimension three. While there is a classification of once-extended oriented topological field theories \cite{bartlett2015}, there is at present no such unoriented classification. The theories $\mathcal{Z}_{\hat{\mathsf{G}}}^{\hat{\lambda}}$ provide a simple class of examples which can be used to explicate the additional structure on the modular tensor category determined by the oriented theory. In the same vein, $\mathcal{A}_{\hat{\mathsf{G}}}^{\hat{\lambda}}$ can be used to study unoriented equivariant three dimensional theories and the resulting Real  generalizations of the $\mathsf{G}$-crossed modular tensor categories. It would be interesting to compare these algebraic structures with the input data of the recently proposed unoriented extension of Turaev--Viro--Barrett--Westbury theory \cite{barkeshli2016}, \cite{bhardwaj2017}. Finally, our geometric approach admits a natural generalization to allow for defects and boundary conditions, along the lines of the oriented case \cite{fuchs2014}. Calculations in the case of Dijkgraaf--Witten theory will shed light on the as-of-yet undeveloped theory of defects in unoriented theories. In three dimensions, this will lead to a refinement of the bimodule classification from the oriented case \cite{fuchs2013}. Finally, we expect that discrete torsion in string and $M$-theory with orientifolds (see \cite{bantay2003}, \cite{sharpe2011}) can be understood in terms of the orientation twisted orbifold construction, by first tensoring the $\mathsf{G}$-equivariant theory describing (mem)branes on a $\mathsf{G}$-space with $\mathcal{Z}_{\hat{\mathsf{G}}}^{\hat{\lambda}}$, and then orbifolding.

The structure of this paper is as follows. Section \ref{sec:backMat} contains relevant background material. In Section \ref{sec:oriTwistHQFT} we introduce orientation twisted homotopy field theories and construct the theories $\mathcal{P}^{\hat{\lambda}}_{\hat{T}}$. In Section \ref{sec:twistEquivThy} we study orientation twisted equivariant theories and describe the orientation twisted orbifold construction. Finally, in Section \ref{sec:twistUnoriDW} we construct twisted unoriented Dijkgraaf--Witten theory, studying in detail the theory in dimensions one and two.

\subsubsection*{Acknowledgements}
The author would like to thank Jeffrey Morton, Lukas M\"{u}ller and Mark Penney for discussions related to the content of this paper.

\section{Background material}
\addtocontents{toc}{\protect\setcounter{tocdepth}{2}}
\label{sec:backMat}

\subsection{Homotopy limits}
\label{sec:hFibProd}

The homotopy limit of a diagram of groupoids of the form
\[
\begin{tikzpicture}[baseline= (a).base]
\node[scale=1] (a) at (0,0){
\begin{tikzcd}
{} & \mathcal{Y} \arrow{d}[right]{G} \\
\mathcal{X} \arrow{r}[below]{F} & \mathcal{Z}
\end{tikzcd}
};
\end{tikzpicture}
\]
is denoted by $\mathcal{X} \times_{\mathcal{Z}}^h \mathcal{Y}$. This groupoid fits into a homotopy commutative diagram
\[
\begin{tikzpicture}[baseline= (a).base]
\node[scale=1] (a) at (0,0){
\begin{tikzcd}
\mathcal{X} \times^h_{\mathcal{Z}} \mathcal{Y} \arrow{r} \arrow{d} & \mathcal{Y} \arrow{d}[right]{G} \arrow[Rightarrow,shorten <= 0.75em, shorten >= 0.75em]{dl}[above left]{\eta}\\
\mathcal{X} \arrow{r}[below]{F} & \mathcal{Z} 
\end{tikzcd}
};
\end{tikzpicture}
\]
and is characterized as being $2$-universal among all such groupoids. We will use the following model for $\mathcal{X} \times_{\mathcal{Z}}^h \mathcal{Y}$. Objects are triples $(x,y; \varphi)$ consisting of $x \in \mathcal{X}$ and $y \in \mathcal{Y}$ and a morphism $\varphi: G(y) \rightarrow F(x)$. A morphism $(x,y; \varphi) \rightarrow (x^{\prime},y^{\prime}; \varphi^{\prime})$ is a pair $(f,g)$ consisting of morphisms $f: x \rightarrow x^{\prime}$ and $g: y \rightarrow y^{\prime}$ for which the diagram
\[
\begin{tikzpicture}[baseline= (a).base]
\node[scale=1] (a) at (0,0){
\begin{tikzcd}
F(x) \arrow{r}[above]{F(f)} & F(x^{\prime}) \\
G(y) \arrow{u}[left]{\varphi} \arrow{r}[below]{G(g)} & G(y^{\prime}) \arrow{u}[right]{\varphi^{\prime}}
\end{tikzcd}
};
\end{tikzpicture}
\]
commutes. The component of $\eta$ at $(x,y;\varphi) \in \mathcal{X} \times_{\mathcal{Z}}^h \mathcal{Y}$ is $\varphi$.

When $F$ is the inclusion of an object $z \in \mathcal{Z}$, the homotopy limit $\{z\} \times^h_{\mathcal{Z}} \mathcal{Y}$ is called the homotopy fibre of $F$ over $z$ and is denoted by $RF^{-1}(z)$.

We will use the following basic result below.

\begin{Lem}
\label{lem:2Fibre}
Let groups $\mathsf{G}$ and $\mathsf{H}$ act on sets $X$ and $Y$, respectively, with the latter action being trivial. Let $\pi: \mathsf{G} \rightarrow \mathsf{H}$ be a surjective group homomorphism and let $F: X \rightarrow Y$ be a $\pi$-equivariant map. Then the homotopy fibre over $y \in Y \git \mathsf{H}$ of the induced functor $\widetilde{F}: X \git \mathsf{G} \rightarrow Y \git \mathsf{H}$ is equivalent to $F^{-1}(y) \git \ker(\pi)$.
\end{Lem}

\begin{proof}
The homotopy fibre $R\widetilde{F}^{-1}(y)$ is the groupoid with objects
\[
\Obj(R\widetilde{F}^{-1}(y)) = \{(x,h) \in X \times \mathsf{H} \mid h \cdot F(x) = y\}
\]
and morphisms
\[
\Hom_{R\widetilde{F}^{-1}(y)}((x,h), (x^{\prime},h^{\prime})) =\{g \in \mathsf{G} \mid g \cdot x = x^{\prime}, \; h^{\prime} \pi(g) = h\}.
\]
The stated assumptions imply that $\Obj(R\widetilde{F}^{-1}(y)) = F^{-1}(y) \times \mathsf{H}$. The explicit form of morphisms in $RF^{-1}(y)$ then gives $R\widetilde{F}^{-1}(y) \simeq (F^{-1}(y) \times \mathsf{H}) \git \mathsf{G}$, where $\mathsf{G}$ acts on $F^{-1}(y)$ and $\mathsf{H}$ through its action on $X$ and $\pi$, respectively. Since $\pi$ is surjective, every object of $R\widetilde{F}^{-1}(y)$ is isomorphic to one of the form $(x^{\prime},e)$, where $e \in \mathsf{H}$ is the identity. The claimed equivalence follows.
\end{proof}

\subsection{Relative mapping \texorpdfstring{$2$}{}-groupoids}
\label{sec:relMapGrpd}

Denote by $I$ the closed interval $[0,1]$. Let $X$ and $Y$ be topological spaces. Recall that the mapping $2$-groupoid $\MMap^{\leq 2}(X,Y)$ is the bicategory whose objects are continuous maps $X \rightarrow Y$ and whose $1$- and $2$-morphisms are homotopies and equivalence classes of homotopies of homotopies relative $X \times \partial I$, respectively. Two homotopies $\eta, \eta^{\prime}: X \times I \times I \rightarrow Y$ relative $X \times \partial I$ are called equivalent, written $\eta \simeq \eta^{\prime}$, if they are homotopic relative $X \times \partial(I \times I)$.

Fix a topological space $B$.

\begin{Def}
Let $X \xrightarrow[]{\pi_X} B$ and $Y \xrightarrow[]{\pi_Y} B$ be topological spaces over $B$. The relative mapping $2$-groupoid $\MMap^{\leq 2}_B(X, Y)$ is the homotopy fibre over $\pi_X$ of the pseudofunctor
\[
\pi_Y \circ (-): \MMap^{\leq 2}(X, Y) \rightarrow \MMap^{\leq 2}(X, B).
\]
\end{Def}

An explicit description of $\MMap^{\leq 2}_B(X, Y)$ is as follows.\footnote{This is a categorification of the model for homotopy limits of groupoids described in Section \ref{sec:hFibProd}.}
\begin{enumerate}[label=(\roman*)]
\item Objects are pairs $(f;m)$ consisting of a continuous map $f:X \rightarrow Y$ and a homotopy $m: X \times I \rightarrow B$ from $\pi_Y \circ f$ to $\pi_X$.

\item $1$-morphisms $(f_1; m_1) \rightarrow (f_2; m_2)$ are pairs $(F; M)$ consisting of a homotopy $F: X \times I \rightarrow Y$ from $f_1$ to $f_2$ and an equivalence class of homotopies $M: X \times I \times I \rightarrow B$ relative $X \times \partial I$ from $m_2 * \pi_Y(F)$ to $m_1$, where $*$ denotes composition (concatenation) of homotopies.

\item $2$-morphisms
\[
\begin{tikzcd}[column sep=8em]
(f_1; m_1)
  \arrow[bend left=12]{r}[name=U,below]{}{(F_1; M_1)} 
  \arrow[bend right=12]{r}[name=D]{}[swap]{(F_2; M_2)}
& 
(f_2; m_2)
\arrow[shorten <=0.15pt,shorten >=0.15pt,Rightarrow,to path={(U) -- (D)}]{}
\end{tikzcd}
\]
are equivalence classes of homotopies $\varphi: X \times I \times I \rightarrow Y$ relative $X \times \partial I$ from $F_1$ to $F_2$ such that
\[
M_2 * (m_2 * \pi_Y(\varphi)) \simeq M_1.
\]
Here $m_2$ denotes the extension of $m_2: X \times I \rightarrow B$ to a map $X \times I \times I \rightarrow B$ which is constant along the second factor of $I$. Equivalently, $m_2 * \pi_Y(\varphi)$ is the whiskering of the $1$-morphism $m_2$ with the $2$-morphism $\pi_Y(\varphi)$.
\end{enumerate}

Similarly, one defines the relative mapping groupoid $\MMap_B(X,Y)$. This groupoid is equivalent to the homotopy groupoid of $\MMap^{\leq 2}_B(X,Y)$. Truncating the explicit description of $\MMap^{\leq 2}_B(X,Y)$ recovers that of $\MMap_B(X,Y)$ which arises from the model for homotopy fibres described Section \ref{sec:hFibProd}.

Finally, there is a relative mapping space $\Map_B(X, Y)$, defined as the homotopy fibre over $\pi_X$ of the continuous map of topological spaces
\[
\pi_Y \circ (-): \Map(X, Y) \rightarrow \Map(X, B).
\]

\begin{Lem}
\label{lem:asphericalMappingGrpd}
Suppose that $Y$ and $B$ are aspherical.
\begin{enumerate}[label=(\roman*)]
\item There is a canonical biequivalence $\MMap_B^{\leq 2}(X,Y) \xrightarrow[]{\sim} \MMap_B(X,Y)$.

\item The geometric realization of $\MMap_B(X, Y)$ and the topological space $\Map_B(X, Y)$ are homotopy equivalent.

\item The topological space $\Map_B(X, Y)$ is aspherical.
\end{enumerate}
\end{Lem}

\begin{proof}
It is well-known that asphericity of a topological space $Z$ implies that the canonical map $\MMap^{\leq 2}(X,Z) \rightarrow \MMap(X,Z)$ is a biequivalence. The first statement then follows from the fact that homotopy limits of biequivalent diagrams of $2$-groupoids are biequivalent.

The second statement follows from the fact that, by asphericity, $\Map(X, Y)$ and $\Map(X, B)$ are homotopy equivalent to the geometric realizations of $\MMap(X, Y)$ and $\MMap(X, B)$, respectively, together with the fact that geometric realization commutes with finite homotopy limits.

The final statement follows from the long exact sequence of homotopy groups associated to the fibration
\[
\Map_B(X, Y) \rightarrow \Map(X, Y) \xrightarrow[]{\pi_Y \circ (-)} \Map(X, B)
\]
and the asphericity of $\Map(X, Y)$ and $\Map(X, B)$.
\end{proof}

\subsection{Groupoids of principal bundles}
\label{sec:moduliBundles}

Let $M$ be a compact topological manifold, possibly with boundary. Given a finite group $\mathsf{G}$, denote by $\BBun_{\mathsf{G}}(M)$ the groupoid whose objects are principal $\mathsf{G}$-bundles $P \rightarrow M$ and whose morphisms $(P \rightarrow M) \rightarrow (P^{\prime} \rightarrow M)$ are $\mathsf{G}$-equivariant maps $P \rightarrow P^{\prime}$ which commute with the structure maps to $M$. The $\mathsf{G}$-action on the total space of a $\mathsf{G}$-bundle is from the right.

The groupoid $\BBun_{\mathsf{G}}(M)$ has a number of equivalent models. To describe the first, assume for simplicity that $M$ is connected and fix a basepoint $m_0 \in M$, which we henceforth omit from the notation. Then there is an equivalence
\[
\BBun_{\mathsf{G}}(M) \simeq \Hom_{\Grp}(\pi_1(M), \mathsf{G}) \git \mathsf{G},
\]
where $\mathsf{G}$ acts on $\Hom_{\Grp}(\pi_1(M), \mathsf{G})$ by conjugation. This equivalence sends a (necessarily flat) $\mathsf{G}$-bundle to its holonomy representation.

The second model is in terms of a classifying space $\mathbf{B} \mathsf{G}$ of $\mathsf{G}$. Namely, there is an equivalence
\[
\BBun_{\mathsf{G}}(M) \simeq \MMap(M, \mathbf{B} \mathsf{G}),
\]
a $\mathsf{G}$-bundle $P \rightarrow M$ being sent to a classifying map $f_P: M \rightarrow \mathbf{B} \mathsf{G}$ of $P$. The $\mathsf{G}$-bundle determined by a map $f: M \rightarrow \mathbf{B} \mathsf{G}$ will be denoted by $P_f \rightarrow M$.

\subsection{Spans of groupoids with local coefficients}
\label{sec:spanGrpd}

A groupoid $\mathcal{X}$ is called essentially finite if $\pi_0(\mathcal{X})$ is finite. The category of essentially finite groupoids is denoted by $\Grpd$. Unless explicitly mentioned otherwise, all groupoids in this paper are assumed to be essentially finite.

The bicategory $\mathsf{Span}(\Grpd)$ has groupoids as objects, spans of groupoids as $1$-morphisms and equivalence classes of spans of spans as $2$-morphisms. For precise definitions, see \cite{morton2011}. We require a decorated version of $\mathsf{Span}(\Grpd)$, in which objects and morphisms carry compatible local systems. A general construction in the context of $(\infty,n)$-categories has been developed in \cite{haugseng2018}. In the bicategorical setting there are explicit constructions \cite{schweigert2019}, \cite{schweigert2018b}.

Let $2\Vect_{\mathbb{C}}$ be the bicategory of Kapranov--Voevodsky $2$-vector spaces, that is, finitely semisimple $\mathbb{C}$-linear additive categories, $\mathbb{C}$-linear functors and natural transformations \cite{kapranov1994}. The Deligne product $\boxtimes$ gives $2\Vect_{\mathbb{C}}$ the structure of a symmetric monoidal bicategory.

Given a groupoid $\mathcal{X}$, the pseudofunctor bicategory $\Hom_{\mathsf{Bicat}}(\mathcal{X}, 2\Vect_{\mathbb{C}})$ is called the bicategory of $2$-vector bundles on $\mathcal{X}$ and is denoted by $2 \Vect_{\mathbb{C}}(\mathcal{X})$.

\begin{Def}[{\cite[Section 4.1]{schweigert2018b}}]
The bicategory $2\Vect_{\mathbb{C}}(\Grpd)$ is defined as follows.
\begin{itemize}
\item Objects are pairs $(\mathcal{X},\alpha)$ consisting of a groupoid $\mathcal{X}$ and a $2$-vector bundle $\alpha: \mathcal{X} \rightarrow 2\Vect_{\mathbb{C}}$.

\item A $1$-morphism $(\mathcal{X}_1,\alpha_1) \rightarrow (\mathcal{X}_2, \alpha_2)$ is a pair $(\mathcal{Y}, \beta)$ consisting of a span
\begin{equation}
\label{eq:1MorSpan}
\begin{tikzpicture}[baseline= (a).base]
\node[scale=1] (a) at (0,0){
\begin{tikzcd}[column sep=2.5em, row sep=0.25em]
{} & \mathcal{Y} \arrow{ld}[above]{s} \arrow{rd}[above]{t} & {} \\
\mathcal{X}_1 & {} & \mathcal{X}_2
\end{tikzcd}
};
\end{tikzpicture}
\end{equation}
and a $1$-morphism $\beta: s^*\alpha_1 \rightarrow t^* \alpha_2$ of $2$-vector bundles.

\item A $2$-morphism $(\mathcal{Y}_1, \beta_1) \Rightarrow (\mathcal{Y}_2, \beta_2)$ is an equivalence class of tuples $(\mathcal{Z},L,R, \gamma)$ consisting of a span of spans
\[
\begin{tikzpicture}[baseline= (a).base]
\node[scale=1] (a) at (0,0){
\begin{tikzcd}[column sep=2.0em, row sep=2.0em]
{} & \mathcal{Y}_1 \arrow{dl}[above left]{s_1} \arrow{dr}[above right]{t_1} \arrow[Rightarrow,bend right]{dd}[left]{L} \arrow[Rightarrow,bend left]{dd}[right]{R} & \\
\mathcal{X}_1 & \mathcal{Z} \arrow{u}[left]{\sigma} \arrow{d}[left]{\tau} & \mathcal{X}_2 \\
{} & \mathcal{Y}_2 \arrow{ul}[below left]{s_2} \arrow{ur}[below right]{t_2} & {}
\end{tikzcd}
};
\end{tikzpicture}
\]
and a $2$-morphism $\gamma: \alpha_2(R) \circ \sigma^* \beta_1 \Rightarrow \tau^* \beta_2 \circ \alpha_1(L)$. Two such tuples are called equivalent if there is an equivalence of spans of spans which respects the $2$-morphisms.
\end{itemize}
The composition of the $1$-morphisms $(\mathcal{X}_1, \alpha_1) \xrightarrow[]{(\mathcal{Y}_1, \beta_1)} (\mathcal{X}_2, \alpha_2)  \xrightarrow[]{(\mathcal{Y}_2, \beta_2)} (\mathcal{X}_3, \alpha_3)$ is the span
\[
\begin{tikzpicture}[baseline= (a).base]
\node[scale=1] (a) at (0,0){
\begin{tikzcd}[column sep=2.5em, row sep=0.25em]
{} & \mathcal{Y}_1 \times^h_{\mathcal{X}_2} \mathcal{Y}_2 \arrow{ld}[above left]{\tilde{s} \circ s_1} \arrow{rd}[above right]{\tilde{t} \circ t_2} & {} \\
\mathcal{X}_1 & {} & \mathcal{X}_3
\end{tikzcd}
};
\end{tikzpicture}
\]
together with the $1$-morphism of $2$-vector bundles
\[
\beta_1 \times^h_{\mathcal{X}_2} \beta_2:
\tilde{s}^* s_1^* \alpha_1 \xrightarrow[]{\tilde{s}^* \beta_1} \tilde{s}^* t_1^* \alpha_2 \rightarrow \tilde{t}^* s_2^* \alpha_2 \xrightarrow[]{\tilde{t}^* \beta_2} \tilde{t}^* t_2^* \alpha_3,
\]
the middle arrow being a coherence $2$-morphism for $\mathcal{Y} \times^h_{\mathcal{X}_2} \mathcal{Y}^{\prime}$.

Similarly, the vertical composition of the $2$-morphisms
\[
(\mathcal{Y}_1, \beta_1) \xRightarrow[]{(\mathcal{Z}_1, R_1,L_1,\gamma_1)} (\mathcal{Y}_2, \beta_2)  \xRightarrow[]{(\mathcal{Z}_2, R_2,L_2,\gamma_2)} (\mathcal{Y}_3, \beta_3)
\]
is
\[
\begin{tikzpicture}[baseline= (a).base]
\node[scale=1] (a) at (0,0){
\begin{tikzcd}[column sep=2.5em, row sep=0.25em]
& (\mathcal{Z}_1 \times^h_{\mathcal{Y}_2} \mathcal{Z}_2, \tilde{L}, \tilde{R}, \gamma \times^h_{\mathcal{Y}_2} \gamma^{\prime}) \arrow{ld}[above]{\tilde{\sigma}} \arrow{rd}[above]{\tilde{\tau}} & \\
(\mathcal{Y}_1, \beta_1) & & (\mathcal{Y}_3, \beta_3),
\end{tikzcd}
};
\end{tikzpicture}
\]
where $\gamma_1 \times^h_{\mathcal{Y}_2} \gamma_2$ is defined by a construction similar to that of $\beta_1 \times^h_{\mathcal{X}_2} \beta_2$. The reader is referred to \cite{schweigert2018b} for the definition of the horizontal composition of $2$-morphisms.
\end{Def}

Cartesian product of groupoids together with the Deligne product of $2$-vector spaces extend to define a symmetric monoidal structure on $2\Vect_{\mathbb{C}}(\Grpd)$.

\subsection{Twisted \texorpdfstring{$2$}{}-linearization}
\label{sec:twistLin}

Classical topological gauge theories can often be understood as topological field theories valued in (higher) categories of spans of groupoids with local systems. Their quantization can then be approached by post-composing with a suitable functor to a sufficiently linear target higher category \cite{freed2010}. In the $(\infty,n)$-categorical setting, such linearization functors were constructed by Haugseng \cite{haugseng2018} while in the present bicategorical setting, following earlier work of Morton \cite{morton2011}, \cite{morton2015} in the case without local systems, Schweigert--Woike \cite{schweigert2019}, \cite{schweigert2018b} constructed a symmetric monoidal pseudofunctor
\[
\Par: 2\Vect_{\mathbb{C}}(\Grpd) \rightarrow 2\Vect_{\mathbb{C}}.
\]
The pseudofunctor $\Par$ assigns to a $2$-vector bundle $\alpha: \mathcal{X} \rightarrow 2\Vect_{\mathbb{C}}$ its space of flat (or parallel) sections,
\[
\Par(\mathcal{X}, \alpha) = \Hom_{2 \Vect_{\mathbb{C}}(\mathcal{X})} (\Vect_{\mathbb{C} \vert \mathcal{X}}, \alpha),
\]
where $\Vect_{\mathbb{C} \vert \mathcal{X}}$ is the trivial $2$-line bundle on $\mathcal{X}$. The $1$-morphism \eqref{eq:1MorSpan} is assigned the pushforward along $\beta$:
\[
\Par(\mathcal{Y}, \beta) : \Par(\mathcal{X}_1,\alpha_1) \xrightarrow[]{s^*}  \Par(\mathcal{X}_1,s^*\alpha_1) \xrightarrow[]{\beta_*} \Par(\mathcal{X}_2,t^*\alpha_2) \xrightarrow[]{t_*} \Par(\mathcal{X}_2,\alpha_2).
\]
For the definition of $\Par$ on $2$-morphisms and the verification that $\Par$ is a symmetric monoidal pseudofunctor, the reader is referred to \cite[\S 4.2]{schweigert2018b}.

\section{Orientation twisted extended homotopy field theories}
\label{sec:oriTwistHQFT}

In this section we define a generalization of unoriented homotopy quantum field theory. The construction is motivated by Atiyah's oriented cobordism groups with coefficients in a double cover \cite{atiyah1961}.

\subsection{Orientations}
\label{sec:orientations}

All manifolds are assumed to be smooth and compact. We allow manifolds to have corners. A manifold with empty boundary is called closed. Unless explicitly stated, all manifolds are assumed to be unoriented, and possibly nonorientable.

Let $\mathbb{Z}_2$ be the multiplicative group $\{ \pm 1\}$. Recall that each  manifold $M$ has a canonical orientation double cover $\ori_M \rightarrow M$. We fix a classifying map $M \rightarrow \mathbf{B} \mathbb{Z}_2$ of the orientation cover which, if it will not lead to confusion, we also denote by $\ori_M$. If $M$ is of dimension $n$, then $\ori_M$ can be constructed as the composition of a classifying map $M \rightarrow \mathbf{B} \mathsf{O}_n$ of the tangent bundle $TM \rightarrow M$ with the canonical map $w_1: \mathbf{B} \mathsf{O}_n \rightarrow \mathbf{B} \mathbb{Z}_2$. The manifold $\ori_M$ is canonically oriented, and this orientation defines the fundamental class $[M] \in H_n(M, \partial M; \mathbb{Z}_{\ori_M})$ of $M$.

\subsection{(Un)oriented cobordism bicategories}
\label{sec:unoriCobord}

Let $T$ be a topological space. Denote by $T \mhyphen \Cob^{\ori}_{\langle n, n-1, n-2 \rangle}$ the bicategory of $n$-dimensional oriented (compact) cobordisms with continuous maps to $T$. An object of $T \mhyphen \Cob^{\ori}_{\langle n, n-1, n-2 \rangle}$ is a closed oriented $(n-2)$-manifold with a continuous map to $T$. A $1$-morphism is an $(n-1)$-dimensional oriented collared cobordism with a continuous map to $T$ which is compatible with the boundaries. A $2$-morphism is an $n$-dimensional oriented collared cobordism with corners with a compatible continuous map to $T$. For precise definitions, see \cite[\S I.1]{turaev2010} in the setting of cobordism categories and \cite[\S 2.1]{schweigert2018} in the once-extended setting. When $T$ is a point, $T \mhyphen \Cob^{\ori}_{\langle n, n-1, n-2 \rangle}$ reduces to the oriented cobordism bicategory $\Cob^{\ori}_{\langle n, n-1, n-2 \rangle}$ of \cite[Chapter 3]{schommer2011}.

There is an unoriented variant $T \mhyphen \Cob_{\langle n, n-1, n-2 \rangle}$ of $T \mhyphen \Cob^{\ori}_{\langle n, n-1 ,n-2 \rangle}$, defined in the same way as $T \mhyphen \Cob^{\ori}_{\langle n, n-1, n-2 \rangle}$ but with all orientation data omitted.

Disjoint union gives $T \mhyphen \Cob^{(\ori)}_{\langle n, n-1, n-2 \rangle}$ the structure of a symmetric monoidal bicategory. Forgetting orientations defines a symmetric monoidal pseudofunctor
\begin{equation}
\label{eq:forgetFunctor}
\mathcal{F}: T \mhyphen \Cob^{\ori}_{\langle n, n-1, n-2 \rangle} \rightarrow T \mhyphen \Cob_{\langle n, n-1,n-2 \rangle}.
\end{equation}

\subsection{Orientation twisted cobordism bicategories}
\label{sec:oriTwistCobord}

In this section we introduce the cobordism bicategory which underlies orientation twisted homotopy field theory.

Let $\Pi: \hat{T} \rightarrow \mathbf{B} \mathbb{Z}_2$ be a topological space over $\mathbf{B} \mathbb{Z}_2$. The map $\Pi$ classifies a double cover, which we denote by $\rho: T \rightarrow \hat{T}$.

While the following definition is rather involved, the basic idea is simple: replace all occurrences of orientations in $T \mhyphen \Cob^{\ori}_{\langle n, n-1,n-2\rangle}$ with the data of maps of spaces over $\mathbf{B} \mathbb{Z}_2$.

\begin{Def}
The orientation twisted cobordism bicategory $\hat{T} \mhyphen \Cob^{\Pi}_{\langle n, n-1, n-2 \rangle}$ is defined as follows:
\begin{itemize}[leftmargin=1\parindent]
\item An object is a triple $(X, f; h)$ consisting of a closed $(n-2)$-manifold $X$, a continuous map $f: X \rightarrow \hat{T}$ and an equivalence class of homotopies $h: X \times I \rightarrow \mathbf{B} \mathbb{Z}_2$ from $\Pi \circ f$ to $\ori_X$.

\item A $1$-morphism $(X_1, f_1; h_1) \rightarrow (X_2, f_2; h_2)$ is a triple $((Y; o_1, o_2),F; H)$ consisting of a collared cobordism
\[
\begin{tikzpicture}[baseline= (a).base]
\node[scale=1] (a) at (0,0){
\begin{tikzcd}[column sep=2.5em, row sep=0.25em]
& Y \arrow[hookleftarrow]{dr}[above]{i_2} & \\ X_1 \times [0,1) \arrow[hookrightarrow]{ur}[above]{i_1} & & X_2 \times (-1,0]
\end{tikzcd}
};
\end{tikzpicture}
\]
with equivalence classes of homotopies
\[
o_1: X_1 \times [0,1) \times I \rightarrow \mathbf{B} \mathbb{Z}_2, \qquad o_2: X_1 \times (-1,0] \times I \rightarrow \mathbf{B} \mathbb{Z}_2
\]
from\footnote{We consider $[0,1)$ with its standard orientation and thereby identify $\ori_{(-) \times [0,1)}$ with $\ori_{(-)}$. Similar identifications will be made below without comment.} $\ori_{X_k}$ to $i_k^* \ori_Y$, $k=1,2$, a continuous map $F: Y \rightarrow \hat{T}$ and an equivalence class of homotopies $H: Y \times I \rightarrow \mathbf{B} \mathbb{Z}_2$ from $\Pi \circ F$ to $\ori_Y$ such that the diagram
\[
\begin{tikzpicture}[baseline= (a).base]
\node[scale=1] (a) at (0,0){
\begin{tikzcd}[column sep=2.5em, row sep=0.25em]
& Y \arrow{dd}[left]{F} \arrow[hookleftarrow]{dr}[above right]{i_2} & \\
X_1 \times \{0\} \arrow{dr}[below left]{f_1} \arrow[hookrightarrow]{ur}[above left]{i_1} & & X_2 \times \{0\} \arrow{dl}[below right]{f_2}\\
& \hat{T} & 
\end{tikzcd}
};
\end{tikzpicture}
\]
commutes and the maps
\[
H \circ (i_{k \vert X_k \times \{0\}} \times \id_I) \; , \; o_k * h_k : X_k \times I \rightarrow B \mathbb{Z}_2
\]
are homotopic relative $X_k \times \partial I$, $k=1,2$.

\item A $2$-morphism
\[
\begin{tikzcd}[column sep=10em]
(X_1, f_1; h_1)
  \arrow[bend left=10]{r}[name=U,below]{}{((Y_1; o_{1,\bullet}),F_1;H_1)} 
  \arrow[bend right=10]{r}[name=D]{}[swap]{((Y_2; o_{2,\bullet}),F_2;H_2)}
& 
(X_2, f_2; h_2)
\arrow[shorten <=0.5pt,shorten >=0.5pt,Rightarrow,to path={(U) --(D)}]{}
\end{tikzcd}
\]
is an equivalence class of triples $((Z; \sigma_{\bullet}), \varphi; \eta)$ consisting of:
\begin{itemize}[leftmargin=1\parindent]
\item A cobordism $Z$ with corners from $Y_1$ to $Y_2$. This is a compact $\langle 2 \rangle$-manifold $Z$, with associated decomposition of its topological boundary $\partial Z = \partial_0 Z \cup \partial_1 Z$, with collars
\[
\begin{tikzpicture}[baseline= (a).base]
\node[scale=1] (a) at (0,0){
\begin{tikzcd}[column sep=2.5em, row sep=0.25em]
& Z \arrow[hookleftarrow]{dr}[above]{j_2} & \\ Y_1 \times [0,1) \arrow[hookrightarrow]{ur}[above]{j_1} & & Y_2 \times (-1,0]
\end{tikzcd}
};
\end{tikzpicture}
\]
of $\partial_0 Z$ together with equivalence classes of homotopies
\[
\sigma_1: Y_1 \times [0,1) \times I \rightarrow \mathbf{B} \mathbb{Z}_2,
\qquad
\sigma_2: Y_2 \times (-1,0] \times I \rightarrow \mathbf{B} \mathbb{Z}_2
\]
from $\ori_{Y_k}$ to $j_k^* \ori_Z$, $k=1,2$, and collars
\[
\begin{tikzpicture}[baseline= (a).base]
\node[scale=1] (a) at (0,0){
\begin{tikzcd}[column sep=2.5em, row sep=0.25em]
& Z \arrow[hookleftarrow]{dr}[above]{i^{\prime}_2} & \\ X_1 \times [0,1) \times [0,1] \arrow[hookrightarrow]{ur}[above]{i^{\prime}_1} & & X_2 \times (-1,0] \times [0,1]
\end{tikzcd}
};
\end{tikzpicture}
\]
of $\partial_1 Z$ together with equivalence classes of homotopies
\[
o^{\prime}_1: X_1 \times [0,1) \times [0,1] \times I \rightarrow \mathbf{B} \mathbb{Z}_2, \qquad
o^{\prime}_2: X_2 \times (-1,0] \times [0,1] \times I \rightarrow \mathbf{B} \mathbb{Z}_2
\]
from $\ori_{X_k}$ to $i_k^{\prime *} \ori_Z$, $k=1,2$, such that there exists an $\epsilon > 0$ for which the diagrams
\[
\begin{tikzpicture}[baseline= (a).base]
\node[scale=0.8] (a) at (0,0){
\begin{tikzcd}[column sep=6.0em, row sep=2.0em]
& Z \arrow[hookleftarrow]{dr}[above]{i^{\prime}_2} & \\ X_1 \times [0,1) \times [0, \epsilon) \arrow[hookrightarrow]{ur}[above]{i^{\prime}_1} \arrow[hookrightarrow]{r}[below]{i^{(1)}_1 \times \id_{[0,\epsilon)}} & Y_1 \times [0,\epsilon) \arrow[hookrightarrow]{u}[left]{j_1} \arrow[hookleftarrow]{r}[below]{i^{(1)}_2 \times \id_{[0,\epsilon)}} & X_2 \times (-1,0] \times [0,\epsilon)
\end{tikzcd}
};
\end{tikzpicture}
\]
and
\[
\begin{tikzpicture}[baseline= (a).base]
\node[scale=0.8] (a) at (0,0){
\begin{tikzcd}[column sep=6.0em, row sep=2.0em]
& Z \arrow[hookleftarrow]{dr}[above]{i^{\prime}_2} & \\ X_1 \times [0,1) \times (1- \epsilon,1] \arrow[hookrightarrow]{ur}[above]{i^{\prime}_1} \arrow[hookrightarrow]{r}[below]{i^{(2)}_1 \times (1-\id_{(1- \epsilon,1]})} & Y_2 \times (-\epsilon,0] \arrow[hookrightarrow]{u}[left]{j_2} \arrow[hookleftarrow]{r}[below]{i^{(2)}_2 \times (1-\id_{(1- \epsilon,1]})} & X_2 \times (-1,0] \times (1- \epsilon,1] 
\end{tikzcd}
};
\end{tikzpicture}
\]
commute. 

\item A continuous map $\varphi: Z \rightarrow \hat{T}$ such that the diagram
\[
\begin{tikzpicture}[baseline= (a).base]
\node[scale=1] (a) at (0,0){
\begin{tikzcd}[column sep=4.5em, row sep=0.25em]
& Z \arrow{dd}[left]{\varphi} \arrow[hookleftarrow]{dr}[above]{i_2^{\prime} \sqcup j_2} & \\
X_1 \times \{0\} \times [0,1] \sqcup Y_1 \arrow{dr}[below left]{f_1 \circ \pr_{X_1} \sqcup F_1} \arrow[hookrightarrow]{ur}[above]{i_1^{\prime} \sqcup j_1} & & X_2 \times \{0\} \times [0,1] \sqcup Y_2 \arrow{dl}[below right]{f_2 \circ \pr_{X_2} \sqcup F_2}\\
& \hat{T} & 
\end{tikzcd}
};
\end{tikzpicture}
\]
commutes.

\item An equivalence class of homotopies $\eta: Z \times I \rightarrow \mathbf{B} \mathbb{Z}_2$ from $\Pi \circ \varphi$ to $\ori_Z$ such that the maps
\[
\eta \circ (j_{k \vert Y_k \times \{0\}} \times \id_I) \; , \; \sigma_{k \vert Y_k \times \{0 \} \times I} * H_k : Y_k \times I \rightarrow \mathbf{B} \mathbb{Z}_2
\]
are homotopic relative $Y_k \times \partial I$, $k=1,2$, and the maps
\[
\eta \circ (i^{\prime}_{k \vert X_k \times \{0\} \times [0,1]} \times \id_I)
\;,\;
o^{\prime}_{k \vert X_k \times \{0\} \times [0,1] \times I} * h_k : X_k \times [0,1] \times I \rightarrow \mathbf{B} \mathbb{Z}_2
\]
are homotopic relative $X_k \times [0,1] \times \partial I$, $k=1,2$. Here $h_k$ is regarded as a map $X_k \times [0,1] \times I \rightarrow \mathbf{B} \mathbb{Z}_2$ which is constant along $[0,1]$.
\end{itemize}
Two triples $((Z; \sigma_{\bullet}), \varphi; \eta)$ and $((Z^{\prime}; \sigma^{\prime}_{\bullet}), \varphi^{\prime}; \eta^{\prime})$ as above are called equivalent if there exists a diffeomorphism $r: Z \xrightarrow[]{\sim} Z^{\prime}$ over $B \mathbb{Z}_2$ which respects the collars, the maps $\varphi, \varphi^{\prime}$ and the homotopies $\eta, \eta^{\prime}$.
\end{itemize}
The various compositions of $1$- and $2$-morphisms are defined completely analogously to the corresponding compositions in $T \mhyphen \Cob^{(\ori)}_{\langle n,n-1,n-2 \rangle}$.
\end{Def}

The bicategory $\hat{T} \mhyphen \Cob^{\Pi}_{\langle n, n-1,n-2 \rangle}$ is symmetric monoidal under disjoint union.

When it will not lead to confusion, we will write $Y$ in place of $(Y; o_1, o_2)$, and similarly for $Z$. We will often omit the explicit mention of collars.

\begin{Rems}
\begin{enumerate}[label=(\roman*)]
\item There is a pointed version $\hat{T}_* \mhyphen \Cob^{\Pi}_{\langle n, n-1,n-2 \rangle}$ of $\hat{T} \mhyphen \Cob^{\Pi}_{\langle n, n-1,n-2 \rangle}$ in which $\hat{T}$ is pointed, objects have a basepoint in each connected component and the map to $\hat{T}$ is pointed, and similarly for $1$- and $2$-morphisms. If $\hat{T}$ is connected, as will always be the case in this paper, then the forgetful map $\hat{T}_* \mhyphen \Cob^{\Pi}_{\langle n, n-1,n-2 \rangle} \rightarrow \hat{T} \mhyphen \Cob^{\Pi}_{\langle n, n-1,n-2 \rangle}$ is a monoidal biequivalence.

\item The above definition truncates to define a category $\hat{T} \mhyphen \Cob_{\langle n, n-1 \rangle}^{\Pi}$. This category is monoidally equivalent to the category of $1$-endomorphisms of the monoidal unit $\varnothing^{n-2}$ of $\hat{T} \mhyphen \Cob_{\langle n, n-1,n-2 \rangle}^{\Pi}$.

\item By construction, the group $\pi_0(\hat{T} \mhyphen \Cob_{\langle n, n-1 \rangle}^{\Pi})$ is isomorphic to the oriented cobordism group $MSO_{n-1}(T, \hat{T})$ of $T$ with coefficients in $\hat{T}$, as introduced by Atiyah \cite[\S 2]{atiyah1961}.
\end{enumerate}
\end{Rems}

We establish some basic properties of $\hat{T} \mhyphen \Cob^{\Pi}_{\langle n, n-1,n-2 \rangle}$.

\begin{Prop}
\label{prop:oriTwistVsUnori}
Let $\hat{T} = T \times \mathbf{B} \mathbb{Z}_2$ with $\Pi$ the projection to the second factor. Then there is a symmetric monoidal pseudofunctor
\[
\Phi: \hat{T} \mhyphen \Cob^{\Pi}_{\langle n, n-1, n-2 \rangle} \rightarrow T \mhyphen \Cob_{\langle n, n-1, n-2 \rangle}
\]
which is essentially surjective, essentially full on $1$-morphisms and locally full on $2$-morphisms.
\end{Prop}

\begin{proof}
The double cover classified by $\Pi$ is homotopy equivalent to the inclusion $T \hookrightarrow T \times \mathbf{B} \mathbb{Z}_2$ at a chosen basepoint of $\mathbf{B} \mathbb{Z}_2$. Let $\pi_T: T \times \mathbf{B} \mathbb{Z}_2 \rightarrow T$ be the projection to the first factor. Then $\Phi$ can be defined to be post-composition with $\pi_T$. More precisely, set $\Phi(X,f; h) = (X, \pi_T \circ f)$ on objects. The functor
\[
\Phi_{X_1,X_2}: 
\Hom_{\hat{T} \mhyphen \Cob^{\Pi}}((X_1,f_1; h_1),(X_2,f_2; h_2)) \rightarrow \Hom_{T \mhyphen \Cob}((X_1,\pi_T \circ f_1),(X_2, \pi_T \circ f_2))
\]
is defined on $1$-morphisms by $\Phi ((Y; o_{\bullet}),F;H) = (Y, \pi_T \circ F)$, and similarly for $2$-morphisms. The required $2$-isomorphisms $\id_{\Phi(X,f,h)} \Rightarrow \Phi_{X,X}(\id_{(X,f,h)})$ and
\[
\Phi_{X_1,X_3}(Y_2 \circ Y_1) \Rightarrow \Phi_{X_2,X_3}(Y_2) \circ \Phi_{X_1,X_2}(Y_1)
\]
can be taken to be the respective identities. There is a canonical lift of $\Phi$ to a symmetric monoidal pseudofunctor.

Given $(X, \tilde{f}) \in T \mhyphen \Cob_{\langle n, n-1, n-2 \rangle}$, we have
\[
\Phi(X, \tilde{f} \times \ori_X ; \ori_X) = (X, \tilde{f}),
\]
where we have written $\ori_X$ both for the map $X \rightarrow \mathbf{B} \mathbb{Z}_2$ and for the associated identity homotopy $X \times I \rightarrow \mathbf{B} \mathbb{Z}_2$. This shows that $\Phi$ is essentially surjective.

Consider now the functor $\Phi_{X_1,X_2}$. Let $(Y, \widetilde{F})$ be an object of the codomain of $\Phi_{X_1,X_2}$. Let $\overline{F}$ be a classifying map of the orientation cover of $Y$ which restricts to
\begin{equation}
\label{eq:boundaryClassifyingMap}
(\Pi \circ f_1) \sqcup (\Pi \circ f_2) : X_1 \sqcup X_2 \rightarrow \mathbf{B} \mathbb{Z}_2
\end{equation}
and let $H$ be a homotopy from $\overline{F}$ to $\ori_Y$. The existence of $\overline{F}$ is ensured by the fact that \eqref{eq:boundaryClassifyingMap} classifies $\ori_{\partial Y}$ while that of $H$ follows from the fact that orientation covers are unique up to equivalence. Put $o_k = (H \circ (i_k \times \id_I)) * h_k^{-1}$, $k=1,2$. Then we have
\[
\Phi_{X_1,X_2} ((Y;o_{\bullet}), \widetilde{F} \times \overline{F}; H)
=
(Y, \widetilde{F}),
\]
proving that $\Phi_{X_1,X_2}$ is essentially full on $1$-morphisms. This construction admits an obvious variation in which the $(n-1)$-cobordism is replaced with an $n$-cobordism with corners. This shows that $\Phi_{X_1,X_2}$ is locally full on $2$-morphisms.
\end{proof}

The next result provides a generalization of the forgetful map \eqref{eq:forgetFunctor}.

\begin{Prop}
\label{prop:forgetFunctor}
There is a symmetric monoidal pseudofunctor
\[
\mathcal{F}: T \mhyphen \Cob^{\ori}_{\langle n, n-1, n-2 \rangle} \rightarrow \hat{T} \mhyphen \Cob^{\Pi}_{\langle n, n-1, n-2 \rangle}.
\]
\end{Prop}

\begin{proof}
Let $\nu$ be a null-homotopy of the composition $T \xrightarrow[]{\rho} \hat{T} \xrightarrow[]{\Pi} \mathbf{B} \mathbb{Z}_2$, say to $z \in \mathbf{B} \mathbb{Z}_2$. We will interpret an orientation of a manifold $M$ as a homotopy $\omega_M$ from $z$ to $\ori_M$. With this notation, the functor $\mathcal{F}$ can be defined as follows. On objects set\footnote{For ease of notation, we have written $\omega_X * \nu$ instead of the more accurate $\omega_X * (\nu \circ (f \times \id_I))$.}
\[
\mathcal{F}(X,f) = (X, \rho \circ f; \omega_X * \nu)
\]
and on $1$-morphisms set
\[
\mathcal{F} \left( (X_1, f_1) \xrightarrow[]{(Y,F)} (X_2, f_2) \right) = (X_1, \rho \circ f_1; \omega_M * \nu) \xrightarrow[]{(Y,\rho \circ F; \omega_Y * \nu)} (X_2, \rho \circ f_2; \omega_M * \nu).
\]
The lift of $i_k :X_k \hookrightarrow Y$ to a map over $\mathbf{B} \mathbb{Z}_2$, which has been omitted from the notation, is obtained from the compatible orientations of $Y$ and $X_1 \sqcup X_2$. The definition on $2$-morphisms is analogous to that on $1$-morphisms. The additional compatibility $2$-isomorphisms and the lift of $\mathcal{F}$ to a symmetric monoidal pseudofunctor are canonical.
\end{proof}

\subsection{Extended orientation twisted homotopy field theories}
\label{sec:unoriHQFT}

For background on (extended) oriented homotopy field theories, the reader is referred to \cite{turaev2010}, \cite{muller2018}.

Let $\mathcal{C}$ be a symmetric monoidal bicategory.

\begin{Def}
A once-extended $n$-dimensional orientation twisted homotopy field theory with target $\hat{T}$ valued in $\mathcal{C}$ is a symmetric monoidal pseudofunctor
\[
\mathcal{Z} : \hat{T} \mhyphen \Cob^{\Pi}_{\langle n, n-1,n-2 \rangle} \rightarrow \mathcal{C}
\]
which is homotopy invariant in the following sense: if
\[
((Z; \sigma_{\bullet}),\varphi; \eta) \; , \; ((Z; \sigma_{\bullet}),\varphi^{\prime};\eta^{\prime})
:
(Y_1,F_1;H_1) \Rightarrow (Y_2,F_2;H_2)
\]
are $2$-morphisms in $\hat{T} \mhyphen \Cob^{\Pi}_{\langle n, n-1, n-2 \rangle}$ for which there exists a homotopy $\kappa: Z \times I \rightarrow \hat{T}$ relative $\partial Z$ from $\varphi$ to $\varphi^{\prime}$ which satisfies $\eta^{\prime} * (\Pi \circ \kappa) \simeq \eta$, then
\[
\mathcal{Z}(Z,\varphi;\eta) = \mathcal{Z}(Z,\varphi^{\prime};\eta^{\prime})
\]
as $2$-morphisms $\mathcal{Z}(Y_1,F_1;H_1) \Rightarrow \mathcal{Z}(Y_2,F_2;H_2)$.
\end{Def}

When it will not lead to confusion, we will omit the adjectives `once-extended' and `homotopy'. The $2$-groupoid $\hat{T} \mhyphen \TFT^{\Pi}_{\langle n, n-1, n-2\rangle} (\mathcal{C})$ is defined to be the full subbicategory of $\Hom_{\mathsf{Bicat}}(\hat{T} \mhyphen \mathsf{Cob}^{\Pi}_{\langle n, n-1, n-2\rangle}, \mathcal{C})$ spanned by orientation twisted homotopy field theories. When $\mathcal{C} = 2 \Vect_{\mathbb{C}}$ we simply write $\hat{T} \mhyphen \TFT^{\Pi}_{\langle n, n-1,n-2\rangle}$. Truncating the previous definition defines the groupoid $\hat{T} \mhyphen \TFT^{\Pi}_{\langle n, n-1\rangle} (\mathcal{A})$ of non-extended orientation twisted field theories valued in a monoidal category $\mathcal{A}$. Moreover, restriction to $\End_{\hat{T} \mhyphen \Cob^{\Pi}_{\langle n, n-1,n-2 \rangle}}(\varnothing^{n-2})$ defines a functor
\[
\pi_{\leq 1}(
\hat{T} \mhyphen \TFT^{\Pi}_{\langle n, n-1, n-2\rangle} (\mathcal{C})) \rightarrow \hat{T} \mhyphen \TFT^{\Pi}_{\langle n, n-1\rangle} (\End_{\mathcal{C}}(\mathbf{1}_{\mathcal{C}})),
\]
the domain being the homotopy category of $\hat{T} \mhyphen \TFT^{\Pi}_{\langle n, n-1, n-2\rangle} (\mathcal{C})$.

\begin{Def}
An orientation twisted lift of an oriented homotopy field theory
\[
\mathcal{Z} : T \mhyphen \Cob^{\ori}_{\langle n, n-1,n-2 \rangle} \rightarrow \mathcal{C}
\]
is the data of a map $\Pi: \hat{T} \rightarrow \mathbf{B} \mathbb{Z}_2$ for which $T$ is (homotopic to) the total space of the associated double cover and an orientation twisted homotopy field theory $\mathcal{Z}^{\textnormal{lift}} : \hat{T} \mhyphen \Cob_{\langle n, n-1,n-2 \rangle}^{\Pi}\rightarrow \mathcal{C}$ which makes the following diagram homotopy commute:
\[
\begin{tikzpicture}[baseline= (a).base]
\node[scale=1] (a) at (0,0){
\begin{tikzcd}[column sep=5em, row sep=2.5em]
\hat{T} \mhyphen \Cob^{\Pi}_{\langle n, n-1, n-2 \rangle}  \arrow{dr}{\mathcal{Z}^{\textnormal{lift}}} & \\ T \mhyphen \Cob^{\ori}_{\langle n, n-1,n-2 \rangle} \arrow{u}[left]{\mathcal{F}} \arrow{r}[below]{\mathcal{Z}} & \mathcal{C}.
\end{tikzcd}
};
\end{tikzpicture}
\]
\end{Def}

The following result is motivated by \cite[Proposition 2.3]{atiyah1961}.

\begin{Prop}
\label{prop:reductionToUnoriented}
Let $\hat{T} = T \times \mathbf{B} \mathbb{Z}_2$ with $\Pi$ the projection to the second factor. Then restriction along the pseudofunctor $\Phi$ from Proposition \ref{prop:oriTwistVsUnori} defines a biequivalence
\[
\Phi^*: 
T \mhyphen \TFT_{\langle n, n-1, n-2 \rangle}(\mathcal{C}) \rightarrow \hat{T} \mhyphen \TFT^{\Pi}_{\langle n, n-1, n-2 \rangle} (\mathcal{C}).
\]
\end{Prop}

\begin{proof}
Let $\mathcal{Z} : T \mhyphen \Cob_{\langle n, n-1, n-2 \rangle} \rightarrow \mathcal{C}$ be an unoriented homotopy field theory. Homotopy invariance of $\mathcal{Z}$ implies that $\mathcal{Z} \circ \Phi$ is homotopy invariant, whence $\Phi^*$ is well-defined.

Recall that a pseudofunctor is a biequivalence if and only if it is essentially surjective, essentially full on $1$-morphisms and locally fully faithful. The proof of Proposition \ref{prop:oriTwistVsUnori} shows that $\Phi$ fails to be a biequivalence only because the functors $\Phi_{X_1,X_2}$ are not faithful. More precisely, the fibre of $\Phi_{X_1,X_2}$ over a $2$-morphism $(Z, \widetilde{\varphi})$ consists of all extensions $\overline{\varphi}$ of $\widetilde{\varphi}$ to a classifying map of the orientation cover of $Z$ which restrict to $\Pi \circ (F_1 \sqcup F_2 \cup f_1 \cup f_2)$. All such extensions are homotopic relative $\partial Z$. In particular, the homotopy invariance axiom implies any $\mathcal{Z} \in \hat{T} \mhyphen \TFT^{\Pi}_{\langle n, n-1, n-2 \rangle} (\mathcal{C})$ collapses the $2$-morphism fibres of $\Phi_{X_1,X_2}$. This ensures that $\Phi^*$ is a biequivalence.
\end{proof}

We will use Proposition \ref{prop:reductionToUnoriented} to identify orientation twisted homotopy field theories with target $T \times \mathbf{B} \mathbb{Z}_2$ with unoriented homotopy field theories with target $T$. In the non-extended setting, unoriented homotopy field theories with various targets have been studied by many authors; see \cite{tagami2012}, \cite{sweet2013} and, when $T = \pt$, also \cite{karimipour1997}, \cite{alexeevski2006}, \cite{turaev2006}.

The following result is crucial for the orbifolding of orientation twisted theories.

\begin{Prop}
\label{prop:vbMappingSpace}
Let $\mathcal{Z} : \hat{T} \mhyphen \Cob^{\Pi}_{\langle n, n-1, n-2 \rangle} \rightarrow \mathcal{C}$ be an orientation twisted field theory. For each closed $(n-2)$-manifold $X$, there is an induced pseudofunctor
\[
\mathcal{R}^X_{\mathcal{Z}}: \MMap^{\leq 2}_{\mathbf{B} \mathbb{Z}_2}(X, \hat{T}) \rightarrow \mathcal{C}.
\] 
\end{Prop}

\begin{proof}
We work with the explicit description of $\MMap^{\leq 2}_{\mathbf{B} \mathbb{Z}_2}(X, \hat{T})$ from Section \ref{sec:relMapGrpd}.

At the level of objects, set $\mathcal{R}^X_{\mathcal{Z}}(f; h) = \mathcal{Z}(X, f; h)$.

Let $(F;H): (f_1; h_1) \rightarrow (f_2; h_2)$ be a $1$-morphism in $\MMap^{\leq 2}_{\mathbf{B} \mathbb{Z}_2}(X, \hat{T})$. Slightly abusively, we denote by $H: X \times I^2 \rightarrow \mathbf{B} \mathbb{Z}_2$ a chosen representative of $H$, which is thus a homotopy relative $X \times \partial I$ from $h_2 * (\Pi \circ F)$ to $h_1$. We can depict $H$ as
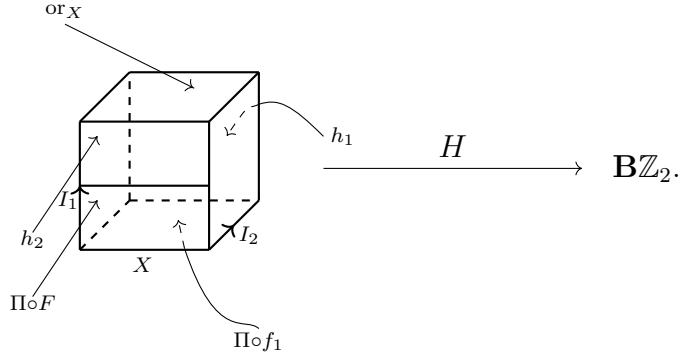
\[
\newcommand{\Depth}{2}
\newcommand{\Height}{2}
\newcommand{\Width}{2}
\begin{tikzpicture}[scale=0.85]
\coordinate (O) at (0,0,0);
\coordinate (A) at (0,\Width,0);
\coordinate (B) at (0,\Width,\Height);
\coordinate (C) at (0,0,\Height);
\coordinate (D) at (\Depth,0,0);
\coordinate (E) at (\Depth,\Width,0);
\coordinate (F) at (\Depth,\Width,\Height);
\coordinate (G) at (\Depth,0,\Height);

\draw[thick,black] (A) -- (E);
\draw[thick,black] (A) -- (B);
\draw[thick,black] (E) -- (F);
\draw[thick,black] (B) -- (F);
\draw[thick,black,decoration={markings, mark=at position 0.5 with {\arrow{>}}},        postaction={decorate}] (C) -- (B);
\draw[thick,black] (E) -- (D);
\draw[thick,black] (C) -- (G);
\draw[thick,black] (F) -- (G);
\draw[thick,black,decoration={markings, mark=at position 0.5 with {\arrow{>}}},        postaction={decorate}] (G) -- (D);
\draw[thick,black,dashed] (C) -- (O);
\draw[thick,black,dashed] (O) -- (D);
\draw[thick,black,dashed] (O) -- (A);
\draw[thick,black] (0,0.5*\Width,\Height) -- (\Depth,0.5*\Width,\Height);

\draw[thin,decoration={markings, mark=at position 1.0 with {\arrow{>}}},        postaction={decorate}] (-1.5,-1.5) to (-0.5,0);
\node at (-1.5,-1.6) {$\scriptstyle \Pi \circ F$};

\draw[thin,decoration={markings, mark=at position 1.0 with {\arrow{>}}},        postaction={decorate}] (-1.5,-0.5) to (-0.5,1);
\node at (-1.5,-0.6) {$\scriptstyle h_2$};

\draw[thin] (3,1) .. controls +(-0.5,0.5) and +(0.3,0.1) .. (2,1.5);
\node at (3.3, 1) {$\scriptstyle h_1$};

\draw[dashed,thin,decoration={markings, mark=at position 0.99 with {\arrow{>}}},        postaction={decorate}] (2,1.5) .. controls +(-0.2,-0.1) and +(0.4,0.6) .. (1.5,1.0);

\draw[thin] (2,-2) .. controls +(-0.2,0.2) and +(0.5,-1.5) .. (0.85,-0.75);
\draw[thin,dashed,decoration={markings, mark=at position 0.99 with {\arrow{>}}},        postaction={decorate}] (0.85,-0.75) to (0.75,-0.3);
\node at (2,-2.2) {$\scriptstyle \Pi \circ f_1$};

\draw[thin,decoration={markings, mark=at position 0.99 with {\arrow{>}}},        postaction={decorate}] (-1.0,2.75) to (1,1.75);
\node at (-1.0,2.95) {$\scriptstyle \ori_X$};

\node at (0.2,-1) {$\scriptstyle X$};
\node at (-0.95,0) {$\scriptstyle I_1$};
\node at (1.85,-0.55) {$\scriptstyle I_2$};

\draw[thin,decoration={markings, mark=at position 1.0 with {\arrow{>}}},        postaction={decorate}] (3,0.5) to (7.0,0.5);
\node at (8,0.5) {$\mathbf{B} \mathbb{Z}_2$.};
\node at (5,0.85) {$H$};
\end{tikzpicture}
\]
Restrictions of $H$ to various faces of $X \times I^2= X \times I_1 \times I_2$ have been indicated. Let $c_1: I^2 \rightarrow I^2$ be a continuous map which is homotopic to the identity and which takes the indicated segments to the indicated segments and corner:
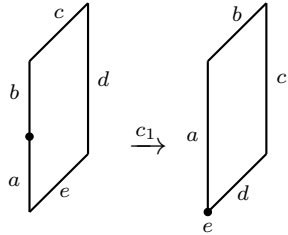
\[
\newcommand{\Depth}{2}
\newcommand{\Height}{2}
\newcommand{\Width}{2}
\begin{tikzpicture}[scale=1.0,baseline= (a).base]
\coordinate (O) at (0,0,0);
\coordinate (A) at (0,\Width,0);
\coordinate (B) at (0,\Width,\Height);
\coordinate (C) at (0,0,\Height);
\coordinate (M) at (0,0.5*\Width,\Height);

\draw[thick,black] (A) -- (B);
\draw[thick,black] (A) -- (O);
\draw[thick,black] (B) -- (C);
\draw[thick,black] (O) -- (C);

\node[draw,fill=black,scale=0.25,circle] at (M) {};

\node at (-0.2,0.2*\Width,\Height) {$\scriptstyle a$};

\node at (-0.2,0.8*\Width,\Height) {$\scriptstyle b$};

\node at (0,\Width+0.25,0.5*\Height) {$\scriptstyle c$};

\node at (0.2,0.5*\Width,0) {$\scriptstyle d$};

\node at (0,-0.2,0.4*\Height) {$\scriptstyle e$};

\end{tikzpicture}
\xrightarrow[]{c_1}
\begin{tikzpicture}[scale=1.0,scale=1.0,baseline= (a).base]
\coordinate (O) at (0,0,0);
\coordinate (A) at (0,\Width,0);
\coordinate (B) at (0,\Width,\Height);
\coordinate (C) at (0,0,\Height);
\coordinate (M) at (0,0,\Height);

\draw[thick,black] (A) -- (B);
\draw[thick,black] (A) -- (O);
\draw[thick,black] (B) -- (C);
\draw[thick,black] (O) -- (C);

\node at (-0.2,0.5*\Width,\Height) {$\scriptstyle a$};

\node at (0,\Width+0.25,0.5*\Height) {$\scriptstyle b$};

\node at (0.2,0.5*\Width,0) {$\scriptstyle c$};

\node at (0,-0.2,0.4*\Height) {$\scriptstyle d$};

\node at (0,-0.2,\Height) {$\scriptstyle e$};

\node[draw,fill=black,scale=0.25,circle] at (M) {};

\end{tikzpicture}
\]
Then $(\id_X \times c_1)^*H$ is a map $G: X \times I^2 \rightarrow \mathbf{B} \mathbb{Z}_2$ with the following indicated restrictions:
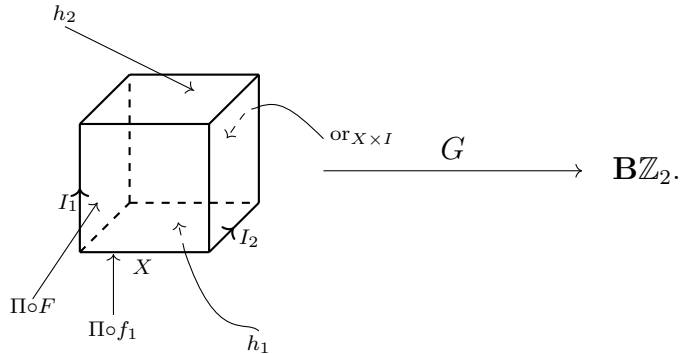
\[
\newcommand{\Depth}{2}
\newcommand{\Height}{2}
\newcommand{\Width}{2}
\begin{tikzpicture}[scale=0.85]
\coordinate (O) at (0,0,0);
\coordinate (A) at (0,\Width,0);
\coordinate (B) at (0,\Width,\Height);
\coordinate (C) at (0,0,\Height);
\coordinate (D) at (\Depth,0,0);
\coordinate (E) at (\Depth,\Width,0);
\coordinate (F) at (\Depth,\Width,\Height);
\coordinate (G) at (\Depth,0,\Height);

\draw[thick,black] (A) -- (E);
\draw[thick,black] (A) -- (B);
\draw[thick,black] (E) -- (F);
\draw[thick,black] (B) -- (F);
\draw[thick,black,decoration={markings, mark=at position 0.5 with {\arrow{>}}},        postaction={decorate}] (C) -- (B);
\draw[thick,black] (E) -- (D);
\draw[thick,black] (C) -- (G);
\draw[thick,black] (F) -- (G);
\draw[thick,black,decoration={markings, mark=at position 0.5 with {\arrow{>}}},        postaction={decorate}] (G) -- (D);
\draw[thick,black,dashed] (C) -- (O);
\draw[thick,black,dashed] (O) -- (D);
\draw[thick,black,dashed] (O) -- (A);

\draw[thin,decoration={markings, mark=at position 1.0 with {\arrow{>}}},        postaction={decorate}] (-1.5,-1.5) to (-0.5,0);
\node at (-1.5,-1.6) {$\scriptstyle \Pi \circ F$};

\draw[thin,decoration={markings, mark=at position 1.0 with {\arrow{>}}},        postaction={decorate}] (-0.25,-1.75) to (-0.25,-0.8);
\node at (-0.25,-2.0) {$\scriptstyle \Pi \circ f_1$};

\draw[thin] (3,1) .. controls +(-0.5,0.5) and +(0.3,0.1) .. (2,1.5);
\node at (3.6, 1) {$\scriptstyle \ori_{X \times I}$};

\draw[dashed,thin,decoration={markings, mark=at position 0.99 with {\arrow{>}}},        postaction={decorate}] (2,1.5) .. controls +(-0.2,-0.1) and +(0.4,0.6) .. (1.5,1.0);

\draw[thin] (2,-2) .. controls +(-0.2,0.2) and +(0.5,-1.5) .. (0.85,-0.75);
\draw[thin,dashed,decoration={markings, mark=at position 0.99 with {\arrow{>}}},        postaction={decorate}] (0.85,-0.75) to (0.75,-0.3);
\node at (2,-2.2) {$\scriptstyle h_1$};

\draw[thin,decoration={markings, mark=at position 0.99 with {\arrow{>}}},        postaction={decorate}] (-1.0,2.75) to (1,1.75);
\node at (-1.0,2.95) {$\scriptstyle h_2$};

\node at (0.2,-1) {$\scriptstyle X$};
\node at (-0.95,0) {$\scriptstyle I_1$};
\node at (1.85,-0.55) {$\scriptstyle I_2$};

\draw[thin,decoration={markings, mark=at position 1.0 with {\arrow{>}}},        postaction={decorate}] (3,0.5) to (7.0,0.5);
\node at (8,0.5) {$\mathbf{B} \mathbb{Z}_2$.};
\node at (5,0.85) {$G$};
\end{tikzpicture}
\]
Here we regard $\ori_{X \times I}$ as the map $X \times I \xrightarrow[]{\pr_X} X \xrightarrow[]{\ori_X} \mathbf{B} \mathbb{Z}_2$. The triple $(X \times I, F; G)$ therefore defines a $1$-morphism $(X, f_1; h_1) \rightarrow (X, f_2; h_2)$ in $\hat{T}\mhyphen \Cob^{\Pi}_{\langle n, n-1, n-2 \rangle}$, allowing us to set
\[
\mathcal{R}^X_{\mathcal{Z}}(F;H) = \mathcal{Z}(X \times I, F; G).
\]
Observe that if $H$ and $H^{\prime}$ are homotopic relative $X \times \partial I$, then so too are the associated maps $G$ and $G^{\prime}$, where we use the same map $c_1$ to define $G$ and $G^{\prime}$. From this we conclude that $\mathcal{R}^X_{\mathcal{Z}}$ is well-defined on $1$-morphisms.

Finally, suppose that we are given a $2$-morphism
\[
\begin{tikzcd}[column sep=10em]
(f_1; h_1)
  \arrow[bend left=15]{r}[name=U,below]{}{(F_1; H_1)} 
  \arrow[bend right=15]{r}[name=D]{}[swap]{(F_2; H_2)}
& 
(f_2; h_2)
\arrow[shorten <=1.5pt,shorten >=1.5pt,Rightarrow,to path={(U) -- node[label=left:{\footnotesize $\varphi$}] {} (D)}]{}
\end{tikzcd}
\]
in $\MMap^{\leq 2}_{\mathbf{B} \mathbb{Z}_2}(X, \hat{T})$, so that $\varphi: X \times I^2 \rightarrow \hat{T}$ is an equivalence class of homotopies relative $X \times \partial I$ from $F_1$ to $F_2$ such that $H_2 * (h_2 * (\Pi \circ \varphi)) \simeq H_1$. Fix a representative of $\varphi$ and choose a homotopy $Q: X \times I^3 \rightarrow \mathbf{B} \mathbb{Z}_2$ relative $X \times \partial (I \times I)$ realizing this equivalence. Suppressing the $X$ direction, this can be pictured as
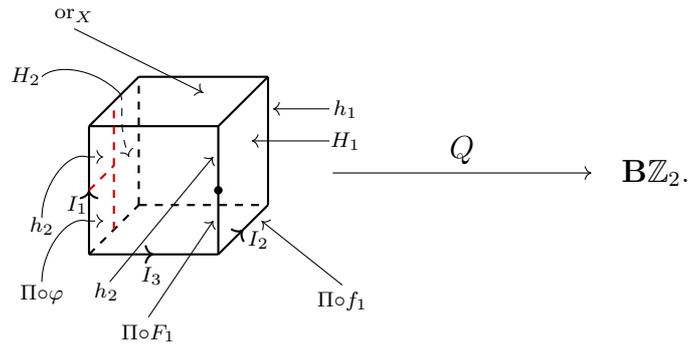
\[
\newcommand{\Depth}{2}
\newcommand{\Height}{2}
\newcommand{\Width}{2}
\begin{tikzpicture}[scale=0.85]
\coordinate (O) at (0,0,0);
\coordinate (A) at (0,\Width,0);
\coordinate (B) at (0,\Width,\Height);
\coordinate (C) at (0,0,\Height);
\coordinate (D) at (\Depth,0,0);
\coordinate (E) at (\Depth,\Width,0);
\coordinate (F) at (\Depth,\Width,\Height);
\coordinate (G) at (\Depth,0,\Height);
\coordinate (I) at (0,0.5*\Width,\Height);
\coordinate (J) at (0,0.5*\Width,0.5*\Height);
\coordinate (K) at (0,0,0.5*\Height);
\coordinate (L) at (0,\Width,0.5*\Height);
\coordinate (M) at (\Depth,0.5*\Width,\Height);

\draw[thick,black] (A) -- (E);
\draw[thick,black] (A) -- (B);
\draw[thick,black] (E) -- (F);
\draw[thick,black] (B) -- (F);
\draw[thick,black,decoration={markings, mark=at position 0.5 with {\arrow{>}}},        postaction={decorate}] (C) -- (B);
\draw[thick,black] (E) -- (D);
\draw[thick,black,decoration={markings, mark=at position 0.5 with {\arrow{>}}},        postaction={decorate}] (C) -- (G);
\draw[thick,black] (F) -- (G);
\draw[thick,black,decoration={markings, mark=at position 0.5 with {\arrow{>}}},        postaction={decorate}] (G) -- (D);
\draw[thick,black,dashed] (C) -- (O);
\draw[thick,black,dashed] (O) -- (D);
\draw[thick,black,dashed] (O) -- (A);
\draw[thick,black,dashed,red!80!black] (I) -- (J);
\draw[thick,black,dashed,red!80!black] (K) -- (L);

\draw[thin,decoration={markings, mark=at position 0.99 with {\arrow{>}}},        postaction={decorate}] (-1.0,2.75) to (1,1.75);
\node at (-1.0,2.95) {$\scriptstyle \ori_X$};

\draw[thin,decoration={markings, mark=at position 1.0 with {\arrow{>}}},        postaction={decorate}] (2.95,1.5) to (2.08,1.5);
\node at (3.2,1.5) {$\scriptstyle h_1$};

\draw[thin,decoration={markings, mark=at position 1.0 with {\arrow{>}}},        postaction={decorate}] (2.95,1.0) to (1.75,1.0);
\node at (3.2,1.0) {$\scriptstyle H_1$};

\draw[thin,decoration={markings, mark=at position 1.0 with {\arrow{>}}},        postaction={decorate}] (3.15,-1.25) to (1.9,-0.25);
\node at (3.15,-1.5) {$\scriptstyle \Pi \circ f_1$};

\draw[thin,decoration={markings, mark=at position 1.0 with {\arrow{>}}},        postaction={decorate}] (-.45,-1.20) to (1.15,0.75);
\node at (-.50,-1.35) {$\scriptstyle h_2$};

\draw[thin,decoration={markings, mark=at position 1.0 with {\arrow{>}}},        postaction={decorate}] (0.15,-1.75) to (1.15,-0.25);
\node at (0.15,-2.0) {$\scriptstyle \Pi \circ F_1$};

\draw[thin] (-1.5,-0.2) .. controls +(0,0.5) and +(-0.3,0.1) .. (-0.8,0.75);
\node at (-1.5,-0.325) {$\scriptstyle h_2$};

\draw[dashed,thin,decoration={markings, mark=at position 0.99 with {\arrow{>}}},        postaction={decorate}] (-0.8,0.75) .. controls +(0.2,0.0) and +(-0.2,0.0) .. (-0.55,0.75);

\draw[thin] (-1.5,-1.2) .. controls +(0,0.5) and +(-0.3,0.1) .. (-0.8,-0.25);
\node at (-1.5,-1.4) {$\scriptstyle \Pi \circ \varphi$};

\draw[dashed,thin,decoration={markings, mark=at position 0.99 with {\arrow{>}}},        postaction={decorate}] (-0.8,-0.25) .. controls +(0.2,0.0) and +(-0.2,0.0) .. (-0.55,-0.25);

\draw[thin] (-1.5,2.0) .. controls +(0.2,0.2) and +(-0.4,0.4) .. (-0.3,1.75);
\node at (-1.75,2.0) {$\scriptstyle H_2$};

\draw[dashed,thin,decoration={markings, mark=at position 0.99 with {\arrow{>}}},        postaction={decorate}] (-0.3,1.75) .. controls +(0.1,-0.1) and +(-0.2,0.5) .. (-0.15,0.75);

\node at (0.2,-1.1) {$\scriptstyle I_3$};
\node at (-0.95,0) {$\scriptstyle I_1$};
\node at (1.85,-0.55) {$\scriptstyle I_2$};

\draw[thin,decoration={markings, mark=at position 1.0 with {\arrow{>}}},        postaction={decorate}] (3,0.5) to (7.0,0.5);
\node at (8,0.5) {$\mathbf{B} \mathbb{Z}_2$.};
\node at (5,0.85) {$Q$};
\node[draw,fill=black,scale=0.25,circle] at (M) {};
\end{tikzpicture}
\]
There exists a map $\tilde{c}_1:I^3 \rightarrow I^3$ which is homotopic to the identity and which restricts to $c_1$ on the regions labeled by $H_1$ and $H_2$ and for which $(\id_X \times \tilde{c}_1)^*Q$ has the form
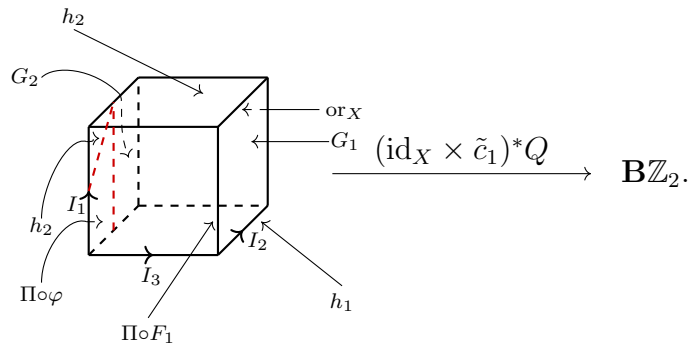
\[
\newcommand{\Depth}{2}
\newcommand{\Height}{2}
\newcommand{\Width}{2}
\begin{tikzpicture}[scale=0.85]
\coordinate (O) at (0,0,0);
\coordinate (A) at (0,\Width,0);
\coordinate (B) at (0,\Width,\Height);
\coordinate (C) at (0,0,\Height);
\coordinate (D) at (\Depth,0,0);
\coordinate (E) at (\Depth,\Width,0);
\coordinate (F) at (\Depth,\Width,\Height);
\coordinate (G) at (\Depth,0,\Height);
\coordinate (I) at (0,0.5*\Width,\Height);
\coordinate (J) at (0,0.5*\Width,0.5*\Height);
\coordinate (K) at (0,0,0.5*\Height);
\coordinate (L) at (0,\Width,0.5*\Height);
\coordinate (M) at (\Depth,0.5*\Width,\Height);

\draw[thick,black] (A) -- (E);
\draw[thick,black] (A) -- (B);
\draw[thick,black] (E) -- (F);
\draw[thick,black] (B) -- (F);
\draw[thick,black,decoration={markings, mark=at position 0.5 with {\arrow{>}}},        postaction={decorate}] (C) -- (B);
\draw[thick,black] (E) -- (D);
\draw[thick,black,decoration={markings, mark=at position 0.5 with {\arrow{>}}},        postaction={decorate}] (C) -- (G);
\draw[thick,black] (F) -- (G);
\draw[thick,black,decoration={markings, mark=at position 0.5 with {\arrow{>}}},        postaction={decorate}] (G) -- (D);
\draw[thick,black,dashed] (C) -- (O);
\draw[thick,black,dashed] (O) -- (D);
\draw[thick,black,dashed] (O) -- (A);
\draw[thick,black,dashed,red!80!black] (I) -- (L);
\draw[thick,black,dashed,red!80!black] (K) -- (L);

\draw[thin] (-1.5,-1.2) .. controls +(0,0.5) and +(-0.3,0.1) .. (-0.8,-0.25);
\node at (-1.5,-1.4) {$\scriptstyle \Pi \circ \varphi$};

\draw[dashed,thin,decoration={markings, mark=at position 0.99 with {\arrow{>}}},        postaction={decorate}] (-0.8,-0.25) .. controls +(0.2,0.0) and +(-0.2,0.0) .. (-0.55,-0.25);

\draw[thin] (-1.5,-0.2) .. controls +(0,0.5) and +(-0.3,0.1) .. (-0.8,1.0);
\node at (-1.5,-0.325) {$\scriptstyle h_2$};

\draw[dashed,thin,decoration={markings, mark=at position 0.99 with {\arrow{>}}},        postaction={decorate}] (-0.8,1.0) .. controls +(0.2,0.0) and +(-0.2,0.0) .. (-0.6,1.05);

\draw[thin] (-1.5,2.0) .. controls +(0.2,0.2) and +(-0.4,0.4) .. (-0.3,1.75);
\node at (-1.75,2.0) {$\scriptstyle G_2$};

\draw[dashed,thin,decoration={markings, mark=at position 0.99 with {\arrow{>}}},        postaction={decorate}] (-0.3,1.75) .. controls +(0.1,-0.1) and +(-0.2,0.5) .. (-0.15,0.75);

\draw[dashed,thin,decoration={markings, mark=at position 0.99 with {\arrow{>}}},        postaction={decorate}] (2,1.5) to (1.6,1.5);
\draw[thin] (2.8, 1.5) to (2,1.5);
\node at (3.2, 1.45) {$\scriptstyle \ori_X$};

\draw[thin,decoration={markings, mark=at position 1.0 with {\arrow{>}}},        postaction={decorate}] (2.95,1.0) to (1.75,1.0);
\node at (3.2,1.0) {$\scriptstyle G_1$};

\draw[thin,decoration={markings, mark=at position 1.0 with {\arrow{>}}},        postaction={decorate}] (3.15,-1.25) to (1.9,-0.25);
\node at (3.15,-1.5) {$\scriptstyle h_1$};

\draw[thin,decoration={markings, mark=at position 1.0 with {\arrow{>}}},        postaction={decorate}] (0.15,-1.75) to (1.15,-0.25);
\node at (0.15,-2.0) {$\scriptstyle \Pi \circ F_1$};

\draw[thin,decoration={markings, mark=at position 0.99 with {\arrow{>}}},        postaction={decorate}] (-1.0,2.75) to (1,1.75);
\node at (-1.0,2.95) {$\scriptstyle h_2$};

\node at (0.2,-1.1) {$\scriptstyle I_3$};
\node at (-0.95,0) {$\scriptstyle I_1$};
\node at (1.85,-0.55) {$\scriptstyle I_2$};

\draw[thin,decoration={markings, mark=at position 1.0 with {\arrow{>}}},        postaction={decorate}] (3,0.5) to (7.0,0.5);
\node at (5,0.85) {$(\id_X \times \tilde{c}_1)^*Q$};
\node at (8,0.5) {$\mathbf{B} \mathbb{Z}_2$.};
\end{tikzpicture}
\]
It is now clear that there exists a map $c_2: I^3 \rightarrow I^3$ which is homotopic to the identity and for which $R=(\id_X \times c_2)^*Q$ has the form
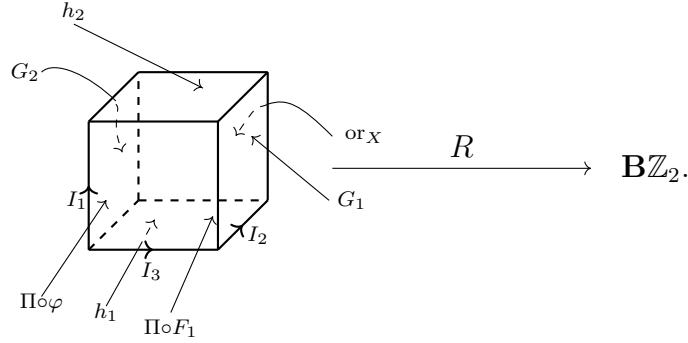
\[
\newcommand{\Depth}{2}
\newcommand{\Height}{2}
\newcommand{\Width}{2}
\begin{tikzpicture}[scale=0.85]
\coordinate (O) at (0,0,0);
\coordinate (A) at (0,\Width,0);
\coordinate (B) at (0,\Width,\Height);
\coordinate (C) at (0,0,\Height);
\coordinate (D) at (\Depth,0,0);
\coordinate (E) at (\Depth,\Width,0);
\coordinate (F) at (\Depth,\Width,\Height);
\coordinate (G) at (\Depth,0,\Height);
\coordinate (I) at (0,0.5*\Width,\Height);
\coordinate (J) at (0,0.5*\Width,0.5*\Height);
\coordinate (K) at (0,0,0.5*\Height);
\coordinate (L) at (0,\Width,0.5*\Height);
\coordinate (M) at (\Depth,0.5*\Width,\Height);

\draw[thick,black] (A) -- (E);
\draw[thick,black] (A) -- (B);
\draw[thick,black] (E) -- (F);
\draw[thick,black] (B) -- (F);
\draw[thick,black,decoration={markings, mark=at position 0.5 with {\arrow{>}}},        postaction={decorate}] (C) -- (B);
\draw[thick,black] (E) -- (D);
\draw[thick,black,decoration={markings, mark=at position 0.5 with {\arrow{>}}},        postaction={decorate}] (C) -- (G);
\draw[thick,black] (F) -- (G);
\draw[thick,black,decoration={markings, mark=at position 0.5 with {\arrow{>}}},        postaction={decorate}] (G) -- (D);
\draw[thick,black,dashed] (C) -- (O);
\draw[thick,black,dashed] (O) -- (D);
\draw[thick,black,dashed] (O) -- (A);

\draw[thin,decoration={markings, mark=at position 1.0 with {\arrow{>}}},        postaction={decorate}] (-1.5,-1.5) to (-0.5,0);
\node at (-1.5,-1.6) {$\scriptstyle \Pi \circ \varphi$};

\draw[thin,decoration={markings, mark=at position 0.99 with {\arrow{>}}},        postaction={decorate}] (-1.0,2.75) to (1,1.75);
\node at (-1.0,2.95) {$\scriptstyle h_2$};

\draw[thin] (-1.5,2.0) .. controls +(0.2,0.2) and +(-0.4,0.4) .. (-0.35,1.55);
\node at (-1.75,2.0) {$\scriptstyle G_2$};

\draw[dashed,thin,decoration={markings, mark=at position 0.99 with {\arrow{>}}},        postaction={decorate}] (-0.35,1.5) .. controls +(0.1,-0.1) and +(-0.2,0.5) .. (-0.25,0.75);

\draw[dashed,thin,decoration={markings, mark=at position 0.99 with {\arrow{>}}},        postaction={decorate}] (0,-0.75) to (0.25,-0.3);
\draw[thin] (-0.5,-1.65) to (0,-0.75);

\node at (-0.5,-1.75) {$\scriptstyle h_1$};

\draw[thin,decoration={markings, mark=at position 1.0 with {\arrow{>}}},        postaction={decorate}] (2.95,0.0) to (1.75,1.0);
\node at (3.3,0.0) {$\scriptstyle G_1$};

\draw[thin] (3,1) .. controls +(-0.5,0.5) and +(0.3,0.1) .. (2,1.5);
\node at (3.5, 1) {$\scriptstyle \ori_X$};

\draw[dashed,thin,decoration={markings, mark=at position 0.99 with {\arrow{>}}},        postaction={decorate}] (2,1.5) .. controls +(-0.2,-0.1) and +(0.4,0.6) .. (1.5,1.0);

\draw[thin,decoration={markings, mark=at position 1.0 with {\arrow{>}}},        postaction={decorate}] (0.5,-1.75) to (1.15,-0.25);
\node at (0.5,-2.0) {$\scriptstyle \Pi \circ F_1$};

\node at (0.2,-1.1) {$\scriptstyle I_3$};
\node at (-0.95,0) {$\scriptstyle I_1$};
\node at (1.85,-0.55) {$\scriptstyle I_2$};

\draw[thin,decoration={markings, mark=at position 1.0 with {\arrow{>}}},        postaction={decorate}] (3,0.5) to (7.0,0.5);
\node at (8,0.5) {$\mathbf{B} \mathbb{Z}_2$.};
\node at (5,0.85) {$R$};
\end{tikzpicture}
\]
It follows that the triple $(X \times I^2, \varphi; R)$ defines a $2$-morphism $(X \times I, H_1; G_1) \Rightarrow (X \times I, H_2; G_2)$ in $\hat{T} \mhyphen \Cob_{\langle n,n-1,n-2 \rangle}^{\Pi}$ and we set
\[
\mathcal{R}_{\mathcal{Z}}^X(\varphi) = \mathcal{Z}(X \times I^2, \varphi; R).
\]
Note that if $\varphi$ is equivalent to $\varphi^{\prime}$, in that $\varphi$ and $\varphi^{\prime}$ represent the same $2$-morphism in $\MMap^{\leq 2}_{\mathbf{B} \mathbb{Z}_2}(X, \hat{T})$, then $\mathcal{Z}(X \times I^2, \varphi; R) = \mathcal{Z}(X \times I^2, \varphi^{\prime}; R^{\prime})$ by the homotopy invariance of $\mathcal{Z}$. It is immediate that if $Q$ and $Q^{\prime}$ are equivalent, then so too are $R$ and $R^{\prime}$. So $\mathcal{R}_{\mathcal{Z}}^X$ is well-defined on $2$-morphisms.

That $\mathcal{R}^X_{\mathcal{Z}}$ defines a functor
\[
\Hom_{\MMap^{\leq 2}_{\mathbf{B} \mathbb{Z}_2}(X, \hat{T})}((f_1,h_1),(f_2,h_2)) \rightarrow \Hom_{\mathcal{C}}(\mathcal{Z}(X,f_1,h_1),\mathcal{Z}(X,f_2,h_2))
\]
can be verified directly. The unit and composition $2$-isomorphisms for $\mathcal{R}^X_{\mathcal{Z}}$ are induced by those of $\mathcal{Z}$.
\end{proof}

\subsection{Orientation twisted theories from twisted cohomology}
\label{sec:primTheories}

In this section we use cohomology with twisted coefficients to construct a basic class of examples of orientation twisted field theories.

Let $\Pi: \hat{S} \rightarrow \mathbf{B} \mathbb{Z}_2$ be a continuous map with associated double cover $S \rightarrow \hat{S}$ and let $\mathsf{A}$ be an abelian group, viewed as a $\mathbb{Z}_2$-module via inversion. Denote by $C_{\bullet}(\hat{S}; \mathsf{A}_{\Pi})$ (resp. $C^{\bullet}(\hat{S}; \mathsf{A}_{\Pi})$) the complex of singular chains (resp. cochains) on $\hat{S}$ with coefficients in the local system $\mathsf{A}_{\Pi} = S \times_{\mathbb{Z}_2} \mathsf{A} \rightarrow \hat{S}$. Pullback along $S \rightarrow \hat{S}$ defines a cochain map $C^{\bullet}(\hat{S}; \mathsf{A}_{\Pi}) \rightarrow C^{\bullet}(S; \mathsf{A})$.

Let $T$ be a topological space and let $\lambda \in Z^n(T; \mathbb{C}^{\times})$. Independently, Turaev \cite[I.2]{turaev2010} and Turner \cite{turner2004} constructed an invertible oriented homotopy field theory
\[
\mathcal{P}_T^{\lambda} : T \mhyphen \Cob^{\ori}_{\langle n, n-1 \rangle} \rightarrow \Vect_{\mathbb{C}}
\]
valued in complex vector spaces. The construction of Turaev is direct while that of Turner is in terms of higher gerbes with connection. Turaev's construction was recently extended by M\"{u}ller and Woike \cite[Theorem 3.19]{muller2018} to give a once-extended theory
\[
\mathcal{P}_T^{\lambda} : T \mhyphen \Cob^{\ori}_{\langle n, n-1,n-2 \rangle} \rightarrow 2\Vect_{\mathbb{C}}.
\]

Fix a continuous map $\Pi: \hat{T} \rightarrow \mathbf{B} \mathbb{Z}_2$ with double cover $T$ and let $\hat{\lambda} \in Z^n(\hat{T}; \mathbb{C}^{\times}_{\Pi})$ be a twisted $n$-cocycle which restricts to $\lambda$. The goal of this section is to modify the constructions of Turaev and M\"{u}ller--Woike so as to define an orientation twisted lift
\[
\mathcal{P}_{\hat{T}}^{\hat{\lambda}} : \hat{T} \mhyphen \Cob^{\Pi}_{\langle n, n-1, n-2\rangle} \rightarrow 2\Vect_{\mathbb{C}}
\]
of $\mathcal{P}_T^{\lambda}$. The basic idea is straightforward: replace fundamental chains at all stages of the constructions \cite{turaev2010}, \cite{muller2018} with their orientation twisted variants.

Let $(X,f;h)$ be an object of $\hat{T} \mhyphen \Cob^{\Pi}_{\langle n, n-1, n-2\rangle}$. Denote by $\Fund(X)$ the groupoid of fundamental cycles of $X$; objects are cycles $c_X \in Z_{n-2}(X; \mathbb{Z}_{\ori_X})$ which represent the fundamental class $[X] \in H_{n-2}(X; \mathbb{Z}_{\ori_X})$ and morphisms $c_{X,1} \rightarrow c_{X,2}$ are chains $d_X \in C_{n-1}(X; \mathbb{Z}_{\ori_X})$ which satisfy $\partial d_X = c_{X,2} - c_{X,1}$. Morphisms are composed using the abelian group structure of $C_{n-1}(X; \mathbb{Z}_{\ori_X})$. Define $\mathcal{P}_{\hat{T}}^{\hat{\lambda}}(X,f;h)$ to be the $\Vect_{\mathbb{C}}$-enriched category whose objects are formal sums
\[
\bigoplus_i W_i \cdot c_{X,i},
\]
with $W_i \in \Vect_{\mathbb{C}}$ and $c_{X,i} \in \Fund (X)$. Morphisms in $\mathcal{P}_{\hat{T}}^{\hat{\lambda}}(X,f;h)$ are defined by $\Vect_{\mathbb{C}}$-linearity and the requirement that
\[
\Hom_{\mathcal{P}_{\hat{T}}^{\hat{\lambda}}(X,f;h)}(c_{X,1}, c_{X,2}) = \mathbb{C}[\Hom_{\Fund(X)}(c_{X,1}, c_{X,2})] \slash \sim,
\]
the quotient of the free vector space on $\Hom_{\Fund(X)}(c_{X,1}, c_{X,2})$ by the relations
\[
d_{X,2} \sim \langle h(f^* \hat{\lambda}), e_X \rangle d_{X,1}
\]
whenever $e_X \in C_n(X; \mathbb{Z}_{\ori_X})$ satisfies $\partial e_X = d_{X,2} - d_{X,1}$. The notation is as follows. The homotopy $h$ induces an isomorphism of local systems $\mathbb{C}^{\times}_{\Pi \circ f} \xrightarrow[]{\sim} \mathbb{C}^{\times}_{\ori_X}$ and we write $h(f^* \hat{\lambda})$ for the image of $f^* \hat{\lambda}$ under the induced map $C^n (X; \mathbb{C}^{\times}_{\Pi \circ f}) \rightarrow C^n (X; \mathbb{C}^{\times}_{\ori_X})$. Note that if $h$ and $h^{\prime}$ are homotopic relative $X \times \partial I$, then the maps $h(-)$ and $h^{\prime}(-)$ on twisted cochains coincide. Finally, $\langle -, -\rangle$ is the canonical pairing between $\ori_X$-twisted cochains and chains.

Next, suppose that
\[
((Y; o_{\bullet}), F; H): (X_1,f_1;h_1) \rightarrow (X_2,f_2;h_2)
\]
is a $1$-morphism in $\hat{T} \mhyphen \Cob^{\Pi}_{\langle n, n-1, n-2\rangle}$. Let $\Fund(Y)$ be the groupoid of fundamental chains of $Y$; objects are chains $c_Y \in C_{n-1}(Y; \mathbb{Z}_{\ori_Y})$ which induce the fundamental class $[Y] \in H_{n-1}(Y, \partial Y; \mathbb{Z}_{\ori_Y})$ and morphisms $c_{Y,1} \rightarrow c_{Y,2}$ are chains $d_Y \in C_n(Y; \mathbb{Z}_{\ori_Y})$ which satisfy $\partial d_Y = c_{Y,2} - c_{Y,1}$. Given fundamental cycles $c_{X_1}$ and $c_{X_2}$ of $X_1$ and $X_2$, respectively, let $\Fund_{c_{X_1}}^{c_{X_2}}(Y) \subset \Fund(Y)$ be the full subgroupoid spanned by fundamental chains which satisfy $\partial c_Y = i_{2*} c_{X_2} - i_{1*} c_{X_1}$. Here we have omitted from the notation the maps $o_k$ which are used to identify the coefficients of $i_{k*} c_{X_k}$ with $\mathbb{Z}_{\ori_Y}$; similar omissions will occur below. Define
\[
Y^{(F;H)}(c_{X_1},c_{X_2}) = \mathbb{C}[\Fund_{c_{X_1}}^{c_{X_2}}(Y)] \slash \sim,
\]
where
\[
c_{Y,2} \sim \langle H(F^*\hat{\lambda}), d_Y \rangle c_{Y,1}
\]
whenever $d_Y \in C_n(Y; \mathbb{Z}_{\ori_Y})$ satisfies $\partial d_Y = c_{Y,2} - c_{Y,1}$.

Let
\[
Y^{(F;H)}(c_{X_1},-) : \Fund(X_2)^{\op} \rightarrow \Vect_{\mathbb{C}}
\]
be the functor which assigns to an object $c_{X_2}$ the vector space $Y^{(F;H)}(c_{X_1},c_{X_2})$ and which assigns to a morphism $d_X: c_{X_2,1} \rightarrow c_{X_2,2}$ the linear map
\[
Y^{(F;H)}(c_{X_1},c_{X_2,2}) \rightarrow Y^{(F;H)}(c_{X_1},c_{X_2,1}), \qquad c_Y \mapsto c_Y - i_{2*} d_X .
\]
Let us verify that this is well-defined. That this assignment defines a linear map $\mathbb{C}[\Fund_{c_{X_1}}^{c_{X_{2},2}}(Y)] \rightarrow \mathbb{C}[\Fund_{c_{X_1}}^{c_{X_2,1}}(Y)]$ follows from the calculation
\begin{eqnarray*}
\partial (c_Y - i_{2*} d_X) &=&
\partial c_Y - i_{2*} \partial d_X \\
&=& (i_{2*} c_{X_2,2} - i_{1*} c_{X_1}) - i_{2*}(c_{X_2,2} - c_{X_2,1}) \\
&=& i_{2*} c_{X_2,1} - i_{1*} c_{X_1}.
\end{eqnarray*}
That this assignment respects the equivalence relation $\sim$ follows from the observation that if $\partial d_Y = c_{Y,2} - c_{Y,1}$, then
\[
(c_{Y,2} - i_{2*} d_X) - (c_{Y,1} - i_{2*} d_X)= \partial d_Y.
\]
Continuing, $Y^{(F;H)}(c_{X_1}, -)$ can be augmented to a bifunctor
\[
Y^{(F;H)}(c_{X_1},-) \cdot (-) : \Fund(X_2)^{\op} \times \Fund(X_2) \rightarrow \mathcal{P}_{\hat{T}}^{\hat{\lambda}}(X_2,f_2;h_2)
\]
by sending an object $(c_{X_2,1}, c_{X_2,2})$ to $Y^{(F;H)}(c_{X_1},c_{X_2,1}) \cdot c_{X_2,2}$. We can then define the required functor $\mathcal{P}_{\hat{T}}^{\hat{\lambda}}((Y; o_{\bullet}), F; H)$ on objects to be the coend
\[
\mathcal{P}_{\hat{T}}^{\hat{\lambda}}((Y; o_{\bullet}), F; H) (c_{X_1}) = \int^{c_{X_2} \in \Fund(X_2)} Y^{(F;H)}(c_{X_1},c_{X_2}) \cdot (c_{X_2}).
\]
The functoriality of $Y^{(F;H)}(c_{X_1},c_{X_2})$ in $c_{X_1}$ determines the action of $\mathcal{P}_{\hat{T}}^{\hat{\lambda}}((Y; o_{\bullet}), F; H)$ on morphisms.

Finally, consider a $2$-morphism
\[
\begin{tikzcd}[column sep=15em]
(X_1, f_1; h_1)
  \arrow[bend left=12]{r}[name=U,below]{}{((Y_1; o_{1,\bullet}),F_1;H_1)} 
  \arrow[bend right=12]{r}[name=D]{}[swap]{((Y_2; o_{2,\bullet}),F_2;H_2)}
& 
(X_2, f_2; h_2).
\arrow[shorten <=3pt,shorten >=3pt,Rightarrow,to path={(U) -- node[label=left:{\scriptsize $((Z;\sigma_{\bullet}), \varphi;\eta)$}] {} (D)}]{}
\end{tikzcd}
\]
For each $c_{X_1} \in \Fund(X_1)$, let
\[
\widetilde{\mathcal{P}}_{\hat{T}}^{\hat{\lambda}}((Z;\sigma_{\bullet}), \varphi;\eta): Y_1^{(F_1;H_1)}(c_{X_1},-) \Rightarrow Y_2^{(F_2;H_2)}(c_{X_1},-)
\]
be the natural transformation whose component $\widetilde{\mathcal{P}}_{\hat{T}}^{\hat{\lambda}}((Z;\sigma_{\bullet}), \varphi;\eta)_{c_{X_2}}$ at $c_{X_2}$ is defined as follows. Given $c_{Y_1} \in \Fund_{c_{X_1}}^{c_{X_2}}(Y_1)$, choose a fundamental chain $c_Z \in C_n(Z, \partial Z; \mathbb{Z}_{\ori_Z})$ of $Z$ which satisfies
\begin{equation}
\label{eq:bigFundChain}
\partial c_Z = j_{2*} c_{Y_2} - j_{1*} c_{Y_1} +(-1)^{n-2} \left( i^{\prime}_{2*} (c_{X_2} \times [0,1]) - i^{\prime}_{1*} (c_{X_1} \times [0,1]) \right)
\end{equation}
with $c_{Y_2} \in \Fund_{c_{X_1}}^{c_{X_2}}(Y_2)$ and set
\begin{equation}
\label{eq:primTheory2Mor}
\widetilde{\mathcal{P}}_{\hat{T}}^{\hat{\lambda}}((Z;\sigma_{\bullet}), \varphi;\eta)_{c_{X_2}} (c_{Y_1}) = \langle \eta (\varphi^* \hat{\lambda}), c_Z \rangle c_{Y_2}.
\end{equation}
The natural transformation $\widetilde{\mathcal{P}}_{\hat{T}}^{\hat{\lambda}}$ induces a natural transformation of coends, thereby producing the required $2$-morphism
\[
\begin{tikzcd}[column sep=18em]
\mathcal{P}_{\hat{T}}^{\hat{\lambda}}(X_1, f_1; h_1)
  \arrow[bend left=12]{r}[name=U,below]{}{\mathcal{P}_{\hat{T}}^{\hat{\lambda}}((Y_1; o_{1,\bullet}),F_1;H_1)} 
  \arrow[bend right=12]{r}[name=D]{}[swap]{\mathcal{P}_{\hat{T}}^{\hat{\lambda}}((Y_2; o_{2,\bullet}),F_2;H_2)}
& 
\mathcal{P}_{\hat{T}}^{\hat{\lambda}}(X_2, f_2; h_2).
\arrow[shorten <=3pt,shorten >=3pt,Rightarrow,to path={(U) -- node[label=left:{\scriptsize $\mathcal{P}_{\hat{T}}^{\hat{\lambda}}((Z;\sigma_{\bullet}), \varphi;\eta)$}] {} (D)}]{}
\end{tikzcd}
\]

We can now state the main result of this section.

\begin{Thm}
\label{thm:unoriPrim}
For each continuous map $\Pi : \hat{T} \rightarrow \mathbf{B} \mathbb{Z}_2$ and twisted $n$-cocycle $\hat{\lambda} \in Z^n(\hat{T}; \mathbb{C}^{\times}_{\Pi})$, the above construction defines an invertible orientation twisted lift
\[
\mathcal{P}_{\hat{T}}^{\hat{\lambda}} : \hat{T} \mhyphen \Cob^{\Pi}_{\langle n, n-1, n-2\rangle} \rightarrow 2\Vect_{\mathbb{C}}
\]
of $\mathcal{P}_T^{\lambda}$. Moreover, the equivalence class of $\mathcal{P}_{\hat{T}}^{\hat{\lambda}}$ in $\hat{T} \mhyphen \TFT^{\Pi}_{\langle n, n-1, n-2 \rangle}$ depends only on the cohomology class $[\hat{\lambda}] \in H^n(\hat{T}; \mathbb{C}^{\times}_{\Pi})$.
\end{Thm}

\begin{proof}
The verification that $\mathcal{P}_{\hat{T}}^{\hat{\lambda}}$ is indeed a symmetric monoidal pseudofunctor proceeds as in the oriented case. The key points of the proof in the oriented case are basic properties of coends, which continue to hold without change in the present setting, and the Glueing Lemma \cite[Lemma 3.3]{muller2018}. The latter admits a straightforward modification in which homology with $\mathbb{Z}_{\ori_{(-)}}$, instead of $\mathbb{Z}$, coefficients is used. The proof is therefore very similar to that of \cite[\S 3]{muller2018} and we omit the details.

Let us verify the homotopy invariance of $\mathcal{P}_{\hat{T}}^{\hat{\lambda}}$. The argument again mirrors the oriented case. Let $\kappa: Z \times I \rightarrow \hat{T}$ be a homotopy relative $\partial Z$ from $\varphi$ to $\varphi^{\prime}$ which satisfies $\eta^{\prime} * (\Pi \circ \kappa) \simeq \eta$. Fix a homotopy $Q$ realizing this equivalence. After suppressing the $Z$ direction, the map $Q$ can be depicted as
\[
\newcommand{\Height}{2}
\newcommand{\Width}{2}
\begin{tikzpicture}[scale=0.8,baseline= (a).base]
\coordinate (O) at (0,0);
\coordinate (A) at (\Width,0);
\coordinate (B) at (\Width,\Height);
\coordinate (C) at (0,\Height);
\coordinate (M) at (0,0.5*\Height);

\draw[thick,black] (A) -- (B);
\draw[thick,black] (A) -- (O);
\draw[thick,black] (B) -- (C);
\draw[thick,black] (O) -- (C);

\node[draw,fill=black,scale=0.25,circle] at (M) {};

\node at (-0.3,0.8* \Height) {$\scriptstyle \eta^{\prime}$};

\node at (-0.5,0.2* \Height) {$\scriptstyle \Pi \circ \kappa$};

\node at (0.5*\Width, \Height+0.2) {$\scriptstyle \ori_Z$};

\node at (\Width + 0.2,0.5* \Height) {$\scriptstyle \eta$};

\node at (0.5*\Width,-0.25) {$\scriptstyle \Pi \circ \varphi$};

\node at (-0.5,0.5* \Height) {$\scriptstyle \Pi \circ \varphi^{\prime}$};

\draw[thin,decoration={markings, mark=at position 1.0 with {\arrow{>}}},        postaction={decorate}] (2.5,0.5*\Height) to (4,0.5*\Height);
\node at (3.25,0.65*\Height) {$Q$};
\node at (4.7,0.5*\Height) {$\mathbf{B} \mathbb{Z}_2$.};
\end{tikzpicture}
\]
As in the proof of Proposition \ref{prop:vbMappingSpace}, we construct from $Q$ a map
\[
\newcommand{\Height}{2}
\newcommand{\Width}{2}
\begin{tikzpicture}[scale=0.8,baseline= (a).base]
\coordinate (O) at (0,0);
\coordinate (A) at (\Width,0);
\coordinate (B) at (\Width,\Height);
\coordinate (C) at (0,\Height);
\coordinate (M) at (0,0);

\draw[thick,black] (A) -- (B);
\draw[thick,black] (A) -- (O);
\draw[thick,black] (B) -- (C);
\draw[thick,black] (O) -- (C);

\node[draw,fill=black,scale=0.25,circle] at (M) {};

\node at (0.5*\Width, \Height+0.2) {$\scriptstyle \ori_Z$};

\node at (\Width + 0.2,0.5* \Height) {$\scriptstyle \eta$};

\node at (0.5*\Width,-0.25) {$\scriptstyle \Pi \circ \kappa$};

\node at (\Width+0.2,-0.25) {$\scriptstyle \Pi \circ \varphi$};

\node at (-0.2,-0.25) {$\scriptstyle \Pi \circ \varphi^{\prime}$};

\node at (-0.4,0.5* \Height) {$\scriptstyle \eta^{\prime}$};

\draw[thin,decoration={markings, mark=at position 1.0 with {\arrow{>}}},        postaction={decorate}] (2.5,0.5*\Height) to (4,0.5*\Height);
\node at (3.25,0.65*\Height) {$R$};
\node at (4.7,0.5*\Height) {$\mathbf{B} \mathbb{Z}_2$,};
\end{tikzpicture}
\]
thereby giving a homotopy from $\Pi \circ \kappa$ to $\ori_{Z \times I}$. By functoriality of homology with local coefficients, the pair $(\varphi; \eta)$ determines a chain map
\[
(\varphi; \eta)_*: C_{\bullet}(Z; \mathbb{Z}_{\ori_Z}) \rightarrow C_{\bullet}(\hat{T}; \mathbb{Z}_{\Pi}),
\]
and similarly for $(\varphi^{\prime}; \eta^{\prime})$ and $(\kappa; R)$. The composed map
\[
r: C_{\bullet}(Z; \mathbb{Z}_{\ori_Z}) \xrightarrow[]{(-) \times [0,1]}C_{\bullet+1}(Z \times [0,1]; \mathbb{Z}_{\ori_Z \times [0,1]}) \xrightarrow[]{(\kappa; R)_*}  C_{\bullet+1}(\hat{T}; \mathbb{Z}_{\ori_{\Pi}})
\]
is a chain homotopy from $(\varphi; \eta)_*$ to $(\varphi^{\prime}; \eta^{\prime})_*$. 

Fix fundamental chains $c_{X_k} \in \Fund(X_k)$, $c_{Y_k} \in \Fund_{c_{X_1}}^{c_{X_2}}(Y_k)$ and $c_Z$ satisfying equation \eqref{eq:bigFundChain}. We compute
\begin{eqnarray*}
\langle \eta(\varphi^* \hat{\lambda}) -\eta^{\prime}(\varphi^{\prime *} \hat{\lambda}), c_Z \rangle 
&=&
\langle \hat{\lambda}, \left( (\varphi;\eta)_* - (\varphi^{\prime};\eta^{\prime})_* \right) c_Z \rangle \\
&=&
\langle \hat{\lambda}, (r_{n-1} \partial  + \partial r_n)c_Z \rangle \\
&=&
\langle \hat{\lambda}, (\kappa;R)_*(\partial c_Z \times [0,1]) \rangle\\
&=&
\langle (\kappa;R)^* \hat{\lambda}, \partial c_Z \times [0,1] \rangle \\
&=&
\langle \left( (\varphi;\eta) \circ p_{\partial Z} \right)^* \hat{\lambda}, \partial c_Z \times [0,1] \rangle \\
&=&
\langle \eta(\varphi^* \hat{\lambda}), p_{\partial Z *} (\partial c_Z \times [0,1]) \rangle.
\end{eqnarray*}
The third and fifth equalities follow from the fact that $\hat{\lambda}$ is closed and that $\kappa$ is a homotopy relative $\partial Z$, respectively. Here $p_{\partial Z} : \partial Z \times [0,1] \rightarrow \partial Z$ is the canonical projection and $p_{\partial Z *}$ denotes the associated map on twisted homology. The $n$-chain $p_{\partial Z *} (\partial c_Z \times [0,1])$ is a cycle, as follows from the construction of $c_Z$, and hence is also a boundary, as $H_n(\partial Z; \mathbb{Z}_{\ori_{\partial Z}})=0$. In view of the definition \eqref{eq:primTheory2Mor}, this completes the verification of the homotopy axiom.

The uniqueness of fundamental classes in twisted cohomology implies that $\mathcal{P}^{\hat{\lambda}}_{\hat{T}}$ factors through the Picard $2$-groupoid of $2 \Vect_{\mathbb{C}}$. This shows that the theory $\mathcal{P}^{\hat{\lambda}}_{\hat{T}}$ is invertible.

Upon restriction to $T \mhyphen \Cob_{\langle n, n-1, n-2 \rangle}^{\ori}$, the orientations trivialize all orientation double covers appearing the construction of $\mathcal{P}_{\hat{T}}^{\hat{\lambda}}$. The definition of $\mathcal{P}_{\hat{T}}^{\hat{\lambda}}$ then reduces to that of \cite{muller2018}. Hence $\mathcal{P}_{\hat{T}}^{\hat{\lambda}}$ is indeed a lift $\mathcal{P}_{T}^{\lambda}$.

Consider then the final statement. Since $\hat{T} \mhyphen \TFT^{\Pi}_{\langle n, n-1, n-2 \rangle}$ is a $2$-groupoid, it suffices to associate to each cochain $\hat{\nu} \in C^{n-1}(\hat{T}; \mathbb{C}^{\times}_{\Pi})$ a symmetric monoidal pseudonatural transformation
\[
\mathcal{Q}_{\hat{T}}^{\hat{\lambda},\hat{\nu}} : \mathcal{P}_{\hat{T}}^{\hat{\lambda}} \rightarrow \mathcal{P}_{\hat{T}}^{\hat{\lambda} \cdot d \hat{\nu}^{-1}}.
\]
This can be done as follows. Define the component of $\mathcal{Q}_{\hat{T}}^{\hat{\lambda},\hat{\nu}}$ at $(X,f;h) \in \hat{T} \mhyphen \Cob_{\langle n, n-1, n-2 \rangle}^{\Pi}$ to be the functor
\[
\mathcal{P}_{\hat{T}}^{\hat{\lambda}}(X,f;h) \rightarrow \mathcal{P}_{\hat{T}}^{\hat{\lambda} \cdot d \hat{\nu}^{-1}}(X,f;h)
\]
which is the identity on objects and which sends $d_X \in \Hom_{\Fund(X)}(c_{X,1}, c_{X,2})$ to $\langle h(f^* \hat{\nu}), d_X \rangle d_X$. Given a $1$-morphism
\[
((Y; o_{\bullet}), F; H): (X_1,f_1;h_1) \rightarrow (X_2,f_2;h_2)
\]
in $\hat{T} \mhyphen \Cob_{\langle n, n-1, n-2 \rangle}^{\Pi}$, define the compatibility $2$-morphism
\[
\begin{tikzpicture}[baseline= (a).base]
\node[scale=1] (a) at (0,0){
\begin{tikzcd}[column sep=10em, row sep=2.5em]
\mathcal{P}_{\hat{T}}^{\hat{\lambda}}(X_1) \arrow{r}[above]{\mathcal{P}_{\hat{T}}^{\hat{\lambda}}(Y)} \arrow{d}[left]{\mathcal{Q}_{\hat{T}}^{\hat{\lambda},\hat{\nu}}(X_1)}& \mathcal{P}_{\hat{T}}^{\hat{\lambda}}(X_2) \arrow{d}{\mathcal{Q}_{\hat{T}}^{\hat{\lambda},\hat{\nu}}(X_2)}  \arrow[Rightarrow,shorten <= 0.65em, shorten >= 0.65em]{dl}[above left]{\mathcal{Q}_{\hat{T}}^{\hat{\lambda},\hat{\nu}}(Y)} \\
\mathcal{P}_{\hat{T}}^{\hat{\lambda} \cdot d \hat{\nu}^{-1}}(X_1) \arrow{r}[below]{\mathcal{P}_{\hat{T}}^{\hat{\lambda} \cdot d \hat{\nu}^{-1}}(Y)} &  \mathcal{P}_{\hat{T}}^{\hat{\lambda} \cdot d \hat{\nu}^{-1}}(X_2)
\end{tikzcd}
};
\end{tikzpicture}
\]
to be that induced by the $\mathbb{C}$-linear map
\[
\mathbb{C}[\Fund_{c_{X_1}}^{c_{X_2}}(Y)] \slash \sim_{\hat{\lambda}} \rightarrow \mathbb{C}[\Fund_{c_{X_1}}^{c_{X_2}}(Y)] \slash \sim_{\hat{\lambda} \cdot d \hat{\nu}^{-1}}, \qquad c_Y \mapsto \langle H(F^* \hat{\nu}), c_Y \rangle c_Y.
\]
The subscripts on $\sim$ indicate the $n$-cocycle used to define the equivalence relation on $\mathbb{C}[\Fund_{c_{X_1}}^{c_{X_2}}(Y)]$. More precisely, these linear maps are the components of a natural transformation
\[
Y^{\hat{\lambda},(F;H)}(c_{X_1}, -) \Rightarrow Y^{\hat{\lambda} \cdot d \hat{\nu}^{-1},(F;H)}(c_{X_1}, -)
\]
which in turn induces the required morphism of coends. The modifications which encode the compatibility of $\mathcal{Q}_{\hat{T}}^{\hat{\lambda},\hat{\nu}}$ with the monoidal structure can be taken to be the identities. 
\end{proof}

\section{Orientation twisted equivariant field theories and orbifolding}
\label{sec:twistEquivThy}

In this section we study the simplest class of orientation twisted field theories, that in which the target $\hat{T}$ is aspherical. Concretely, we take $\hat{T}$ to be the classifying space of a finite $\mathbb{Z}_2$-graded group. This leads to an interpretation in terms of equivariant field theories.

\subsection{Finite \texorpdfstring{$\mathbb{Z}_2$}{}-graded groups}

Let $\Grp$ be the category of finite groups. The slice category $\Grp_{\slash \mathbb{Z}_2}$ is the category of finite $\mathbb{Z}_2$-graded groups. The identity map $\mathbb{Z}_2 \xrightarrow[]{\id} \mathbb{Z}_2$ is a terminal object of $\Grp_{\slash \mathbb{Z}_2}$. Objects of $\Grp_{\slash \mathbb{Z}_2}$ will be denoted by $\pi_{\hat{\mathsf{G}}}: \hat{\mathsf{G}} \rightarrow \mathbb{Z}_2$. We write $\pi$ for $\pi_{\hat{\mathsf{G}}}$ if it will not cause confusion. If $\pi$ is non-trivial, which we will assume to be the case unless explicitly mentioned otherwise, then $\hat{\mathsf{G}}$ is an extension
\[
1 \rightarrow \mathsf{G} \xrightarrow[]{i} \hat{\mathsf{G}} \xrightarrow[]{\pi} \mathbb{Z}_2 \rightarrow 1.
\]
The map $\pi$ induces a morphism of classifying groupoids $B \pi: B \hat{\mathsf{G}} \rightarrow B \mathbb{Z}_2$. The associated groupoid double cover is equivalent to $Bi: B \mathsf{G} \rightarrow B \hat{\mathsf{G}}$. Passing to classifying spaces gives a map $\mathbf{B} \pi: \mathbf{B} \hat{\mathsf{G}} \rightarrow \mathbf{B} \mathbb{Z}_2$ with associated double cover $\mathbf{B} i: \mathbf{B} \mathsf{G} \rightarrow \mathbf{B} \hat{\mathsf{G}}$.

\subsection{Groupoids of orientation twisted principal bundles}
\label{sec:oriFramedBundles}

In this section we introduce the groupoid of orientation twisted principal $\mathsf{G}$-bundles on a manifold. These groupoids are central to our construction of unoriented Dijkgraaf--Witten theory.

Let $M$ be a manifold with classifying map $\ori_M: M \rightarrow \mathbf{B} \mathbb{Z}_2$. Fix a $\mathbb{Z}_2$-graded group $\pi: \hat{\mathsf{G}} \rightarrow \mathbb{Z}_2$. The homomorphism $\pi$ induces a functor
\[
\Ind_{\pi}: \BBun_{\hat{\mathsf{G}}}(M) \rightarrow \BBun_{\mathbb{Z}_2}(M),
\]
a principal $\hat{\mathsf{G}}$-bundle $P \rightarrow M$ being sent to the double cover $P \times_{\hat{\mathsf{G}}}^{\pi} \mathbb{Z}_2 \rightarrow M$.

\begin{Def}
The groupoid $\BBun_{\hat{\mathsf{G}}}^{\ori}(M)$ of orientation twisted principal $\mathsf{G}$-bundles on $M$ is the homotopy fibre of the functor $\Ind_{\pi}$ over $\ori_M \rightarrow M$.
\end{Def}

Objects of $\BBun_{\hat{\mathsf{G}}}^{\ori}(M)$ are thus pairs $(P, \epsilon)$ consisting of a principal $\hat{\mathsf{G}}$-bundle $P \rightarrow M$ and an isomorphism $\epsilon: P \times^{\pi}_{\hat{\mathsf{G}}} \mathbb{Z}_2 \xrightarrow[]{\sim} \ori_M$ of double covers, which we call an orientation framing of $P$. A morphism $(P, \epsilon) \rightarrow (P^{\prime}, \epsilon^{\prime})$ is a morphism $f: P \rightarrow P^{\prime}$ of principal $\hat{\mathsf{G}}$-bundles which satisfies $\epsilon^{\prime} \circ \Ind_{\pi}(f) = \epsilon$.

\begin{Prop}
\label{prop:restrFunctor}
A morphism $(k; h): (M, \ori_M) \rightarrow (N, \ori_N)$ of smooth manifolds over $\mathbf{B} \mathbb{Z}_2$ induces a functor $(k; h)^* : \BBun_{\hat{\mathsf{G}}}^{\ori}(N) \rightarrow \BBun_{\hat{\mathsf{G}}}^{\ori}(M)$. 
\end{Prop}

\begin{proof}
Let $(P, \epsilon) \in \BBun_{\hat{\mathsf{G}}}^{\ori}(M)$. The composition
\[
\Ind_{\pi}(k^*P) \simeq k^* \Ind_{\pi}(P) \xrightarrow[]{f^* \epsilon} k^* \ori_N \xrightarrow[]{h} \ori_M
\]
defines an orientation framing of $k^*P$, where the unnamed isomorphism is canonical. This defines $(k; h)^*$ on objects. On morphisms $(k; h)^*$ acts as pullback by $k$.
\end{proof}

\begin{Prop}
A morphism $\hat{\phi}: \hat{\mathsf{G}} \rightarrow \hat{\mathsf{H}}$ of finite $\mathbb{Z}_2$-graded groups induces a functor $\hat{\phi}_*: \BBun_{\hat{\mathsf{G}}}^{\ori}(M) \rightarrow \BBun_{\hat{\mathsf{H}}}^{\ori}(M)$.
\end{Prop}

\begin{proof}
Given $(P, \epsilon) \in \BBun_{\hat{\mathsf{G}}}^{\ori}(M)$, let $\hat{\phi}_*(P, \epsilon) = (P \times_{\hat{\mathsf{G}}} \hat{\mathsf{H}}, \epsilon)$ where, by a slight abuse of notation, the map $\epsilon$ on the right hand side denotes the composition
\[
(P \times^{\hat{\phi}}_{\hat{\mathsf{G}}} \hat{\mathsf{H}} )\times^{\pi_{\hat{\mathsf{H}}}}_{\hat{\mathsf{H}}} \mathbb{Z}_2 \simeq P \times^{\pi_{\hat{\mathsf{G}}}}_{\hat{\mathsf{G}}} \mathbb{Z}_2 \xrightarrow[]{\epsilon} \ori_M.
\]
At the level of morphisms let $\hat{\phi}_*(f) = f \times^{\hat{\phi}}_{\hat{\mathsf{G}}} \id_{\hat{\mathsf{H}}}$.
\end{proof}

The following three propositions give alternative models of $\BBun_{\hat{\mathsf{G}}}^{\ori}(M)$.

\begin{Prop}
\label{prop:holonomyDescr}
Let $M$ be connected with chosen basepoint. There is an equivalence
\[
\BBun_{\hat{\mathsf{G}}}^{\ori}(M) \simeq \Hom_{\Grp}^{\ori_M}(\pi_1(M), \hat{\mathsf{G}}) \git \mathsf{G},
\]
where $\Hom_{\Grp}^{\ori_M}(\pi_1(M), \hat{\mathsf{G}})$ denotes the fibre of the map
\[
\pi \circ (-): 
\Hom_{\Grp}(\pi_1(M), \hat{\mathsf{G}}) \rightarrow \Hom_{\Grp}(\pi_1(M), \mathbb{Z}_2)
\]
over $\ori_M$ and $\mathsf{G}$ acts on $\Hom_{\Grp}^{\ori_M}(\pi_1(M), \hat{\mathsf{G}})$ by conjugation.
\end{Prop}

\begin{proof}
As recalled in Section \ref{sec:moduliBundles}, there are equivalences
\[
\BBun_{\hat{\mathsf{G}}}(M) \simeq \Hom_{\Grp}(\pi_1(M), \hat{\mathsf{G}}) \git \hat{\mathsf{G}},
\qquad
\BBun_{\mathbb{Z}_2}(M) \simeq \Hom_{\Grp}(\pi_1(M), \mathbb{Z}_2) \git \mathbb{Z}_2.
\]
The claimed equivalence now follows by applying Lemma \ref{lem:2Fibre} to the map $\pi \circ (-)$.
\end{proof}

\begin{Prop}
\label{prop:mappingSpaceDescr}
There is an equivalence $\BBun_{\hat{\mathsf{G}}}^{\ori}(M) \simeq \MMap_{\mathbf{B} \mathbb{Z}_2}(M, \mathbf{B} \hat{\mathsf{G}})$.
\end{Prop}

\begin{proof}
As recalled in Section \ref{sec:moduliBundles}, there are equivalences
\[
\BBun_{\hat{\mathsf{G}}}(M) \simeq \MMap(M, \mathbf{B} \hat{\mathsf{G}}),
\qquad
\BBun_{\mathbb{Z}_2}(M) \simeq \MMap(M, \mathbf{B} \mathbb{Z}_2).
\]
Up to natural isomorphism, these equivalences intertwine the functors $\Ind_{\pi}$ and $\mathbf{B} \pi \circ (-)$. The  claimed equivalence now follows from the fact that homotopy limits of equivalent diagrams of groupoids are equivalent.
\end{proof}

\begin{Prop}
\label{prop:redStrDescr}
The groupoid $\BBun_{\hat{\mathsf{G}}}^{\ori}(M)$ is equivalent to the category of principal $\hat{\mathsf{G}}$-bundles $P \rightarrow M$ with a section of $\Ind_{\pi}(P) \otimes_{\mathbb{Z}_2} \ori_M \rightarrow M$ and their section preserving morphisms.
\end{Prop}

\begin{proof}
The notation $\otimes_{\mathbb{Z}_2}$ indicates the monoidal structure on $\BBun_{\mathbb{Z}_2}(M)$ induced by the abelian group structure of $\mathbb{Z}_2$. After fixing an identification of $\ori_M \otimes_{\mathbb{Z}_2} \ori_M$ with the trivial double cover $M \times \mathbb{Z}_2$, an orientation framing of $P$ can be used to construct a map
\[
\epsilon \otimes_{\mathbb{Z}_2} \ori_M: \Ind_{\pi}(P) \otimes_{\mathbb{Z}_2} \ori_M \xrightarrow[]{\epsilon} \ori_M \otimes_{\mathbb{Z}_2} \ori_M \xrightarrow[]{\sim} M \times \mathbb{Z}_2,
\]
which determines the required section. In this way we obtain a functor from $\BBun_{\hat{\mathsf{G}}}^{\ori}(M)$ to the category $\hat{\mathsf{G}}$-bundles together with a section of $\Ind_{\pi}(P) \otimes_{\mathbb{Z}_2} \ori_M$. Reversing the above construction defines a quasi-inverse.
\end{proof}

Recall that sections of $\Ind_{\pi}(P) \rightarrow M$ are equivalent to reductions of structure group of $P$ from $\hat{\mathsf{G}}$ to $\mathsf{G}$. From this perspective, Proposition \ref{prop:redStrDescr} gives an interpretation of $\BBun_{\hat{\mathsf{G}}}^{\ori}(M)$ as a groupoid of principal $\hat{\mathsf{G}}$-bundles together with an $\ori_M$-twisted reduction of structure group to $\mathsf{G}$. This explains their naming.

We now describe two situations in which $\BBun_{\hat{\mathsf{G}}}^{\ori}(M)$ reduces to a familiar groupoid.

\begin{Prop}
\label{prop:orientableFrBun}
Suppose that $M$ is orientable. The choice of an orientation $\omega_M$ of $M$ induces an equivalence
\[
\varphi^{\omega_M}: 
\BBun_{\mathsf{G}}(M) \xrightarrow[]{\sim} \BBun_{\hat{\mathsf{G}}}^{\ori}(M).
\]
\end{Prop}

\begin{proof}
Interpret the orientation of $M$ as an isomorphism $\omega_M: M \times \mathbb{Z}_2 \xrightarrow[]{\sim} \ori_M$. Then $\omega_M$ determines a homotopy commutative diagram
\[
\begin{tikzpicture}[baseline= (a).base]
\node[scale=1] (a) at (0,0){
\begin{tikzcd}[column sep=5em, row sep=2.5em]
\BBun_{\mathsf{G}}(M) \arrow{r}[above]{\Ind_i} \arrow{d} & \BBun_{\hat{\mathsf{G}}}(M) \arrow{d}[right]{\Ind_{\pi}} \\
\{ \ori_M \} \arrow{r} \arrow[Rightarrow,shorten <= 0.75em, shorten >= 0.75em]{ur}[above]{\eta^{\omega_M}} & \BBun_{\mathbb{Z}_2}(M).
\end{tikzcd}
};
\end{tikzpicture}
\]
The component of $\eta^{\omega_M}$ at $Q \in \BBun_{\mathsf{G}}(M)$ is the composition
\[
\ori_M \xrightarrow[]{\omega_M^{-1}} M \times \mathbb{Z}_2 \simeq Q \times_{\mathsf{G}}^{\pi \circ i} \mathbb{Z}_2 \simeq (Q \times^i_{\mathsf{G}} \hat{\mathsf{G}}) \times_{\hat{\mathsf{G}}}^{\pi} \mathbb{Z}_2,
\]
where the unnamed isomorphisms are canonical. By the $2$-universal property of $\BBun_{\hat{\mathsf{G}}}^{\ori}(M)$, there is an induced functor $\varphi^{\omega_M}: \BBun_{\mathsf{G}}(M) \rightarrow \BBun_{\hat{\mathsf{G}}}^{\ori}(M)$ which is compatible with the above diagram.

A quasi-inverse of $\varphi^{\omega_M}$ is constructed as follows. Let $(P, \epsilon) \in \BBun_{\hat{\mathsf{G}}}^{\ori}(M)$. The orientation framing $\epsilon$ and orientation $\omega_M$ determine a section $s(\epsilon, \omega_M)$ of $\Ind_{\pi}(P) \rightarrow M$ through the composition
\[
\Ind_{\pi}(P) \xrightarrow[]{\epsilon} \ori_M \xrightarrow[]{\omega_M} M \times \mathbb{Z}_2.
\]
The assignment $(P, \epsilon) \mapsto s(\epsilon, \omega_M)^*(P \rightarrow P \slash \mathsf{G})$ extends to a functor $\BBun_{\hat{\mathsf{G}}}^{\ori}(M) \rightarrow \BBun_{\mathsf{G}}(M)$ which is a quasi-inverse of $\varphi^{\omega_M}$.
\end{proof}

\begin{Prop}
\label{prop:splitFrBun}
Suppose that $\hat{\mathsf{G}} = \mathsf{G} \times \mathbb{Z}_2$ with $\pi$ the projection to the second factor. Then there is a canonical equivalence $\BBun_{\hat{\mathsf{G}}}^{\ori}(M) \xrightarrow[]{\sim} \BBun_{\mathsf{G}}(M)$.
\end{Prop}

\begin{proof}
Since $\hat{\mathsf{G}} = \mathsf{G} \times \mathbb{Z}_2$, a $\hat{\mathsf{G}}$-bundle $P$ is the data of a $\mathsf{G}$-bundle $Q \rightarrow M$ and a double cover $E \rightarrow M$.  The assignment $(P, \epsilon) \mapsto Q$ extends to a functor $\BBun_{\hat{\mathsf{G}}}^{\ori}(M) \rightarrow \BBun_{\mathsf{G}}(M)$. On the other hand, the $2$-universal property of $\BBun_{\hat{\mathsf{G}}}^{\ori}(M)$ implies that the commutative diagram
\[
\begin{tikzpicture}[baseline= (a).base]
\node[scale=1] (a) at (0,0){
\begin{tikzcd}[column sep=5em, row sep=2.5em]
\BBun_{\mathsf{G}}(M) \arrow{r}[above]{(\Ind_i, \ori_M)} \arrow{d} \arrow{d} & \BBun_{\hat{\mathsf{G}}}(M) \arrow{d}[right]{\Ind_{\pi}}\\
\{ \ori_M \} \arrow{r} & \BBun_{\mathbb{Z}_2}(M)
\end{tikzcd}
};
\end{tikzpicture}
\]
induces a functor $\BBun_{\mathsf{G}}(M) \rightarrow \BBun_{\hat{\mathsf{G}}}^{\ori}(M)$. It is straightforward to verify that this functor is quasi-inverse to that constructed above.
\end{proof}

\begin{Rems}
\begin{enumerate}[label=(\roman*)]
\item The mapping space interpretation of $\BBun_{\hat{\mathsf{G}}}^{\ori}(M)$ given in Proposition \ref{prop:mappingSpaceDescr} can also be used to prove of Propositions \ref{prop:orientableFrBun} and \ref{prop:splitFrBun}.

\item Orientation twisted principal bundles are examples of the twisted bundles introduced in \cite{maier2012}, where the groupoid $\BBun_{\hat{\mathsf{G}}}^{\ori}(M)$ would be denoted by $\mathcal{A}_{\hat{\mathsf{G}}}(\ori_M \rightarrow M)$. Propositions \ref{prop:holonomyDescr} and \ref{prop:orientableFrBun} generalize to arbitrary twisted bundles \cite[Propositions 3.7 and 3.8]{maier2012}.
\end{enumerate}
\end{Rems}

\begin{Ex}
We use Proposition \ref{prop:holonomyDescr} to give explicit models of groupoids of orientation twisted $\mathsf{G}$-bundles in simple cases. We fix basepoints and orientations where necessary without comment.

\begin{enumerate}[label=(\roman*)]
\item Let $\mathbb{T}^n \simeq (S^1)^n$ be an $n$-dimension torus. There is an equivalence
\[
\BBun_{\hat{\mathsf{G}}}^{\ori}(\mathbb{T}^n) \simeq \Lambda^n B \mathsf{G} \simeq \mathsf{G}^{(n)} \git \mathsf{G},
\]
where $\Lambda^n B \mathsf{G}$ is the $n$-fold loop groupoid of $B \mathsf{G}$ and $\mathsf{G}^{(n)} \subset \mathsf{G}^n$ is the subset of commuting $n$-tuples. Here, and in the examples which follow, the action of $\mathsf{G}$ is by conjugation. 

\item Let $\mathbb{M}$ be the crosscap, that is, the complement of an open disk in the real projective plane $\mathbb{RP}^2$. The double cover $\ori_{\mathbb{M}}$ is a cylinder and the map $\pi_1(\ori_{\mathbb{M}}) \simeq \mathbb{Z} \rightarrow \pi_1(\mathbb{M}) \simeq \mathbb{Z}$ is multiplication by two. It follows that there is an equivalence
\begin{equation}
\label{eq:crosscapGroupoid}
\BBun_{\hat{\mathsf{G}}}^{\ori}(\mathbb{M}) \simeq (\hat{\mathsf{G}} \backslash \mathsf{G}) \git \mathsf{G}.
\end{equation}

\item As $\pi_1(\mathbb{RP}^2) \simeq \mathbb{Z}_2$ and the holonomy representation of $\ori_{\mathbb{RP}^2}$ sends the generator to $-1$, there is an equivalence
\begin{equation}
\label{eq:projPlaneGroupoid}
\BBun^{\ori}_{\hat{\mathsf{G}}}(\mathbb{RP}^2) \simeq \{ \omega \in \hat{\mathsf{G}} \backslash \mathsf{G} \mid \omega^2 =e\} \git \mathsf{G}.
\end{equation}

\item Let $\mathbb{K}$ be the Klein bottle. The double cover $\ori_{\mathbb{K}}$ is the torus $\mathbb{T}^2$. Writing $\pi_1(\mathbb{T}^2)= \langle A, B \vert ABA^{-1}B^{-1} \rangle$ and $\pi_1(\mathbb{K}) \simeq \langle a, b \mid abab^{-1} \rangle$, the covering $\ori_{\mathbb{K}} \rightarrow \mathbb{K}$ induces the homomorphism
\[
\pi_1(\mathbb{T}^2) \rightarrow \pi_1(\mathbb{K}), \qquad A \mapsto a, \;\; B \mapsto b^2.
\]
It follows that there is an equivalence
\begin{equation}
\label{eq:kleinGroupoid}
\BBun^{\ori}_{\hat{\mathsf{G}}}(\mathbb{K}) \simeq \{ (g,\varsigma) \in \mathsf{G} \times (\hat{\mathsf{G}} \backslash \mathsf{G}) \mid \varsigma g^{-1} \varsigma^{-1} = g\} \git \mathsf{G}.
\end{equation}

\item Below we will consider $\mathbb{T}^2$ and $\mathbb{K}$ as comprising the one loop sector of the theory. By parts (i) and (iv), the one loop moduli space
\[
\BBun_{\hat{\mathsf{G}}}^{\ori, 1 \mhyphen \textnormal{loop}} = \BBun_{\hat{\mathsf{G}}}^{\ori} (\mathbb{T}^2) \sqcup \BBun_{\hat{\mathsf{G}}}^{\ori} (\mathbb{K})
\]
is equivalent to $\hat{\mathsf{G}}^{(2)} \git \mathsf{G}$, where $\hat{\mathsf{G}}^{(2)} = \{ (g, \omega) \in \mathsf{G} \times \hat{\mathsf{G}} \mid \omega g^{\pi(\omega)} \omega^{-1}= g\}$. There is a double cover $\hat{\mathsf{G}}^{(2)} \git \mathsf{G} \rightarrow \hat{\mathsf{G}}^{(2)} \git \hat{\mathsf{G}}$, where $\hat{\mathsf{G}}$ acts by Real conjugation and conjugation on $\mathsf{G}$ and $\hat{\mathsf{G}}$, respectively. The groupoid $\hat{\mathsf{G}}^{(2)} \git \hat{\mathsf{G}}$ is equivalent to $\Lambda \Lambda_{\pi}^{\refl} B \hat{\mathsf{G}}$, the loop groupoid of the unoriented loop groupoid of $B \mathsf{G}$, which plays an important role in the Real (categorical) representation theory of $\mathsf{G}$ \cite{mbyoung2018c}.
\end{enumerate}
\end{Ex}

\subsection{Orientation twisted transgression}
\label{sec:unoriTrans}

Let $M$ be a smooth compact manifold. For notational simplicity, we will assume that $M$ has constant dimension. Consider the following correspondence of stacks:
\[
\begin{tikzpicture}[baseline= (a).base]
\node[scale=1] (a) at (0,0){
\begin{tikzcd}[column sep=2.5em, row sep=0.25em]
{} & \BBun_{\hat{\mathsf{G}}}^{\ori}(M) \times M \arrow{dl}[above left]{\ev^{\ori}_{\hat{\mathsf{G}}}} \arrow{dr}[above right]{\pr_1} & \\
B \hat{\mathsf{G}} &{} & \BBun_{\hat{\mathsf{G}}}^{\ori}(M).
\end{tikzcd}
};
\end{tikzpicture}
\]
The map $\ev_{\hat{\mathsf{G}}}^{\ori}$ is defined to be the composition
\[
\ev^{\ori}_{\hat{\mathsf{G}}} : \BBun_{\hat{\mathsf{G}}}^{\ori}(M) \times M \xrightarrow[]{\sim} \MMap_{B \mathbb{Z}_2}(M, B \hat{\mathsf{G}}) \times M  \xrightarrow[]{\ev} B \hat{\mathsf{G}},
\]
where the first map is the isomorphism of Proposition \ref{prop:mappingSpaceDescr}. The map $\pr_1$ is the projection to the first factor and so is proper and representable. Denote by $\pr_2 : \BBun_{\hat{\mathsf{G}}}^{\ori}(M) \times M \rightarrow M$ the projection to the second factor.

\begin{Prop}
\label{prop:pullbackTwist}
The double cover $\pr_2^* \ori_M \rightarrow \BBun_{\hat{\mathsf{G}}}^{\ori}(M) \times M$ is equivalent to the homotopy limit of the diagram
\[
\begin{tikzpicture}[baseline= (a).base]
\node[scale=1] (a) at (0,0){
\begin{tikzcd}
{} & B \mathsf{G} \arrow{d}{B i} \\
\BBun_{\hat{\mathsf{G}}}^{\ori}(M) \times M \arrow{r}[below]{\ev^{\ori}_{\hat{\mathsf{G}}}} &  B \hat{\mathsf{G}}.
\end{tikzcd}
};
\end{tikzpicture}
\]
\end{Prop}

\begin{proof}
As the double cover $Bi$ is classified by $B\pi: B \hat{\mathsf{G}} \rightarrow B \mathbb{Z}_2$, the homotopy limit under consideration is the double cover classified by $B \pi \circ \ev^{\ori}_{\hat{\mathsf{G}}}$. This composition fits into the following diagram:
\[
\begin{tikzpicture}[baseline= (a).base]
\node[scale=1] (a) at (0,0){
\begin{tikzcd}[column sep=4em, row sep=1.5em]
\BBun_{\hat{\mathsf{G}}}^{\ori}(M) \times M \arrow{rr}[above]{\ev^{\ori}_{\hat{\mathsf{G}}}} \arrow{dd}[left]{\pr_2} \arrow{dr}[below left]{\textnormal{can} \times \id_M} & & B \hat{\mathsf{G}} \arrow{dd}[right]{B\pi} & \\
& \BBun_{\hat{\mathsf{G}}}(M) \times M \arrow{ur}[below right]{\ev_{\hat{\mathsf{G}}}} \arrow{d}[left]{\Ind_{\pi} \times \id_M} & \\
M \arrow{r}[below]{\{\ori_M\} \times \id_M} & \BBun_{\mathbb{Z}_2}(M) \times M \arrow{r}[below]{\ev_{\mathbb{Z}_2}} & B \mathbb{Z}_2.
\end{tikzcd}
};
\end{tikzpicture}
\]
The left hand square homotopy commutes by the definition of $\BBun_{\hat{\mathsf{G}}}^{\ori}(M)$. The upper triangle commutes by definition and the right hand square commutes by inspection. It follows that the outside square homotopy commutes, proving the desired statement.
\end{proof}

Pullback along $\ev^{\ori}_{\hat{\mathsf{G}}}$ gives a map
\[
(\ev^{\ori}_{\hat{\mathsf{G}}})^*: C^{\bullet}(B \hat{\mathsf{G}}; \mathbb{C}^{\times}_{B i}) \rightarrow C^{\bullet}(\BBun_{\hat{\mathsf{G}}}^{\ori}(M) \times M; \mathbb{C}^{\times}_{\ev^{\ori *}_{\hat{\mathsf{G}}} B i}).
\]
The equivalence of Proposition \ref{prop:pullbackTwist} yields an isomorphism of twisted cochain complexes:
\[
\mathfrak{i}:
C^{\bullet}(\BBun_{\hat{\mathsf{G}}}^{\ori}(M) \times M; \mathbb{C}^{\times}_{\ev^{\ori *}_{\hat{\mathsf{G}}} B i}) \xrightarrow[]{\sim} C^{\bullet}(\BBun_{\hat{\mathsf{G}}}^{\ori}(M) \times M; \mathbb{C}^{\times}_{\ori_M}).
\]
As the orientation twist of the map $\pr_1$ is the double cover $\ori_M \rightarrow M$, pushforward along $\pr_1$ is a map
\[
\pr_{1!}: C^{\bullet}(\BBun_{\hat{\mathsf{G}}}^{\ori}(M) \times M; \mathbb{C}^{\times}_{\ori_M}) \rightarrow C^{\bullet - \dim M}(\BBun_{\hat{\mathsf{G}}}^{\ori}(M); \mathbb{C}^{\times}).
\]
These considerations lead to the following definition.

\begin{Def}
The orientation twisted transgression map along $M$, denoted by
\[
\uptau_M^{\ori}: C^{\bullet}(B \hat{\mathsf{G}}; \mathbb{C}^{\times}_{\pi}) \rightarrow C^{\bullet - \dim M}(\BBun_{\hat{\mathsf{G}}}^{\ori}(M); \mathbb{C}^{\times}),
\]
is defined to be the composition $\pr_{1!} \circ \mathfrak{i} \circ (\ev^{\ori}_{\hat{\mathsf{G}}})^*$.
\end{Def}

As the following result shows, $\uptau_{M}^{\ori}$ reduces to the standard transgression map $\uptau_M$ when $M$ is oriented.

\begin{Prop}
\label{prop:oriVsOriTwistTrans}
Suppose that $M$ is oriented. Under the identification $\BBun_{\hat{\mathsf{G}}}^{\ori}(M) \simeq \BBun_{\mathsf{G}}(M)$ of Proposition \ref{prop:orientableFrBun}, the map $\uptau_M^{\ori}$ factors through $\uptau_M$.
\end{Prop}

\begin{proof}
The equivalence $\varphi^{\omega_M}$ of Proposition \ref{prop:orientableFrBun} fits into the following homotopy commutative diagram:
\[
\begin{tikzpicture}[baseline= (a).base]
\node[scale=1] (a) at (0,0){\begin{tikzcd}[column sep=1.5em, row sep=1.25em]
{} & \BBun_{\hat{\mathsf{G}}}^{\ori}(M) \times M \arrow{dl}[above left]{\ev^{\ori}_{\hat{\mathsf{G}}}} \arrow{dr}[above right]{\pr_1} & \\
B \hat{\mathsf{G}} &{} & \BBun_{\hat{\mathsf{G}}}^{\ori}(M) \\
B \mathsf{G} \arrow{u}[left]{B \pi} & {} & \BBun_{\mathsf{G}}(M) \arrow{u}[right]{\varphi^{\omega_M}} \\
{} & \BBun_{\mathsf{G}}(M)  \times M. \arrow{ul}[below left]{\ev_{\mathsf{G}}} \arrow{ur}[below right]{\pr_1} \arrow{uuu}[right]{\varphi^{\omega_M} \times \id_M } & {}
\end{tikzcd}
};
\end{tikzpicture}
\]
The map $\pr_{1 !} \circ \ev_{\mathsf{G}}^* : C^{\bullet}(B \mathsf{G}; \mathbb{C}^{\times}) \rightarrow C^{\bullet - \dim M}(\BBun_{\mathsf{G}}(M); \mathbb{C}^{\times})$ is by definition the standard (oriented) transgression map $\uptau_M$ along $M$, as defined in \cite{willerton2008}, for example. The proposition follows.
\end{proof}

Because the boundary $\partial M$ is not assumed to be empty, the transgressed cochain $\uptau^{\ori}_M(\hat{\lambda})$ need not be closed, even if $\hat{\lambda}$ is so. A precise statement is as follows.

\begin{Prop}
\label{prop:diffTransgression}
Let $\hat{\lambda} \in Z^n(B \hat{\mathsf{G}}; \mathbb{C}_{\pi}^{\times})$ and let $M$ be a manifold with boundary $j: \partial M \hookrightarrow M$. Then the equality
\[
d \uptau_M^{\ori}(\hat{\lambda}) = -j^*\uptau_{\partial M}^{\ori}(\hat{\lambda})
\]
holds. In particular, if $M$ is closed, then $\uptau^{\ori}_M$ restricts to a map
\[
\uptau_M^{\ori}: Z^{\bullet}(B \hat{\mathsf{G}} ; \mathbb{C}^{\times}_{\pi}) \rightarrow Z^{\bullet - \dim M}(\BBun_{\hat{\mathsf{G}}}^{\ori}(M); \mathbb{C}^{\times}).
\]
\end{Prop}

\begin{proof}
Let $d = n- \dim M$ and let $\Delta_d$ be a $d$-chain on $\BBun_{\hat{\mathsf{G}}}^{\ori}(M)$. We compute
\begin{eqnarray*}
j^* \uptau^{\ori}_{\partial M} (\hat{\lambda})(\Delta_d) 
&=& \int_{\Delta_d \times \partial M} \ev_{\hat{\mathsf{G}}}^{\ori *} \hat{\lambda} \\
&=&
\int_{\ev_{\hat{\mathsf{G}}}^{\ori}(\Delta_d \times \partial M)} \hat{\lambda} \\
&=&
\int_{\ev_{\hat{\mathsf{G}}}^{\ori}( \partial (\Delta_d \times M))} \hat{\lambda}
-
\int_{\ev_{\hat{\mathsf{G}}}^{\ori}(\partial \Delta_d \times M)} \hat{\lambda} \\
&=&
\int_{\ev_{\hat{\mathsf{G}}}^{\ori} (\Delta_d \times M)} d \hat{\lambda} -\uptau^{\ori}_M(\hat{\lambda})(\partial \Delta_d)
\\
&=& -(d \uptau^{\ori}_M(\hat{\lambda}))(\Delta_d).
\end{eqnarray*}
The penultimate equality follows from the unoriented form of Stokes' Theorem, as in \cite[Theorem 7.2.15]{abraham1988}. The final equality follows from the assumption that $\hat{\lambda}$ is a twisted cocycle. This completes the proof.
\end{proof}

\subsection{Orientation twisted equivariant topological field theories}
\label{sec:oriTwEqTQFT}

For ease of notation, write $\hat{\mathsf{G}} \mhyphen \Cob^{\pi}_{\langle n, n-1, n-2 \rangle}$ in place of $\mathbf{B} \hat{\mathsf{G}} \mhyphen \Cob^{\mathbf{B} \pi}_{\langle n, n-1,n-2 \rangle}$. As follows from Proposition \ref{prop:mappingSpaceDescr}, objects of $\hat{\mathsf{G}} \mhyphen \Cob^{\pi}_{\langle n, n-1, n-2 \rangle}$ are triples $(X, P; \epsilon)$, with $X$ a closed $(n-2)$-manifold and $(P,\epsilon)$ an orientation twisted $\mathsf{G}$-bundle on $X$, and similarly for $1$- and $2$-morphisms. For this reason (see also the comments after Proposition \ref{prop:redStrDescr}), we regard $ \hat{\mathsf{G}} \mhyphen \Cob^{\pi}_{\langle n, n-1, n-2 \rangle}$ as a $\mathsf{G}$-equivariant, as opposed to $\hat{\mathsf{G}}$-equivariant, cobordism category

\begin{Def}
An orientation twisted $\mathsf{G}$-equivariant topological field theory is an orientation twisted homotopy field theory $\mathcal{Z}: \hat{\mathsf{G}} \mhyphen \Cob^{\pi}_{\langle n, n-1, n-2 \rangle} \rightarrow \mathcal{C}$.
\end{Def}

If we restrict attention to non-extended (and pointed) theories, then we recover the equivariant unoriented topological field theories of \cite{tagami2012}, \cite{kapustin2017}.

We record the following immediate corollary of Theorem \ref{thm:unoriPrim}.

\begin{Cor}
\label{cor:primEquivTheory}
The data of a finite $\mathbb{Z}_2$-graded group $\hat{\mathsf{G}}$ and a twisted $n$-cocycle $\hat{\lambda} \in Z^n(B \hat{\mathsf{G}} ; \mathsf{U}(1)_{\pi})$ defines an unoriented lift $\mathcal{P}_{\hat{\mathsf{G}}}^{\hat{\lambda}}$ of $\mathcal{P}_{\mathsf{G}}^{\lambda}$.
\end{Cor}

\begin{Rem}
There is a $\hat{\mathsf{G}}$-equivariant isomorphism $\mathbb{C}^{\times}_{\pi} \simeq \mathsf{U}(1)_{\pi} \times \mathbb{R}$ which we use to identify $H^{\bullet}(B \hat{\mathsf{G}} ; \mathbb{C}^{\times}_{\pi})$ with $H^{\bullet}(B \hat{\mathsf{G}} ; \mathsf{U}(1)_{\pi})$. We will therefore restrict attention to unitary cocycles in what follows.
\end{Rem}

\subsection{Orientation twisted orbifolding}
\label{sec:orbifolding}

The orbifold construction is a well-known physical procedure for passing from an equivariant to a non-equivariant field theory. In the setting of oriented topological field theory, the orbifold construction has been studied by many authors; see \cite{kaufmann2003} and \cite{kirillov2004} for early algebraic discussions in dimensions two and three, respectively. In this section we adapt the functorial formulation of Schweigert--Woike \cite{schweigert2019}, \cite{schweigert2018} to the orientation twisted setting.

Let $\hat{\phi}: \hat{\mathsf{G}} \rightarrow \hat{\mathsf{H}}$ be a morphism of finite $\mathbb{Z}_2$-graded groups. Let
\[
\mathcal{Z} : \hat{\mathsf{G}} \mhyphen \Cob^{\pi_{\hat{\mathsf{G}}}}_{\langle n, n-1, n-2 \rangle} \rightarrow 2\Vect_{\mathbb{C}}
\]
be an orientation twisted $\mathsf{G}$-equivariant topological field theory. The primary goal of this section is to modify the oriented pushforward construction of \cite[\S 3.3]{schweigert2019}, \cite[\S 3.1]{schweigert2018} so as to obtain from $\mathcal{Z}$ an orientation twisted $\hat{\mathsf{H}}$-equivariant topological field theory
\[
\mathcal{Z}^{\hat{\phi}} : \hat{\mathsf{H}} \mhyphen \Cob^{\pi_{\hat{\mathsf{H}}}}_{\langle n, n-1, n-2\rangle} \rightarrow 2\Vect_{\mathbb{C}}(\Grpd).
\]
To begin, suppose that we are given an object $(X,f; h) \in \hat{\mathsf{H}} \mhyphen \Cob^{\pi_{\hat{\mathsf{H}}}}_{\langle n, n-1, n-2\rangle}$. Define $ \mathcal{Z}^{\hat{\phi}}(X,f; h) \in 2\Vect_{\mathbb{C}}(\Grpd)$ to be the pseudofunctor
\[
\mathcal{R}_{\mathcal{Z}}^{X, \hat{\phi}}: R \Ind^{-1}_{\hat{\phi}}(X,f; h) \rightarrow \BBun_{\hat{\mathsf{G}}}^{\ori}(X) \xrightarrow[]{\mathcal{R}^X_{\mathcal{Z}}} 2\Vect_{\mathbb{C}}.
\]
The first functor is canonical and $\mathcal{R}_{\mathcal{Z}}^X$ is the pseudofunctor defined in Proposition \ref{prop:vbMappingSpace}. More precisely, we have used Lemma \ref{lem:asphericalMappingGrpd} to identify\footnote{More precisely, we have implicitly chosen a quasi-inverse
\[
\overline{(-)}: \MMap_{\mathbf{B} \mathbb{Z}_2}(X, \mathbf{B} \hat{\mathsf{G}}) \rightarrow \MMap^{\leq 2}_{\mathbf{B} \mathbb{Z}_2}(X, \mathbf{B} \hat{\mathsf{G}})
\]
of the canonical map $\MMap^{\leq 2}_{\mathbf{B} \mathbb{Z}_2}(X, \mathbf{B} \hat{\mathsf{G}}) \rightarrow
\MMap_{\mathbf{B} \mathbb{Z}_2}(X, \mathbf{B} \hat{\mathsf{G}})$ by lifting each object $(f;h) \in \MMap_{\mathbf{B} \mathbb{Z}_2}(X, \mathbf{B} \hat{\mathsf{G}})$ to $(f; \overline{h}) \in \MMap^{\leq 2}_{\mathbf{B} \mathbb{Z}_2}(X, \mathbf{B} \hat{\mathsf{G}})$, where $\overline{h}$ is a representative of the equivalence class $h$, and similarly for morphisms. We will identify $\overline{h}$ with $h$ in what follows.} the domain $\MMap^{\leq 2}_{\mathbf{B} \mathbb{Z}_2}(X, \mathbf{B}\hat{\mathsf{G}})$ of $\mathcal{R}_{\mathcal{Z}}^X$ with the groupoid $\BBun_{\hat{\mathsf{G}}}^{\ori}(X)$. To a $1$-morphism
\[
((Y,o_{\bullet}),F; H): (X_1,f_1; h_1) \rightarrow (X_2,f_2; h_2)
\]
in $\hat{\mathsf{H}} \mhyphen \Cob^{\pi_{\hat{\mathsf{H}}}}_{\langle n, n-1, n-2\rangle}$ the pseudofunctor $\mathcal{Z}^{\hat{\phi}}$ assigns a $1$-morphism whose underlying span of groupoids is
\[
\begin{tikzpicture}[baseline= (a).base]
\node[scale=1] (a) at (0,0){
\begin{tikzcd}[column sep=1.5em, row sep=0.5em]
{} & R \Ind^{-1}_{\hat{\phi}}(Y,F; H) \arrow{dl}[above]{s} \arrow{dr}[above]{t} & \\
R \Ind^{-1}_{\hat{\phi}}(X_1,f_1; h_1) && R \Ind^{-1}_{\hat{\phi}}(X_2,f_2; h_2).
\end{tikzcd}
};
\end{tikzpicture}
\]
To define $s$ and $t$, observe that there is a homotopy commutative diagram
\[
\begin{tikzpicture}[baseline= (a).base]
\node[scale=1] (a) at (0,0){
\begin{tikzcd}[column sep=5.0em, row sep=0.75em]
R\Ind_{\hat{\phi}}^{-1}(Y,F;H) \arrow{r} \arrow{d} & \BBun_{\hat{\mathsf{G}}}^{\ori}(Y) \arrow{r}[above]{(i_k,o_k)^*} & \BBun_{\hat{\mathsf{G}}}^{\ori}(X_k) \arrow{dd}[right]{\Ind_{\hat{\phi}}}\\
\{ (Y,F;H)\} \arrow{d} & R \Ind_{\hat{\phi}}^{-1}(X_k,f_k; h_k) \arrow{dl} \arrow{ur} & {} \\
\{ (X_k, f_k; h_k) \} \arrow{rr} &{}& \BBun_{\hat{\mathsf{H}}}^{\ori}(X_k).
\end{tikzcd}
};
\end{tikzpicture}
\]
The $2$-universal property of $R \Ind_{\hat{\phi}}^{-1}(X_k,f_k; h_k)$ then determines $s$ and $t$. The required pseudonatural transformation
\[
\mathcal{R}_{\mathcal{Z}}^{Y, \hat{\phi}}: s^*\mathcal{R}_{\mathcal{Z}}^{X_1, \hat{\phi}} \rightarrow t^*\mathcal{R}_{\mathcal{Z}}^{X_2, \hat{\phi}}
\]
is defined as follows. Let $((\widetilde{F}; \widetilde{H});\tilde{\epsilon}) \in R \Ind^{-1}_{\hat{\phi}}(Y,F; H)$, that is, $(\widetilde{F}; \widetilde{H}) \in \BBun_{\hat{\mathsf{G}}}^{\ori}(Y)$ and $\tilde{\epsilon} : R\Ind_{\hat{\phi}}(\widetilde{F}; \widetilde{H}) \rightarrow (F; H)$. Put
\[
\mathcal{R}_{\mathcal{Z}}^{Y, \hat{\phi}}((\widetilde{F}; \widetilde{H}); \tilde{\epsilon})
=
\mathcal{Z}(Y, \widetilde{F}; \widetilde{H}),
\]
where we have implicitly used $\tilde{\epsilon}$ to identify $\widetilde{F}_{\vert X_k}$ and $\widetilde{H}_{\vert X_k}$ with $f_k$ and $h_k$, respectively. A morphism $((\widetilde{F}_1; \widetilde{H}_1); \tilde{\epsilon}_1) \rightarrow ((\widetilde{F}_2; \widetilde{H}_2); \tilde{\epsilon}_2)$ in $R \Ind^{-1}_{\hat{\phi}}(Y,F; H)$ is an equivalence class of homotopies $\eta : Y \times I \rightarrow \mathbf{B}\hat{\mathsf{G}}$ from $\widetilde{F}_1$ to $\widetilde{F}_2$ which respects the orientation framings. Completely analogously to \cite[Theorem 3.1]{schweigert2018}, a representative of the homotopy $\eta$ can be used to construct the coherence $2$-isomorphisms of $\mathcal{R}_{\mathcal{Z}}^{Y, \hat{\phi}}$.

It remains to define the value of $\mathcal{Z}^{\hat{\phi}}$ on a $2$-morphism
\[
\begin{tikzcd}[column sep=15em]
(X_1, f_1; h_1)
  \arrow[bend left=12]{r}[name=U,below]{}{((Y_1; o_{1,\bullet}),F_1;H_1)} 
  \arrow[bend right=12]{r}[name=D]{}[swap]{((Y_2; o_{2,\bullet}),F_2;H_2)}
& 
(X_2, f_2; h_2).
\arrow[shorten <=3pt,shorten >=3pt,Rightarrow,to path={(U) -- node[label=left:{\scriptsize $((Z;\sigma_{\bullet}), \varphi; \eta)$}] {} (D)}]{}
\end{tikzcd}
\]
The underlying span of spans of groupoids is
\[
\begin{tikzpicture}[baseline= (a).base]
\node[scale=1] (a) at (0,0){
\begin{tikzcd}[column sep=3.0em, row sep=1.0em]
& R\Ind_{\hat{\phi}}^{-1}(Y_1, F_1; H_1) \arrow{rd}[above]{t_1} \arrow{ld}[above]{s_1} & \\
R\Ind_{\hat{\phi}}^{-1}(X_1, f_1; h_1) & R\Ind_{\hat{\phi}}^{-1}(Z, \varphi; \eta) \arrow{u}[left]{\sigma} \arrow{d}[left]{\tau} & R\Ind_{\hat{\phi}}^{-1}(X_2, f_2; h_2) \\
& R\Ind_{\hat{\phi}}^{-1}(Y_2, F_2; H_2), \arrow{ru}[below]{t_2} \arrow{lu}[below]{s_2} &
\end{tikzcd}
};
\end{tikzpicture}
\]
the functors $\sigma$ and $\tau$ being constructed in the same way as $s$ and $t$ above. The component of the map of intertwiners
\[
\hat{\mathcal{Z}}^{\hat{\phi}}((Z;\sigma_{\bullet}), \varphi; \eta):\sigma^*\mathcal{R}_{\mathcal{Z}}^{Y_1, \hat{\phi}} \Rightarrow \tau^*\mathcal{R}_{\mathcal{Z}}^{Y_2, \hat{\phi}}
\]
at $((Z;\sigma_{\bullet}), \varphi;\eta) \in R\Ind_{\hat{\phi}}^{-1}(Z, \varphi; \eta)$ is defined to be $\mathcal{Z}((Z;\sigma_{\bullet}), \varphi;\eta)$.

\begin{Thm}
\label{thm:unoriOrbifold}
The above construction defines an orientation twisted $\mathsf{H}$-equivariant topological field theory $\mathcal{Z}^{\hat{\phi}} : \hat{\mathsf{H}} \mhyphen \Cob^{\pi_{\hat{\mathsf{H}}}}_{\langle n, n-1, n-2\rangle} \rightarrow 2\Vect_{\mathbb{C}}(\Grpd)$.
\end{Thm}

\begin{proof}
The proof that $\mathcal{Z}^{\hat{\phi}}$ is symmetric monoidal pseudofunctor is a direct modification of the oriented case \cite[\S 3]{schweigert2018}. The key point there is the gluing property of the stack $\BBun_{\mathsf{G}}(-)$, which ensures that $\BBun_{\mathsf{G}}(-)$ is compatible with the various compositions of cobordisms. The corresponding property of $\BBun^{\ori}_{\hat{\mathsf{G}}}(-)$ follows from its construction as a homotopy fibre of $\pi \circ (-) : \BBun_{\hat{\mathsf{G}}}(-) \rightarrow \BBun_{\mathbb{Z}_2}(-)$.

To verify the homotopy invariance of $\mathcal{Z}^{\hat{\phi}}$, suppose that $\kappa: Z \times I \rightarrow \mathbf{B} \hat{\mathsf{G}}$ is a homotopy relative $\partial Z$ from $\varphi$ to $\varphi^{\prime}$ which satisfies $\eta^{\prime} * (\Pi \circ \kappa) \simeq \eta$. The map $\kappa$ induces a functor
\[
R\Ind_{\hat{\phi}}^{-1}(Z, \varphi; \eta) \rightarrow R\Ind_{\hat{\phi}}^{-1}(Z, \varphi^{\prime}; \eta^{\prime})
\]
under which $\sigma^{\prime}$ and $\tau^{\prime}$ and pullback to $\sigma$ and $\tau$, respectively. The spans of spans of groupoids which underlie $R\Ind_{\hat{\phi}}^{-1}(Z, \varphi; \eta)$ and $R\Ind_{\hat{\phi}}^{-1}(Z, \varphi^{\prime}; \eta^{\prime})$ are therefore equivalent. Under this identification, the homotopy invariance of $\mathcal{Z}$ then ensures that the resulting maps of intertwiners are equivalent.
\end{proof}

The functor $\mathcal{Z}^{\hat{\phi}}$ is called the $\hat{\phi}$-pushforward of $\mathcal{Z}$. Motivated by \cite{freed2010}, \cite{morton2015}, we interpret $\mathcal{Z}^{\hat{\phi}}$ as a classical topological $\mathsf{H}$-gauge theory. The quantization of $\mathcal{Z}^{\hat{\phi}}$ is then given by the composition
\[
\mathcal{Z}^{\hat{\phi},\textnormal{orb}}: \hat{\mathsf{H}}\mhyphen \Cob^{\pi_{\hat{\mathsf{H}}}}_{\langle n, n-1, n-2\rangle} \xrightarrow[]{\mathcal{Z}^{\hat{\phi}}} 2\Vect_{\mathbb{C}}(\Grpd) \xrightarrow[]{\Par} 2\Vect_{\mathbb{C}},
\]
with $\Par$ as in Section \ref{sec:twistLin}. The composition $\mathcal{Z}^{\hat{\phi},\textnormal{orb}}$ is called the $\hat{\phi}$-orbifold of $\mathcal{Z}$.

In particular, for any $\mathbb{Z}_2$-graded group $\hat{\mathsf{G}}$ we can apply the orientation twisted pushforward construction to the map $\hat{\phi} = \pi_{\hat{\mathsf{G}}}$, regarded as a morphism from $\hat{\mathsf{G}}$ to the terminal object $(\mathbb{Z}_2 \xrightarrow[]{\id} \mathbb{Z}_2) \in \Grp_{\slash \mathbb{Z}_2}$. This allows us to formulate the following definition, which makes implicit use of Proposition \ref{prop:reductionToUnoriented}.

\begin{Def}
The $\mathsf{G}$-orbifold of $\mathcal{Z} \in \hat{\mathsf{G}} \mhyphen \TFT^{\pi}_{\langle n, n-1, n-2\rangle}$ is the unoriented topological field theory $\mathcal{Z}^{\pi,\textnormal{orb}} \in \TFT_{\langle n, n-1, n-2 \rangle}(2\Vect_{\mathbb{C}})$.
\end{Def}

Note that we orbifold by the group $\mathsf{G}$, as opposed to $\hat{\mathsf{G}}$. This is a key difference between the orientation twisted perspective taken in this paper and the equivariant perspective of \cite{maier2012}.

Oriented and orientation twisted orbifolding are compatible in the following sense.

\begin{Prop}
\label{prop:orbCompatibility}
For each morphism $\hat{\phi}: \hat{\mathsf{G}} \rightarrow \hat{\mathsf{H}}$ of $\mathbb{Z}_2$-graded groups, the assignment $\mathcal{Z} \mapsto \mathcal{Z}^{\hat{\phi}}$ extends to a pseudofunctor
\[
(-)^{\hat{\phi}}: 
\hat{\mathsf{G}} \mhyphen \TFT_{\langle n,n-1,n-2\rangle}^{\pi_{\hat{\mathsf{G}}}} \rightarrow \hat{\mathsf{H}} \mhyphen \TFT^{\pi_{\hat{\mathsf{H}}}}_{\langle n, n-1,n-2\rangle} (2\Vect_{\mathbb{C}}(\Grpd)).
\]
Moreover, the diagram
\[
\begin{tikzpicture}[baseline= (a).base]
\node[scale=1] (a) at (0,0){
\begin{tikzcd}[column sep=5.0em, row sep=1.5em]
\hat{\mathsf{G}} \mhyphen \TFT^{\pi_{\hat{\mathsf{G}}}}_{\langle n, n-1,n-2\rangle} \arrow{d}[left]{(-)^{\hat{\phi}}} \arrow{r}[above]{\mathcal{F}} & \mathsf{G} \mhyphen \TFT^{\ori}_{\langle n, n-1,n-2\rangle} \arrow{d}[right]{(-)^{\phi}} 
\\
\hat{\mathsf{H}} \mhyphen \TFT^{\pi_{\hat{\mathsf{H}}}}_{\langle n, n-1,n-2\rangle} (2\Vect_{\mathbb{C}}(\Grpd)) \arrow{d}[left]{\Par \circ(-)} \arrow{r}[above]{\mathcal{F}}& \mathsf{H} \mhyphen \TFT^{\ori}_{\langle n, n-1,n-2\rangle} (2\Vect_{\mathbb{C}}(\Grpd))  \arrow{d}[right]{\Par \circ( - )} 
\\
\hat{\mathsf{H}} \mhyphen \TFT^{\pi_{\hat{\mathsf{H}}}}_{\langle n, n-1,n-2\rangle} \arrow{r}[above]{\mathcal{F}} & \mathsf{H} \mhyphen \TFT^{\ori}_{\langle n, n-1,n-2\rangle}
\end{tikzcd}
};
\end{tikzpicture}
\]
consists of homotopy commutative squares.
\end{Prop}

\begin{proof}
We omit the construction of the pseudofunctor $(-)^{\hat{\phi}}$, which is a straightforward, but tedious, exercise. The second statement follows from a direct comparison between our construction of the orientation twisted pushforward map and that of the oriented pushforward from \cite{schweigert2018b}.
\end{proof}

\subsection{Non-extended theories in two dimensions}
\label{sec:twoDimOrb}

We study two dimensional non-extended orientation twisted equivariant topological field theories and their orbifolds. This leads to a different perspective on the algebraic classification of such theories given by Kapustin--Turzillo \cite{kapustin2017}.

We begin by recalling the classification of two dimensional (un)oriented topological field theories. Let $(\mathcal{A}, \otimes, \mathbf{1})$ be a symmetric monoidal category.

\begin{Def}[{\cite{turaev2006}, \cite{alexeevski2006}}]
An unoriented Frobenius algebra in $\mathcal{A}$ is the data of
\begin{enumerate}[label=(\roman*)]
\item a commutative Frobenius algebra object $A$ in $\mathcal{A}$, with multiplication $m: A \otimes A \rightarrow A$ and comultiplication $\Delta: A \rightarrow A \otimes A$,
\item a Frobenius algebra morphism $p:A^{\coop} \rightarrow A$, where $A^{\coop}$ is $A$ with the opposite multiplication and coopposite comultiplication, which satisfies $p \circ p = \id_A$, and
\item a morphism $Q: \mathbf{1} \rightarrow A$, called the charge of the crosscap,
\end{enumerate}
for which the diagrams
\begin{equation}
\label{eq:linConstraint}
\begin{tikzpicture}[baseline= (a).base]
\node[scale=1] (a) at (0,0){
\begin{tikzcd}[column sep={5.5em,between origins},row sep={1.76em,between origins}]
{} & {} & A \otimes A \arrow{drr}[above right]{m} & {} & {} \\
\mathbf{1} \otimes A \arrow{urr}[above left]{Q \otimes \id_A} \arrow{dr}[below left]{Q \otimes \id_A} & {} & {} & {} & A \\
{} & A \otimes A \arrow{rr}[below]{m} & {} & A \arrow{ur}[below right]{p} & {}
\end{tikzcd}
};
\end{tikzpicture}
\end{equation}
and
\[
\begin{tikzpicture}[baseline= (a).base]
\node[scale=1] (a) at (0,0){
\begin{tikzcd}[column sep={6.0em,between origins},row sep={1.75em,between origins}]
{} & \mathbf{1} \otimes \mathbf{1} \arrow{rr}[above]{Q \otimes Q} & {} & A \otimes A \arrow{rd}[above right]{m} \\
\mathbf{1} \arrow{ru}[above]{\sim} \arrow{rd}[below left]{\textnormal{unit}} & {} & {} & {} & A \\
{} & A \arrow{r}[below]{\Delta} & A \otimes A \arrow{r}[below]{p \otimes \id_A} & A \otimes A \arrow{ru}[below right]{m}
\end{tikzcd}
};
\end{tikzpicture}
\]
commute.
\end{Def}

A morphism of unoriented Frobenius algebras is a morphism\footnote{In the present context, a morphism of Frobenius algebras is a map which is both an algebra and a coalgebra homomorphism.} of the underlying commutative Frobenbius algebras which intertwines the involutions and respects the charges of the crosscap. Unoriented Frobenius algebras assemble to a category, which is in fact a groupoid.

\begin{Thm}
\phantomsection
\label{thm:2dClassification}
\begin{enumerate}[label=(\roman*)]
\item The groupoid $\TFT^{\ori}_{\langle 2,1 \rangle}(\mathcal{A})$ is equivalent to the groupoid of Frobenius algebras in $\mathcal{A}$.

\item The groupoid $\TFT_{\langle 2,1 \rangle}(\mathcal{A})$ is equivalent to the groupoid of unoriented Frobenius algebras in $\mathcal{A}$.
\end{enumerate}
\end{Thm}

\begin{proof}
The first statement is proved in \cite[Theorem 3.6.19]{kock2004}, for example. When $\mathcal{A} = \Vect_{\mathbb{C}}$ the second statement appears as \cite[Proposition 2.9]{turaev2006}, \cite[Theorem 4.4]{alexeevski2006}. These proofs can be interpreted as giving a generators and relations presentation of the symmetric monoidal category $\Cob_{\langle 2, 1\rangle}$ and so also apply to a general target category $\mathcal{A}$.
\end{proof}

We briefly indicate the construction of an unoriented Frobenius algebra from an unoriented topological field theory. The morphisms of the category $\Cob_{\langle 2,1\rangle}$ are generated by the image of the forgetful functor $\mathcal{F}: \Cob_{\langle 2, 1 \rangle}^{\ori} \rightarrow \Cob_{\langle 2, 1\rangle}$ together with the incompatibly oriented cylinder $S^1 \rightarrow S^1$ and the crosscap $\mathbb{M}: \varnothing^1 \rightarrow S^1$, which we draw as
\[
\begin{tikzpicture}[very thick,scale=2.01,color=black,baseline]
\coordinate (r1) at (0.1,0.15);
\coordinate (r2) at (0.1,-0.15);
\coordinate (r3) at (0.8,0.15);
\coordinate (r4) at (0.8,-0.15);
\draw (r1) to (r3);
\draw (r2) to (r4);
\draw[very thick, blue!80!black,decoration={markings, mark=at position 0.5 with {\arrow{<}}},        postaction={decorate}] (r1) .. controls +(0.15,0) and +(0.15,0) ..  (r2); 
\draw[very thick, blue!80!black] (r1) .. controls +(-0.15,0) and +(-0.15,0) ..  (r2); 
\draw[very thick, blue!80!black,decoration={markings, mark=at position 0.5 with {\arrow{>}}},        postaction={decorate}] (r3) .. controls +(0.15,0) and +(0.15,0) ..  (r4); 
\draw[very thick, blue!80!black,opacity=0.2] (r3) .. controls +(-0.15,0) and +(-0.15,0) ..  (r4); 
\end{tikzpicture}
\qquad \textnormal{and}\qquad
\begin{tikzpicture}[very thick,scale=2.01,color=black, baseline]
\coordinate (p1) at (0,-0.175);
\coordinate (p2) at (0,0.175);
\coordinate (x1) at (-0.24,-0.045);
\coordinate (x2) at (-0.17,0.045);
\coordinate (y1) at (-0.24,0.045);
\coordinate (y2) at (-0.17,-0.045);
\draw[very thick] (p1) .. controls +(-0.4,0) and +(-0.4,0) ..  (p2); 
\draw[very thick, blue!80!black] (p1) .. controls +(0.15,0) and +(0.15,0) ..  (p2); 
\draw[very thick, blue!80!black, opacity=0.2] (p1) .. controls +(-0.15,0) and +(-0.15,0) ..  (p2); 
\draw[thick] (x1) to (x2);
\draw[thick] (y1) to (y2);
\draw[thick] (-0.205,0) circle (0.05cm);
\end{tikzpicture}
,
\]
respectively. Here the arrows indicate the embeddings of the boundary circles. Under the equivalence of Theorem \ref{thm:2dClassification}, these cobordisms determine the involution $p$ and the morphism $Q$. This explains the naming of $Q$. The constraint \eqref{eq:linConstraint}, for example, is imposed by the following equality of composed cobordisms in $\Cob_{\langle 2, 1 \rangle}$:
\begin{equation}
\label{eq:quadCrossConstr}
\Sigma_1
=
\begin{tikzpicture}[very thick,scale=2.01,baseline=-0.35cm,color=black]
\coordinate (q1) at (-0.4,0.175);
\coordinate (q2) at (-0.4,0.525);
\coordinate (x1) at (-0.63,0.375);
\coordinate (x2) at (-0.58,0.325);
\coordinate (y1) at (-0.63,0.325);
\coordinate (y2) at (-0.58,0.375);
\draw[very thick] (q1) .. controls +(-0.4,0) and +(-0.4,0) ..  (q2); 
\draw[very thick, blue!80!black,decoration={markings, mark=at position 0.5 with {\arrow{>}}},        postaction={decorate}] (q1) .. controls +(0.15,0) and +(0.15,0) ..  (q2); 
\draw[very thick, blue!80!black, opacity=0.2] (q1) .. controls +(-0.15,0) and +(-0.15,0) ..  (q2); 
\draw[thick] (x1) to (x2);
\draw[thick] (y1) to (y2);
\draw[thick] (-0.605,0.35) circle (0.05cm);
\coordinate (p1) at (0,-0.575);
\coordinate (p2) at (0,-0.225);
\coordinate (p3) at (0,0.175);
\coordinate (p4) at (0,0.525);
\coordinate (p5) at (0.8,0.15);
\coordinate (p6) at (0.8,-0.15);
\draw (p2) .. controls +(0.35,0) and +(0.35,0) ..  (p3); 
\draw (p4) .. controls +(0.5,0) and +(-0.5,0) ..  (p5); 
\draw (p6) .. controls +(-0.5,0) and +(0.5,0) ..  (p1); 
\draw[very thick, blue!80!black,decoration={markings, mark=at position 0.5 with {\arrow{>}}},        postaction={decorate}] (p1) .. controls +(0.15,0) and +(0.15,0) ..  (p2); 
\draw[very thick, blue!80!black] (p1) .. controls +(-0.15,0) and +(-0.15,0) ..  (p2); 
\draw[very thick, blue!80!black,decoration={markings, mark=at position 0.5 with {\arrow{>}}},        postaction={decorate}] (p3) .. controls +(0.15,0) and +(0.15,0) ..  (p4); 
\draw[very thick, blue!80!black] (p3) .. controls +(-0.15,0) and +(-0.15,0) ..  (p4); 
\draw[very thick, blue!80!black,decoration={markings, mark=at position 0.5 with {\arrow{<}}},        postaction={decorate}] (p5) .. controls +(0.15,0) and +(0.15,0) ..  (p6); 
\draw[very thick, blue!80!black,opacity=0.2] (p5) .. controls +(-0.15,0) and +(-0.15,0) ..  (p6); 
\node at (0.05,-0.75)  {\scriptsize $\ell^{(1)}_1$};
\node at (0.05,0.70)  {\scriptsize $\ell^{(1)}_2$};
\node at (0.90,-0.3)  {\scriptsize $\ell^{(1)}_3$};
\end{tikzpicture}
\;\;
=
\;\;
\begin{tikzpicture}[very thick,scale=2.01,color=black,baseline]
\coordinate (q1) at (-0.4,0.175);
\coordinate (q2) at (-0.4,0.525);
\coordinate (x1) at (-0.63,0.375);
\coordinate (x2) at (-0.58,0.325);
\coordinate (y1) at (-0.63,0.325);
\coordinate (y2) at (-0.58,0.375);
\draw[very thick] (q1) .. controls +(-0.4,0) and +(-0.4,0) ..  (q2); 
\draw[very thick, blue!80!black,decoration={markings, mark=at position 0.5 with {\arrow{>}}},        postaction={decorate}] (q1) .. controls +(0.15,0) and +(0.15,0) ..  (q2); 
\draw[very thick, blue!80!black, opacity=0.2] (q1) .. controls +(-0.15,0) and +(-0.15,0) ..  (q2); 
\draw[thick] (x1) to (x2);
\draw[thick] (y1) to (y2);
\draw[thick] (-0.605,0.35) circle (0.05cm);
\coordinate (p1) at (0,-0.575);
\coordinate (p2) at (0,-0.225);
\coordinate (p3) at (0,0.175);
\coordinate (p4) at (0,0.525);
\coordinate (p5) at (0.8,0.15);
\coordinate (p6) at (0.8,-0.15);
\draw (p2) .. controls +(0.35,0) and +(0.35,0) ..  (p3); 
\draw (p4) .. controls +(0.5,0) and +(-0.5,0) ..  (p5); 
\draw (p6) .. controls +(-0.5,0) and +(0.5,0) ..  (p1); 
\draw[very thick, blue!80!black,decoration={markings, mark=at position 0.5 with {\arrow{>}}},        postaction={decorate}] (p1) .. controls +(0.15,0) and +(0.15,0) ..  (p2); 
\draw[very thick, blue!80!black] (p1) .. controls +(-0.15,0) and +(-0.15,0) ..  (p2); 
\draw[very thick, blue!80!black,decoration={markings, mark=at position 0.5 with {\arrow{>}}},        postaction={decorate}] (p3) .. controls +(0.15,0) and +(0.15,0) ..  (p4); 
\draw[very thick, blue!80!black] (p3) .. controls +(-0.15,0) and +(-0.15,0) ..  (p4); 
\draw[very thick, blue!80!black,decoration={markings, mark=at position 0.5 with {\arrow{<}}},        postaction={decorate}] (p5) .. controls +(0.15,0) and +(0.15,0) ..  (p6); 
\draw[very thick, blue!80!black,opacity=0.2] (p5) .. controls +(-0.15,0) and +(-0.15,0) ..  (p6); 
\node at (0.05,-0.75)  {\scriptsize $\ell^{(2)}_1$};
\node at (0.05,0.70)  {\scriptsize $\ell^{(2)}_2$};
\node at (1.0,-0.4)  {\scriptsize $\ell^{(2)}_3$};
\node at (2.00,-0.30)  {\scriptsize $\ell^{(2)}_4$};
\coordinate (r1) at (1.2,0.15);
\coordinate (r2) at (1.2,-0.15);
\coordinate (r3) at (1.9,0.15);
\coordinate (r4) at (1.9,-0.15);
\draw (r1) to (r3);
\draw (r2) to (r4);
\draw[very thick, blue!80!black,decoration={markings, mark=at position 0.5 with {\arrow{<}}},        postaction={decorate}] (r1) .. controls +(0.15,0) and +(0.15,0) ..  (r2); 
\draw[very thick, blue!80!black] (r1) .. controls +(-0.15,0) and +(-0.15,0) ..  (r2); 
\draw[very thick, blue!80!black,decoration={markings, mark=at position 0.5 with {\arrow{>}}},        postaction={decorate}] (r3) .. controls +(0.15,0) and +(0.15,0) ..  (r4); 
\draw[very thick, blue!80!black,opacity=0.2] (r3) .. controls +(-0.15,0) and +(-0.15,0) ..  (r4); 
\end{tikzpicture}
= \Sigma_2.
\end{equation}

Consider now the $\mathsf{G}$-equivariant setting. We begin with the oriented case. Let $\mathcal{Z} \in \mathsf{G} \mhyphen \TFT^{\ori}_{\langle 2,1\rangle}$. According to Theorem \ref{thm:2dClassification} (see also \cite[Example 3.11]{schweigert2019}), the pushforward $\mathcal{Z}^{\mathsf{G} \rightarrow \pt} \in \TFT^{\ori}_{\langle 2,1\rangle} (\Vect_{\mathbb{C}}(\Grpd))$ is determined by a commutative Frobenius algebra structure on
\begin{equation}
\label{eq:GTuraevFunctor}
\mathcal{Z}^{\mathsf{G} \rightarrow \pt}(S^1) = \left(
\BBun_{\mathsf{G}}(S^1) \simeq \mathsf{G} \git \mathsf{G} \xrightarrow[]{\mathcal{R}^{S^1}_{\mathcal{Z}}} \Vect_{\mathbb{C}} \right),
\end{equation}
considered as an object of $\Vect_{\mathbb{C}}(\Grpd)$. Concretely, this is the data of a $\mathsf{G}$-graded algebra
\[
A= \bigoplus_{g \in \mathsf{G}} A_g
\]
together with a group homomorphism $\mathfrak{a} : \mathsf{G} \rightarrow \Aut_{\textnormal{alg}}(A)$ and a trace $\langle - \rangle_e : A_e \rightarrow \mathbb{C}$ which satisfy the following conditions:
\begin{enumerate}[label=(\roman*)]
\item For all $g,h \in \mathsf{G}$, the map $\mathfrak{a}_g: A \rightarrow A$ restricts to a map $A_h \rightarrow A_{g hg^{-1}}$.

\item The trace $\langle - \rangle_e$ is $\mathsf{G}$-invariant and induces a nondegenerate bilinear form $\langle -, - \rangle_g : A_g \otimes_{\mathbb{C}} A_{g^{-1}} \rightarrow \mathbb{C}$ for each $g \in \mathsf{G}$.

\item For each $x \in A_h$ and $y \in A_g$, the equality $\mathfrak{a}_g(xy) = y x$ holds.
\end{enumerate}

The theory $\mathcal{Z}^{\mathsf{G} \rightarrow \pt}$ has additional structure, however, arising from its construction as a pushforward. Namely, it also satisfies the following conditions (see \cite[\S II.3.2]{turaev2010}):
\begin{enumerate}[label=(\roman*)]
\setcounter{enumi}{3}
\item For each $g \in \mathsf{G}$, the restriction of $\mathfrak{a}_g$ to $A_g$ is the identity.

\item For all $g, h \in \mathsf{G}$, the equality
\[
\sum_i \mathfrak{a}_h(x_i)x^i = \sum_j y_j \mathfrak{a}_g(y^j)
\]
holds in $A_{hgh^{-1}g^{-1}}$, where
\[
\Delta_g(1) = \sum_i x_i \otimes x^i \in A_g \otimes_{\mathbb{C}} A_{g^{-1}},
\qquad
\Delta_h(1) = \sum_j y_j \otimes y^j \in A_h \otimes_{\mathbb{C}} A_{h^{-1}}
\]
are the twisted sector comultiplications.
\end{enumerate}

\begin{Def}[{\cite[\S II.3.2]{turaev2010}}]
A triple $(A, \mathfrak{a}, \langle - \rangle_e)$ which satisfies conditions \textnormal{(i)-(v)} is called a $\mathsf{G}$-Turaev algebra.
\end{Def}

The image of the oriented pushforward $(-)^{\mathsf{G} \rightarrow \pt}$ can be characterized as follows.

\begin{Prop}[{\cite[Theorem 3.1]{turaev2010}; see also \cite[\S 3.3]{schweigert2019}}]
\label{prop:oriEquivClass}
The oriented pushforward
\[
(-)^{\mathsf{G} \rightarrow \pt}: \mathsf{G} \mhyphen \TFT^{\ori}_{\langle 2, 1\rangle} \rightarrow \TFT^{\ori}_{\langle 2, 1\rangle} (\Vect_{\mathbb{C}}(\Grpd))
\]
is an equivalence onto the non-full subcategory of $\TFT^{\ori}_{\langle 2, 1\rangle} (\Vect_{\mathbb{C}}(\Grpd))$ spanned by $\mathsf{G}$-Turaev algebras.
\end{Prop}

\begin{Rem}
Interpreted as objects of $\TFT^{\ori}_{\langle 2, 1\rangle} (\Vect_{\mathbb{C}}(\Grpd))$, a morphism of $\mathsf{G}$-Turaev algebras, as defined in \cite[\S II.3.2]{turaev2010}, is a morphism in $\TFT^{\ori}_{\langle 2, 1\rangle} (\Vect_{\mathbb{C}}(\Grpd))$ which is the identity on the underlying spans.
\end{Rem}

Consider now the unoriented case. Let $\mathcal{Z} \in \hat{\mathsf{G}} \mhyphen \TFT^{\pi}_{\langle 2,1\rangle}$. By Theorem \ref{thm:2dClassification}, the pushforward $\mathcal{Z}^{\pi} \in \TFT_{\langle 2,1\rangle} (\Vect_{\mathbb{C}}(\Grpd))$ is determined by an unoriented Frobenius algebra structure on the functor \eqref{eq:GTuraevFunctor}. In addition to the underlying commutative Frobenius algebra, this unpacks to the data of an extension of $\mathfrak{a}$ to a $\mathbb{Z}_2$-graded group homomorphism $\hat{\mathfrak{a}} : \hat{\mathsf{G}} \rightarrow \Aut^{\text{\textnormal{gen}}}_{\mathsf{alg}}(A)$, where $\Aut^{\text{\textnormal{gen}}}_{\mathsf{alg}}(A)$ is the $\mathbb{Z}_2$-graded group of automorphisms and anti-automorphisms of $A$, and elements $Q_{\varsigma} \in A_{\varsigma^2}$, $\varsigma \in \hat{\mathsf{G}} \backslash \mathsf{G}$, which satisfy the following conditions:
\begin{enumerate}[label=(\roman*)]
\setcounter{enumi}{5}
\item For each $g \in \mathsf{G}$ and $\omega \in \hat{\mathsf{G}}$, the linear map $\hat{\mathfrak{a}}_{\omega}: A \rightarrow A$ restricts to a map $A_g \rightarrow A_{\omega g^{\pi(\omega)} \omega^{-1}}$.

\item The trace $\langle - \rangle_e$ is $\hat{\mathsf{G}}$-invariant.

\item For each $\varsigma \in \hat{\mathsf{G}} \backslash \mathsf{G}$ and $\omega \in \hat{\mathsf{G}}$, the equality $\hat{\mathfrak{a}}_{\omega}(Q_{\varsigma}) = Q_{\omega \varsigma^{\pi(\omega)} \omega^{-1}}$ holds.

\item For each $\varsigma \in \hat{\mathsf{G}} \backslash \mathsf{G}$ and $x \in A_g$, the equality $Q_{\varsigma} x = \hat{\mathfrak{a}}_{\varsigma}(x)  Q_{\varsigma g}$ holds.

\item Let $\varsigma_1, \varsigma_2 \in \hat{\mathsf{G}} \backslash \mathsf{G}$. Given a basis $\{x^i_{\varsigma_1 \varsigma_2}\}_i$ of $A_{(\varsigma_1 \varsigma_2)^{-1}}$, with dual basis $\{x_i^{\varsigma_1 \varsigma_2}\}_i$ of $A_{\varsigma_1 \varsigma_2}$, the equality
\[
\sum_i \hat{\mathfrak{a}}_{\varsigma_1} (x^i_{\varsigma_1 \varsigma_2} ) x_i^{\varsigma_1 \varsigma_2} = Q_{\varsigma_1} Q_{\varsigma_2}
\]
holds in $A_{\varsigma^2_1 \varsigma^2_2}$.
\end{enumerate}

We briefly explain the origin of some of these conditions. The value of $\mathcal{Z}^{\pi}$ on the crosscap is a morphism which, in terms of the equivalence \eqref{eq:crosscapGroupoid}, has underlying span
\begin{equation}
\label{eq:ccSpan}
\begin{tikzpicture}[baseline= (a).base]
\node[scale=1] (a) at (0,0){
\begin{tikzcd}[column sep=1.5em, row sep=1.5em]
& \BBun^{\ori}_{\hat{\mathsf{G}}}(\mathbb{M}) \arrow{ld} \arrow{rd}[above]{t} & \\
\pt && \mathsf{G} \git \mathsf{G}.
\end{tikzcd}
};
\end{tikzpicture}
\end{equation}
The functor $t$ is given on objects by $t(\varsigma) = \varsigma^2$ and on morphisms by the identity. We must also give a morphism from the trivial line bundle on $\BBun^{\ori}_{\hat{\mathsf{G}}}(\mathbb{M})$ to $t^* \mathcal{R}_{\mathcal{Z}}^{S^1}$. The latter is the data of linear maps
\[
Q_{\varsigma}: \mathbb{C} \rightarrow \mathcal{R}_{\mathcal{Z}}^{S^1}(\varsigma^2),
\qquad
\varsigma \in \hat{\mathsf{G}} \backslash \mathsf{G}
\]
henceforth identified with their values on $1 \in \mathbb{C}$, which satisfy $\mathfrak{a}_g(Q_{\varsigma}) = Q_{g \varsigma g^{-1}}$. This explains the $\mathsf{G}$-sector of condition (viii); the $\hat{\mathsf{G}} \backslash \mathsf{G}$-sector is obtained by capping the constraint \eqref{eq:linConstraint} from the left. The constraint \eqref{eq:linConstraint} itself implies condition (ix). To see this, consider the enhancement of equation \eqref{eq:quadCrossConstr} to an equation in $\hat{\mathsf{G}} \mhyphen \Cob^{\pi}_{\langle 2,1\rangle}$ by equipping the cobordisms with orientation twisted $\mathsf{G}$-bundles whose holonomies along the indicated loops are
\[
\ell^{(1)}_1 = g, \qquad \ell^{(1)}_2 = \varsigma^2, \qquad \ell^{(1)}_3 = \varsigma^2 g
\]
and
\[
\ell^{(2)}_1 = g, \qquad \ell^{(2)}_2 = (\varsigma^{-1} g^{-1})^2, \qquad \ell^{(2)}_3 = \varsigma^{-1} g^{-1} \varsigma^{-1}, \qquad \ell^{(2)}_4 = \varsigma^2 g.
\]
A short argument using the Seifert--van Kampen theorem and Proposition \ref{prop:holonomyDescr} shows that
\[
\BBun_{\hat{\mathsf{G}}}^{\ori}(\Sigma_k) \simeq \{ (g^{\prime}, \omega^{\prime}) \in \mathsf{G} \times \hat{\mathsf{G}} \backslash \mathsf{G} \} \git \mathsf{G}, \qquad k=1,2
\]
where $g^{\prime}$ and $\omega^{\prime}$ are the holonomies around $\ell_1^{(k)}$ and $\ell_2^{(k)}$, respectively. The associated decorated spans are
\[
\begin{tikzpicture}[baseline= (a).base]
\node[scale=1] (a) at (0,0){
\begin{tikzcd}[column sep=3.0em, row sep=1.0em]
& \BBun_{\hat{\mathsf{G}}}^{\ori}(\Sigma_k) \arrow{ld}[above]{s} \arrow{rd}[above]{t_k} & \\
\mathsf{G} \git \mathsf{G} \arrow[d, "\mathcal{Z}^{\pi}(S^1)"{name=A, left}] && \mathsf{G} \git \mathsf{G} \arrow[d, "\mathcal{Z}^{\pi}(S^1)"{name=B, right}]\\
\Vect_{\mathbb{C}} && \Vect_{\mathbb{C}}.
\arrow[Rightarrow, shorten <= 1.35em, shorten >= 1.35em, from=A, to=B, "\mathcal{Z}^{\pi}(\Sigma_k)" above]
\end{tikzcd}
};
\end{tikzpicture}
\]
The functors are the identity on morphisms and are given on objects by $s(g,\omega) = g$, $t_1(g,\omega) = \omega^2 g$ and $t_2(g, \omega)= \varsigma (\omega^2 g)^{-1} \varsigma^{-1}$. The equality of these decorated spans, $k=1,2$, recovers condition (ix).

\begin{Def}
A tuple $(A, \hat{\mathfrak{a}}, \langle - \rangle_e, Q)$ for which $(A, \mathfrak{a}, \langle - \rangle_e)$ is a $\mathsf{G}$-Turaev algebra and which satisfies conditions \textnormal{(vi)-(x)} is called an unoriented $\hat{\mathsf{G}}$-Turaev algebra.
\end{Def}

\begin{Prop}
\label{prop:unoriEquivClass}
The orientation twisted pushforward
\[
(-)^{\pi}: \hat{\mathsf{G}} \mhyphen \TFT^{\pi}_{\langle 2, 1\rangle} \rightarrow \TFT_{\langle 2,1\rangle} (\Vect_{\mathbb{C}}(\Grpd))
\]
is an equivalence onto the non-full subcategory of $\TFT_{\langle 2, 1\rangle} (\Vect_{\mathbb{C}}(\Grpd))$ spanned by unoriented $\hat{\mathsf{G}}$-Turaev algebras.
\end{Prop}

\begin{proof}
For special $\hat{\mathsf{G}}$, the proposition was proved at the level of objects in \cite[Theorem 3.11]{tagami2012}, \cite[Proposition 3.2.7]{sweet2013}. For the general case in the pointed setting, see \cite[Proposition 1]{kapustin2017}. The functor $(-)^{\hat{\phi}}$ is faithful. Since morphisms of $\mathsf{G}$-Turaev algebras are the identity on underlying spans, $(-)^{\hat{\phi}}$ is also full when regarded as a functor to unoriented $\hat{\mathsf{G}}$-Turaev algebras. Now use the equivalence of pointed and non-pointed theories.
\end{proof}

Unoriented $\hat{\mathsf{G}}$-Turaev algebras were discovered by Kapustin--Turzillo \cite[\S 4.2]{kapustin2017}, who argued along the lines of Turaev's derivation of $\mathsf{G}$-Turaev algebras. See \cite{tagami2012}, \cite{sweet2013} for earlier work in this direction. The above discussion shows that unoriented $\hat{\mathsf{G}}$-Turaev algebras are simply unoriented Frobenius algebras in $\Vect_{\mathbb{C}}(\Grpd)$ whose underlying Frobenius algebra is $\mathsf{G}$-Turaev. It is surprising that there are no additional conditions, analogous to (iv) and (v) in the oriented case, which must be added by hand to arrive at the definition of an unoriented $\hat{\mathsf{G}}$-Turaev algebra.

A $\mathsf{G}$-Turaev algebra $A$ defines, in particular, a flat vector bundle on the loop groupoid $\Lambda B \mathsf{G}$, the fibre over $g \in \Lambda B \mathsf{G}$ being $A_g$ and the parallel transport along $h$ being induced by $\mathfrak{a}_h$. This perspective is emphasized in \cite{shklyarov2008}. In the same vein, an unoriented $\hat{\mathsf{G}}$-Turaev algebra defines an extension of $A$ to a flat vector bundle on the unoriented loop groupoid $\Lambda_{\pi}^{\refl} B \hat{\mathsf{G}}$. This geometric perspective makes the orientation twisted orbifold procedure particularly natural.

\begin{Prop}
\label{prop:unoriOrbifold}
The orbifold of $\mathcal{Z} \in \hat{\mathsf{G}} \mhyphen \TFT^{\pi}_{\langle 2,1 \rangle}$ is the unoriented Frobenius algebra whose underlying commutative Frobenius algebra is the space of flat sections $\Gamma_{\Lambda B \mathsf{G}}(A)$, whose involution $p$ is induced by the extension of $A$ to $\Lambda_{\pi}^{\refl} B \hat{\mathsf{G}}$ and whose crosscap is the flat section given by the assignment
\[
g \mapsto \sum_{\{\varsigma \in \hat{\mathsf{G}} \backslash \mathsf{G} \mid \varsigma^2 = g \}} Q_{\varsigma}, \qquad g \in \mathsf{G}.
\]
\end{Prop}

\begin{proof}
That the underlying commutative Frobenius algebra of $\mathcal{Z}^{\pi, \textnormal{orb}}$ is $\Gamma_{\Lambda B \mathsf{G}}(A)$ was proved in \cite[Theorem 4.9]{schweigert2019}; see also \cite[Proposition 2.1]{kaufmann2003}. Keeping the notation of diagram \eqref{eq:ccSpan}, we find that the component of the charge of the crosscap at $g \in \Lambda B \mathsf{G}$ is
\begin{eqnarray*}
\Par(\mathcal{Z}^{\pi}(\mathbb{M}))(\mathbf{1})(g)
&=& \int_{(\varsigma, k) \in Rt^{-1}(g)} \mathcal{Z}^{\mathsf{G} \rightarrow \pt}(S^1)(k)(Q_{\varsigma})\\
&=&
\sum_{\{(\varsigma, k) \in (\hat{\mathsf{G}} \backslash \mathsf{G}) \times \mathsf{G} \mid k \varsigma^2 k^{-1} = g \}} \frac{\mathfrak{a}_k(Q_{\varsigma})}{\vert \mathsf{G} \vert} \\
&\overset{\scriptstyle \textnormal{Cond. (viii)}}{=}&
\sum_{\{(\varsigma, k) \in (\hat{\mathsf{G}} \backslash \mathsf{G}) \times \mathsf{G} \mid k \varsigma^2 k^{-1} = g \}} \frac{Q_{k \varsigma k^{-1}}}{\vert \mathsf{G} \vert} \\
&=&
\sum_{\{\varsigma \in \hat{\mathsf{G}} \backslash \mathsf{G} \mid \varsigma^2 = g \}} Q_{\varsigma}.
\end{eqnarray*}
Here $\mathbf{1}$ is the unit section of the trivial line bundle and $\int$ denotes integration with respect to groupoid cardinality, as in \cite[\S 2.1]{willerton2008}. The first equality follows from the explicit description of $\Par$ given in \cite[Remarks 3.18(a)]{schweigert2019}. For the second equality we have used that $\mathfrak{a}_k$ is the value of $\mathcal{Z}^{\mathsf{G} \rightarrow \pt}(S^1)$ on the morphism $(-) \rightarrow k(-) k^{-1}$ in $\Lambda B \mathsf{G}$.

A similar calculation shows that the involution of $\Gamma_{\Lambda B \mathsf{G}}(A)$ is given by the restriction of the anti-automorphism $\hat{\mathfrak{a}}_{\varsigma}$, for any choice of  $\varsigma \in \hat{\mathsf{G}} \backslash \mathsf{G}$.
\end{proof}

\section{Twisted unoriented Dijkgraaf--Witten theory}
\label{sec:twistUnoriDW}

We use the results of the previous sections to construct unoriented lifts of twisted Dijkgraaf--Witten theory. We study this theory in detail in dimensions one and two. For related discussions in the oriented case, see \cite{dijkgraaf1990}, \cite{freed1993}, \cite{freed1994}, \cite{freed2010}, \cite{morton2015}, \cite{muller2018}.

\subsection{Construction}
\label{sec:unoriDWConstr}

Given a finite group $\mathsf{G}$ and an $n$-cocycle $\lambda \in Z^n(B \mathsf{G}; \mathsf{U}(1))$, denote by $\mathcal{Z}_{\mathsf{G}}^{\lambda} \in \TFT_{\langle n, n-1, n-2 \rangle}^{\ori}$ the associated oriented Dijkgraaf--Witten theory, as constructed in \cite[Definition 5.1]{muller2018}.

\begin{Def}
Let $\hat{\mathsf{G}}$ be a finite $\mathbb{Z}_2$-graded group and let $\hat{\lambda} \in Z^n(B \hat{\mathsf{G}}; \mathsf{U}(1)_{\pi})$. The unoriented topological quantum field theory
\[
\mathcal{Z}_{\hat{\mathsf{G}}}^{\hat{\lambda}} : \Cob_{\langle n, n-1, n-2 \rangle} \rightarrow 2 \Vect_{\mathbb{C}}
\]
given by the orbifold of the orientation twisted theory $\mathcal{P}_{\hat{\mathsf{G}}}^{\hat{\lambda}}$ of Corollary \ref{cor:primEquivTheory} is called the twisted unoriented Dijkgraaf--Witten theory associated to the pair $(\hat{\mathsf{G}}, \hat{\lambda})$.
\end{Def}

We can now state our main result, which justifies our naming of $\mathcal{Z}_{\hat{\mathsf{G}}}^{\hat{\lambda}}$.

\begin{Thm}
\label{thm:unoriDW}
The theory $\mathcal{Z}_{\hat{\mathsf{G}}}^{\hat{\lambda}}$ is an unoriented lift of $\mathcal{Z}_{\mathsf{G}}^{\lambda}$.
\end{Thm}

\begin{proof}
The oriented theory $\mathcal{Z}_{\mathsf{G}}^{\lambda}$ was constructed in \cite{muller2018} as the $\mathsf{G}$-orbifold of $\mathcal{P}_{\mathsf{G}}^{\lambda}$. That $\mathcal{Z}_{\hat{\mathsf{G}}}^{\hat{\lambda}}$ is an unoriented lift of $\mathcal{Z}_{\mathsf{G}}^{\lambda}$ therefore follows by applying Corollary \ref{cor:primEquivTheory} and the object-level statement of Proposition \ref{prop:orbCompatibility}.
\end{proof}

\subsection{Relation with orientation twisted transgression}

As expected from the point of view of path integral (or cohomological) quantization, the values of $\mathcal{Z}_{\hat{\mathsf{G}}}^{\hat{\lambda}}$ can be expressed in terms of appropriate pushforwards of $\hat{\lambda}$. In the present setting, the following result shows that the relevant pushforward is the orientation twisted transgression map of Section \ref{sec:unoriTrans}.

\begin{Thm}
\phantomsection
\label{thm:closedPartition}
\begin{enumerate}[label=(\roman*)]
\item For $Z$ a closed $n$-manifold, we have
\[
\mathcal{Z}_{\hat{\mathsf{G}}}^{\hat{\lambda}}(Z)
=
\int_{\BBun_{\hat{\mathsf{G}}}^{\ori}(Z)} \uptau^{\ori}_Z(\hat{\lambda}).
\]

\item For $Y$ a closed $(n-1)$-manifold, we have
\[
\mathcal{Z}_{\hat{\mathsf{G}}}^{\hat{\lambda}}(Y)
\simeq
\Gamma_{\BBun_{\hat{\mathsf{G}}}^{\ori}(Y)}(\uptau_Y^{\ori}(\hat{\lambda})_{\mathbb{C}}),
\]
where $\uptau_Y^{\ori}(\hat{\lambda})_{\mathbb{C}}$ is the flat complex line bundle on $\BBun_{\hat{\mathsf{G}}}^{\ori}(Y)$ determined by the $1$-cocycle $\uptau_Y^{\ori}(\hat{\lambda})$.

\item For $X$ a closed $(n-2)$-manifold, we have
\[
\mathcal{Z}_{\hat{\mathsf{G}}}^{\hat{\lambda}}(X)
\simeq
\Vect_{\mathbb{C}}^{\uptau_X^{\ori}(\hat{\lambda})}(\BBun_{\hat{\mathsf{G}}}^{\ori}(X)),
\]
where the right hand side is the category of $\uptau_X^{\ori}(\hat{\lambda})$-twisted complex vector bundles on $\BBun_{\hat{\mathsf{G}}}^{\ori}(X)$.
\end{enumerate}
\end{Thm}

\begin{Rem}
In view of the compatibility of oriented and orientation twisted transgression (see Proposition \ref{prop:oriVsOriTwistTrans}), Theorem \ref{thm:closedPartition} specializes to the known expressions \cite[\S 4.3]{muller2018} for $\mathcal{Z}_{\mathsf{G}}^{\lambda}$ on closed oriented manifolds of codimension at most two.
\end{Rem}

\begin{proof}[Proof of Theorem \ref{thm:closedPartition}.]
Let $Z$ be a closed $n$-manifold. Then $\mathcal{P}_{\hat{\mathsf{G}}}^{\hat{\lambda}}(Z)$ is an intertwiner of the identity morphism of the trivial $2$-line bundle on $\BBun_{\hat{\mathsf{G}}}^{\ori}(Z)$. Unravelling definitions, we find that this is the data of a locally constant function on $\BBun_{\hat{\mathsf{G}}}^{\ori}(Z)$. The construction of $\mathcal{P}_{\hat{\mathsf{G}}}^{\hat{\lambda}}(Z)$ gives
\[
\mathcal{P}_{\hat{\mathsf{G}}}^{\hat{\lambda}}(Z)(F;H)
=
\langle H(F^* \hat{\lambda}), c_Z\rangle = \uptau_Z^{\ori}(\hat{\lambda})(F;H)
\]
for each $(F;H) \in \BBun_{\hat{\mathsf{G}}}^{\ori}(Z)$. We then find
\[
\mathcal{Z}_{\hat{\mathsf{G}}}^{\hat{\lambda}}(Z) = \sum_{(F;H) \in \pi_0(\BBun_{\hat{\mathsf{G}}}^{\ori}(Z))} \frac{\uptau_Z^{\ori}(\hat{\lambda})(F;H)}{\vert \Aut_{\BBun_{\hat{\mathsf{G}}}^{\ori}(Z)}(F;H) \vert} = \int_{\BBun_{\hat{\mathsf{G}}}^{\ori}(Z)} \uptau^{\ori}_Z(\hat{\lambda}),
\]
as claimed.

Let now $Y$ be a closed $(n-1)$-manifold. Then $\mathcal{P}_{\hat{\mathsf{G}}}^{\hat{\lambda}}(Y)$ is a $1$-endomorphism of the trivial $2$-line bundle on $\BBun_{\hat{\mathsf{G}}}^{\ori}(Y)$ and so can be interpreted as an element of $Z^1(\BBun_{\hat{\mathsf{G}}}^{\ori}(Y) ; \mathsf{U}(1))$. Let $H: (f_1; h_1) \rightarrow (f_2; h_2)$ be a morphism in $\MMap_{\mathbf{B} \mathbb{Z}_2}(Y, \mathbf{B} \hat{\mathsf{G}}) \simeq \BBun_{\hat{\mathsf{G}}}^{\ori}(Y)$. Note that $H$ determines a $1$-chain on $\BBun_{\hat{\mathsf{G}}}^{\ori}(Y)$. As in the proof of Proposition \ref{prop:vbMappingSpace}, the choice of a homotopy from $h_2 * (\pi \circ H)$ to $h_1$ induces a homotopy $G$ from $\pi \circ H$ to $\ori_{Y \times I}$. There is a commutative diagram
\[
\begin{tikzpicture}[baseline= (a).base]
\node[scale=1] (a) at (0,0){
\begin{tikzcd}[column sep=3.0em, row sep=2.0em]
Y \times I \arrow{r}[above]{H} \arrow{d}[left]{\id_Y \times H^{\textnormal{adj}}} & \mathbf{B} \hat{\mathsf{G}} \\
Y \times \MMap_{\mathbf{B} \mathbb{Z}_2}(Y, \mathbf{B} \hat{\mathsf{G}}) \arrow{ru}[below right]{\ev_{\hat{\mathsf{G}}}^{\ori} \circ \textnormal{swap}} & {}
\end{tikzcd}
};
\end{tikzpicture}
\]
where $H^{\textnormal{adj}}$ is the adjoint of $H$. Fix a fundamental cycle $c_Y$ of $Y$, thereby identifying $\mathcal{P}_{\hat{\mathsf{G}}}^{\hat{\lambda}}(Y, f_k; h_k)$ with $\mathbb{C}$, $k=1,2$. With this notation, we compute (\textit{cf.} \cite[\S 4.3]{muller2018})
\begin{eqnarray*}
\mathcal{P}_{\hat{\mathsf{G}}}^{\hat{\lambda}}(Y)(H)
&=&
\langle G(H^* \hat{\lambda}), (-1)^{n-1} c_Y \times I \rangle \\
&=&
\langle \ev_{\hat{\mathsf{G}}}^{\ori *} \hat{\lambda} , (-1)^{n-1} \textnormal{swap}_*(c_Y \times H^{\textnormal{adj}}(I) )\rangle \\
&=&
\langle \ev_{\hat{\mathsf{G}}}^{\ori *} \hat{\lambda} ,H^{\textnormal{adj}}(I) \times c_Y \rangle \\
&=&
\uptau_Y^{\ori}(\hat{\lambda})(H),
\end{eqnarray*}
where we have used the map $G$ to identify $\ev_{\hat{\mathsf{G}}}^{\ori *} \hat{\lambda}$ with an $\ori_{Y \times I}$-twisted $n$-cocycle on $\BBun_{\hat{\mathsf{G}}}^{\ori}(Y \times I)$. Finally, the definition of $\Par$ shows that the vector space $\mathcal{Z}_{\hat{\mathsf{G}}}^{\hat{\lambda}}(Y)$ is naturally identified with the space of flat sections of $\mathcal{P}_{\hat{\mathsf{G}}}^{\hat{\lambda}}(Y) = \uptau_Y^{\ori}(\hat{\lambda})$.

The computation of $\mathcal{Z}_{\hat{\mathsf{G}}}^{\hat{\lambda}}(X)$, for $X$ a closed $(n-2)$-manifold, is very similar to that in the oriented case \cite[Theorem 4.5]{muller2018}, and we will not repeat it here. The only new ingredient is the asphericity of the relative mapping space $\Map_{\mathbf{B} \mathbb{Z}_2}(X, \mathbf{B} \hat{\mathsf{G}})$, which was established in Lemma \ref{lem:asphericalMappingGrpd}.
\end{proof}

\begin{Rem}
More generally, Proposition \ref{prop:diffTransgression} can be used to express all values of the non-extended truncation of $\mathcal{Z}_{\hat{\mathsf{G}}}^{\hat{\lambda}}$. Turned around, Proposition \ref{prop:diffTransgression} can be used to give an \textit{a priori} construction of the non-extended truncation of $\mathcal{Z}_{\hat{\mathsf{G}}}^{\hat{\lambda}}$.
\end{Rem}

The expression from part (i) of Theorem \ref{thm:closedPartition} can be rewritten in terms of orientation twisted $\mathsf{G}$-bundles as
\begin{equation}
\label{eq:DWPartitionFunction}
\mathcal{Z}_{\hat{\mathsf{G}}}^{\hat{\lambda}}(Z)
=
\sum_{[(P,\epsilon)] \in \pi_0(\BBun_{\hat{\mathsf{G}}}^{\ori}(Z))} \frac{\langle \epsilon(f_P^*\hat{\lambda}), [Z] \rangle}{\vert \Aut_{\BBun_{\hat{\mathsf{G}}}^{\ori}(Z)}(P, \epsilon) \vert}.
\end{equation}
Using Proposition \ref{prop:holonomyDescr}, this can also be written in terms of representations of the fundamental group of $Z$ on $\hat{\mathsf{G}}$, generalizing a familiar formula from the oriented case:
\[
\mathcal{Z}_{\hat{\mathsf{G}}}^{\hat{\lambda}}(Z) = \frac{1}{\vert \mathsf{G} \vert} \sum_{\rho \in \Hom_{\Grp}^{\ori_Z}(\pi_1(Z), \hat{\mathsf{G}})}
\frac{\langle f_{P_{\rho}}^*\hat{\lambda}, [Z] \rangle}{\vert \Aut_{\mathsf{G}}(\rho) \vert}.
\]

Using the results of Willerton \cite[\S 2]{willerton2008}, we also find that
\[
\dim_{\mathbb{C}} \mathcal{Z}_{\hat{\mathsf{G}}}^{\hat{\lambda}}(Y) = \int_{\Lambda \BBun_{\hat{\mathsf{G}}}^{\ori}(Y)} \uptau_{S^1} \uptau_Y^{\ori} (\hat{\lambda})
\]
and
\[
K_0(\mathcal{Z}_{\hat{\mathsf{G}}}^{\hat{\lambda}}(X)) \otimes_{\mathbb{Z}} \mathbb{C} \simeq \Gamma_{\Lambda \BBun_{\hat{\mathsf{G}}}^{\ori}(X)} (\uptau_{S^1} \uptau_X^{\ori} (\hat{\lambda})_{\mathbb{C}}),
\]
the latter of which has dimension $\int_{\Lambda^2 \BBun_{\hat{\mathsf{G}}}^{\ori}(X)} \uptau_{\mathbb{T}^2} \uptau_X^{\ori} (\hat{\lambda})$.

\begin{Ex}
Consider the trivial Real structure, $\hat{\mathsf{G}} =\mathsf{G} \times \mathbb{Z}_2$, and the trivial cocycle, $\hat{\lambda} =1$. By combining Proposition \ref{prop:splitFrBun} and Theorem \ref{thm:closedPartition}, we find, for example,
\[
\mathcal{Z}_{\hat{\mathsf{G}} = \mathsf{G} \times \mathbb{Z}_2}^{\hat{\lambda}=1}(Z) = \frac{\vert \Hom_{\Grp}(\pi_1(Z), \mathsf{G}) \vert}{\vert \mathsf{G} \vert},
\qquad
\mathcal{Z}_{\hat{\mathsf{G}} = \mathsf{G} \times \mathbb{Z}_2}^{\hat{\lambda}=1}(Y)
\simeq
\bigoplus_{[P] \in \pi_0(\BBun_{\mathsf{G}}(Y))} \mathbb{C} \cdot [P].
\]
In fact, inspection of the definitions shows that the orientation twisted pushforward of $\mathcal{P}_{\hat{\mathsf{G}} = \mathsf{G} \times \mathbb{Z}_2}^{\hat{\lambda}=1}$ is simply $\BBun_{\mathsf{G}}(-)$, viewed as a symmetric monoidal pseudofunctor $\Cob_{\langle n, n-1, n-2 \rangle} \rightarrow \mathsf{Span}(\Grpd)$. We therefore recover the unoriented untwisted Dijkgraaf--Witten theory from \cite{freed1993}, \cite{freed2018}.
\end{Ex}

More generally, the partial orbifold of $\mathcal{P}_{\hat{\mathsf{G}}}^{\hat{\lambda}}$ along a morphism $\hat{\phi}: \hat{\mathsf{G}} \rightarrow \hat{\mathsf{H}}$ is an orientation twisted lift of the $\mathsf{H}$-equivariant twisted oriented Dijkgraaf--Witten theory studied in \cite{maier2012}, \cite{muller2018}. We leave the study of these theories for later work.

\subsection{One dimension}
\label{sec:oneDim}

We begin with the rather simple one dimensional case. For different perspectives on the oriented theory, see \cite[\S 5]{freed1993}, \cite[\S 1]{freed2010}. Let $\hat{\lambda} \in Z^1(B \hat{\mathsf{G}}; \mathsf{U}(1)_{\pi})$. The oriented pushforward $\mathcal{A}_{\mathsf{G}}^{\lambda}$ of $\mathcal{P}^{\lambda}_{\mathsf{G}}$ is determined by its value on the positively oriented point,
\[
\mathcal{A}_{\mathsf{G}}^{\lambda}(\pt^+): B \mathsf{G} \rightarrow \Vect_{\mathbb{C}},
\]
which is the one dimensional representation $\rho_{\lambda}$ of $\mathsf{G}$ in which $\rho_{\lambda}(g) : \mathbb{C} \rightarrow \mathbb{C}$ is multiplication by $\lambda([g])$.

Consider then the unoriented theory. The value of $\mathcal{A}_{\hat{\mathsf{G}}}^{\hat{\lambda}}$ on the incompatibly oriented interval
\begin{equation}
\label{eq:unoriInt}
I^{-+}= 
\begin{tikzpicture}[very thick,scale=2.01,color=black,baseline]
\coordinate (r1) at (0,0);
\coordinate (r2) at (1,0);
\draw (r1) to (r2);
\draw[very thick, black,decoration={markings, mark=at position 0.55 with {\arrow{>}}},        postaction={decorate}] (r1) -- (r2);
\node at (0,0.2)  {\scriptsize $-$};
\node at (1,0.2)  {\scriptsize $+$};
\node[draw=blue!80!black,circle,fill=blue!80!black,scale=0.25] at (r1) {};
\node[draw=blue!80!black,circle,fill=blue!80!black,scale=0.25] at (r2) {};
\end{tikzpicture}
\end{equation}
is the diagram
\begin{equation}
\label{eq:intervalSpan}
\begin{tikzpicture}[baseline= (a).base]
\node[scale=1] (a) at (0,0){
\begin{tikzcd}[column sep=3.0em, row sep=1.0em]
& B \mathsf{G} \arrow{ld}[above left]{\textnormal{Ad}_{\varsigma}} \arrow{rd}[above right]{\id} & \\
B \mathsf{G} \arrow[d, "\mathcal{A}_{\mathsf{G}}^{\lambda}(\pt^-)"{name=A, left}] && B\mathsf{G} \arrow[d, "\mathcal{A}_{\mathsf{G}}^{\lambda}(\pt^+)"{name=B, right}]\\
\Vect_{\mathbb{C}} && \Vect_{\mathbb{C}}.
\arrow[Rightarrow, shorten <= 1.35em, shorten >= 1.35em, from=A, to=B, "\mathcal{A}_{\hat{\mathsf{G}}}^{\hat{\lambda}}(I^{-+})" above]
\end{tikzcd}
};
\end{tikzpicture}
\end{equation}
Here $\varsigma$ is an arbitrary element of $\hat{\mathsf{G}} \backslash \mathsf{G}$ and $\textnormal{Ad}_{\varsigma}$ is the weak involution given by conjugation by $\varsigma$. Up to equivalence, the underlying span of groupoids depends only on $\hat{\mathsf{G}}$, and not on the particular choice of $\varsigma$. The functor $\textnormal{Ad}_{\varsigma}$ arises by first applying Proposition \ref{prop:restrFunctor}, and then identifying $\BBun_{\hat{\mathsf{G}}}^{\ori}(I)$ and $\BBun_{\hat{\mathsf{G}}}^{\ori}(\pt)$ with their skeleta consisting of the trivial $\hat{\mathsf{G}}$-bundle with the identity orientation framing. That $\mathcal{A}_{\hat{\mathsf{G}}}^{\hat{\lambda}}(I^{-+})$ is indeed equal to \eqref{eq:intervalSpan} now follows from the observation that gauge transformations coming from $\hat{\mathsf{G}} \backslash \mathsf{G}$ reverse orientation framings.

General principles of topological field theory imply that $\mathcal{A}_{\mathsf{G}}^{\lambda}(\pt^-)$ is isomorphic to $\mathcal{A}_{\mathsf{G}}^{\lambda}(\pt^+)^{\vee}$, the dual of $\mathcal{A}_{\mathsf{G}}^{\lambda}(\pt^+)$ considered as an object of the monoidal category $\Vect_{\mathbb{C}}(\Grpd)$. This coincides with the dual representation $\rho_{\lambda}^{\vee}$. The natural transformation $\mathcal{A}_{\hat{\mathsf{G}}}^{\hat{\lambda}}(I^{-+})$ is therefore determined by a linear map $\mathfrak{i}:\mathbb{C}^{\vee} \rightarrow \mathbb{C}$
which satisfies
\[
\mathfrak{i} \circ \rho(\varsigma g \varsigma^{-1})^{- \vee} = \rho_{\lambda}(g) \circ \mathfrak{i}, \qquad g \in \mathsf{G}.
\]
The definition of $\mathcal{P}^{\hat{\lambda}}_{\hat{\mathsf{G}}}$ shows that $\mathfrak{i}$ is multiplication by $\hat{\lambda}([\varsigma^{-1}])$. In this way, $\mathcal{A}_{\hat{\mathsf{G}}}^{\hat{\lambda}}$ lifts $\rho_{\lambda}$ to the Real representation of $\mathsf{G}$ determined by $\hat{\lambda}$, in the sense of \cite[\S 2.2]{mbyoung2018c}.

Passing to quantum theories, that is, applying the functor $\Par$, amounts to taking $\mathsf{G}$-invariants. In the oriented case we get
\[
\mathcal{Z}_{\mathsf{G}}^{\lambda}(\pt^+) = \mathbb{C}^{(\rho_{\lambda},\mathsf{G})},
\]
which is non-trivial if and only if $\lambda$ is trivial. In the unoriented case $\mathbb{C}^{(\rho_{\lambda},\mathsf{G})}$ inherits an orthogonal structure from the Real structure on $\rho_{\lambda}$.

\subsection{Two dimensions}
\label{sec:twoDim}

Discussions of two dimensional oriented Dijkgraaf--Witten theory can be found in \cite[\S 6]{freed1994}, \cite[\S 2.3]{freed2010}.

Fix $\hat{\lambda} \in Z^2(B \hat{\mathsf{G}}; \mathsf{U}(1)_{\pi})$. The oriented theory $\mathcal{A}_{\mathsf{G}}^{\lambda}$ assigns to $\pt^+$ the pseudofunctor
\[
\mathcal{A}_{\mathsf{G}}^{\lambda}(\pt^+):
B \mathsf{G} \rightarrow 2 \Vect_{\mathbb{C}}
\]
given by the $2$-representation $\rho$ of $\mathsf{G}$ on $\Vect_{\mathbb{C}} \in 2 \Vect_{\mathbb{C}}$ in which $\mathsf{G}$ acts by identity functors and the coherence $2$-isomorphisms
\[
\psi_{g_2, g_1} : \rho(g_2) \circ \rho(g_1) \Rightarrow \rho(g_2 g_1)
\]
are multiplication by $\lambda([g_2 \vert g_1])$. The value of $\mathcal{A}_{\hat{\mathsf{G}}}^{\hat{\lambda}}$ on the incompatibly oriented interval \eqref{eq:unoriInt} is again of the form \eqref{eq:intervalSpan}, but with $\Vect_{\mathbb{C}}$ replaced with $2 \mathsf{Vect}_{\mathbb{C}}$. The dual $\mathcal{A}_{\mathsf{G}}^{\lambda}(\pt^+)^{\vee}$ is naturally identified with the category of linear functors from $\Vect_{\mathbb{C}}$ to the monoidal unit (which is also $\Vect_{\mathbb{C}}$) with the induced $2$-representation structure. The intertwiner $\mathcal{A}_{\hat{\mathsf{G}}}^{\hat{\lambda}}(I^{-+})$ is the data of a linear functor $\beta_{\pt}:\rho^{\vee} \rightarrow \rho$ and natural isomorphisms
\[
\beta(g) :\rho(g) \circ \beta_{\pt} \Rightarrow \beta_{\pt} \circ \rho(\varsigma g \varsigma^{-1})^{\vee}, \qquad g \in \mathsf{G}.
\]
Explicitly, $\beta(g)$ is multiplication by $\frac{\hat{\lambda}([g \vert \varsigma^{-1}])}{\hat{\lambda}([\varsigma^{-1} \vert \varsigma g \varsigma^{-1}])}$. In particular, if we set $\rho(\varsigma) = \beta_{\pt}$ for all $\varsigma \in \hat{\mathsf{G}} \backslash \mathsf{G}$ and $\psi_{\omega_2, \omega_1}= \hat{\lambda}([\omega_2 \vert \omega_1])$ for all $\omega_1, \omega_2 \in \hat{\mathsf{G}}$, so that
\[
\psi_{\varsigma^{-1}, \varsigma g \varsigma^{-1}}^{-1} \circ \psi_{g, \varsigma^{-1}} = \beta(g),
\]
then $\mathcal{A}_{\hat{\mathsf{G}}}^{\hat{\lambda}}(I^{-+})$ defines a lift of $\rho$ to a one dimensional Real $2$-representation of $\mathsf{G}$, in the sense of \cite[\S 5.3]{mbyoung2018c}.

Passing to quantum theories amounts to taking bicategories of $\mathsf{G}$-equivariant objects. Hence, the oriented theory assigns to the point
\[
\mathcal{Z}_{\mathsf{G}}^{\lambda} (\pt^+) \simeq \mathcal{A}_{\mathsf{G}}^{\lambda} (\pt^+)^{\mathsf{G}} \simeq \mathsf{Rep}_{\mathbb{C}}^{\lambda}(\mathsf{G}),
\]
the category of finite dimensional $\lambda$-twisted representations of $\mathsf{G}$ endowed with its standard Calabi--Yau structure. See \cite[Proposition 5.4]{muller2018}. The unoriented lift $\mathcal{Z}_{\hat{\mathsf{G}}}^{\hat{\lambda}}$ determines is a $\mathsf{G}$-torsor of exact duality structures on $\mathsf{Rep}_{\mathbb{C}}^{\lambda}(\mathsf{G})$. Concretely, given an element $\varsigma \in \hat{\mathsf{G}} \backslash \mathsf{G}$, the associated duality functor
\[
P^{\varsigma} : \mathsf{Rep}_{\mathbb{C}}^{\lambda}(\mathsf{G})^{\op} \rightarrow \mathsf{Rep}_{\mathbb{C}}^{\lambda}(\mathsf{G})
\]
sends a $\lambda$-twisted representation $(V, \phi)$ to $(V^{\vee}, \phi^{\varsigma})$, where
\[
\phi^{\varsigma}(g) = \uptau_{\pi}^{\refl}(\hat{\lambda})([\varsigma] g)^{-1} \phi(\varsigma g^{-1} \varsigma^{-1})^{\vee}, \qquad g \in \mathsf{G}
\]
and
\[
\uptau^{\refl}_{\pi}(\hat{\lambda})([\omega]g) = \hat{\lambda}([g^{-1} \vert g])^{\frac{\pi(\omega)-1}{2}} \frac{\hat{\lambda}([\omega g^{\pi(\omega)} \omega^{-1} \vert \omega])}{\hat{\lambda}([\omega \vert g^{\pi(\omega)}])}, \qquad g \in \mathsf{G}, \; \omega \in \hat{\mathsf{G}}.
\]
The natural isomorphism $\Theta^{\varsigma}: 1_{\mathsf{Rep}^{\lambda}(\mathsf{G})} \Rightarrow P^{\varsigma} \circ (P^{\varsigma})^{\op}$ has components
\[
\Theta^{\varsigma}_{\phi} = \hat{\lambda}([\varsigma \vert \varsigma]) \ev_{\phi} \circ \phi(\varsigma^2), \qquad \phi \in \mathsf{Rep}^{\lambda}(\mathsf{G})
\]
where $\ev_{\phi}$ is the evaluation isomorphism from a finite dimensional vector space to its double dual. For each $g \in \mathsf{G}$, the natural transformation $F^{\varsigma}_g : P^{\varsigma} \Rightarrow P^{g \varsigma}$ whose component at $\phi \in \mathsf{Rep}_{\mathbb{C}}^{\lambda}(\mathsf{G})$ is $\hat{\lambda}([g \vert \varsigma]) \phi(g)^{- \vee}$ defines a non-singular form functor
\[
(1_{\mathsf{Rep}_{\mathbb{C}}^{\lambda}(\mathsf{G})}, F^{\varsigma}_g): (P^{\varsigma}, \Theta^{\varsigma}) \rightarrow (P^{g \varsigma}, \Theta^{g \varsigma}).
\]
A direct calculation shows that these natural transformations satisfy $F^{g_1\varsigma}_{g_2} \circ F^{\varsigma}_{g_1} = F^{\varsigma}_{g_2 g_1}$.

Alternatively, we can regard $\mathcal{Z}_{\hat{\mathsf{G}}}^{\hat{\lambda}}$ as taking values in the Morita bicategory $\mathsf{Alg}_{\mathbb{C}}$ of associative algebras, bimodules and intertwiners. In this case, $\mathcal{Z}_{\hat{\mathsf{G}}}^{\hat{\lambda}}(\pt^+)$ is the twisted group algebra $\mathbb{C}^{\lambda}[\mathsf{G}]$, considered as a symmetric Frobenius algebra. Elements of $\mathbb{C}^{\lambda}[\mathsf{G}]$ will be written as $\sum_{g \in \mathsf{G}} a_g l_g$ with $a_g \in \mathbb{C}$. Replacing the duality structure $P^{\varsigma}$ is the anti-automorphism
\[
p^{\varsigma} : \mathbb{C}^{\lambda}[\mathsf{G}]^{\op} \rightarrow \mathbb{C}^{\lambda}[\mathsf{G}], \qquad
\sum_{g \in \mathsf{G}} a_g l_g \mapsto  \sum_{g \in \mathsf{G}} \uptau_{\pi}^{\refl}(\hat{\lambda})([\varsigma] g)^{-1} a_g l_{\varsigma g^{-1} \varsigma^{-1}}.
\]
The natural isomorphism $\Theta^{\varsigma}$ becomes an isomorphism from the Morita context determined by $p^{\varsigma}$ to its opposite. Up to a scalar, this isomorphism is simply left multiplication by $l_{\varsigma^2}$.

Continuing, in dimension one we have
\[
\mathcal{Z}_{\hat{\mathsf{G}}}^{\hat{\lambda}}(S^1) \simeq \Gamma_{\Lambda B \mathsf{G}} (\uptau_{S^1}(\lambda)^{-1}_{\mathbb{C}}),
\]
identified with either $K_0(\mathsf{Rep}^{\lambda}_{\mathbb{C}}(\mathsf{G})) \otimes_{\mathbb{Z}} \mathbb{C}$ or the centre of $\mathbb{C}^{\lambda}[\mathsf{G}]$. The $\mathsf{G}$-torsor of duality structures or anti-automorphisms induces a single isometric involution $p$ of $\mathcal{Z}_{\hat{\mathsf{G}}}^{\hat{\lambda}}(S^1)$.

Turning to surfaces, it is well-known (see, for example, \cite[Corollary 12]{willerton2008}) that for the sphere and torus we have
\[
\mathcal{Z}_{\hat{\mathsf{G}}}^{\hat{\lambda}}(S^2)= 1,
\qquad \mathcal{Z}_{\hat{\mathsf{G}}}^{\hat{\lambda}}(\mathbb{T}^2) = \frac{1}{\vert \mathsf{G} \vert} \sum_{(g_1,g_2) \in \mathsf{G}^{(2)}} \frac{\lambda([g_2 \vert g_1])}{\lambda([g_1 \vert g_2])}.
\]

Consider then the crosscap $\mathbb{M}$. Using the equivalence \eqref{eq:crosscapGroupoid},  we find
\[
Q= \mathcal{Z}_{\hat{\mathsf{G}}}^{\hat{\lambda}}(\mathbb{M})= \sum_{\varsigma \in \hat{\mathsf{G}} \backslash \mathsf{G}} \hat{\lambda}([\varsigma \vert \varsigma]) l_{\varsigma^2} \in \mathcal{Z}_{\hat{\mathsf{G}}}^{\hat{\lambda}}(S^1).
\]
This is the $(\hat{\mathsf{G}}, \hat{\lambda})$-twisted Frobenius element, in the sense that the assignment $V \mapsto \tr_V Q $ restricts to the $(\hat{\mathsf{G}},\hat{\lambda})$-Frobenius--Schur indicator
\[
\nu^{(\hat{\mathsf{G}},\hat{\lambda})}: \textnormal{Irr}_{\mathbb{C}}^{\lambda}(\mathsf{G}) \rightarrow \{-1,0,+1\}
\]
on the set of irreducible $\lambda$-twisted representations of $\mathsf{G}$.
Using this observation, if we denote by $\mathfrak{p}_V \in \mathcal{Z}_{\hat{\mathsf{G}}}^{\hat{\lambda}}(S^1)$ the primitive orthogonal idempotent associated to $V \in \textnormal{Irr}_{\mathbb{C}}^{\lambda}(\mathsf{G})$, then we can write
\begin{equation}
\label{eq:frobSchurExpres}
Q= \sum_{V \in \textnormal{Irr}_{\mathbb{C}}^{\lambda}(\mathsf{G})} \nu^{(\hat{\mathsf{G}}, \hat{\lambda})}(V) \frac{\vert \mathsf{G} \vert}{\dim_{\mathbb{C}} V} \mathfrak{p}_V.
\end{equation}

Consider now the projective plane. In terms of the equivalence \eqref{eq:projPlaneGroupoid}, we compute $\uptau_{\mathbb{RP}^2}^{\ori}(\hat{\lambda})(\varsigma) = \hat{\lambda}([\varsigma \vert \varsigma])$, from which it follows that
\[
\mathcal{Z}_{\hat{\mathsf{G}}}^{\hat{\lambda}}(\mathbb{RP}^2) = \frac{1}{\vert \mathsf{G} \vert} \sum_{\substack{\varsigma \in \hat{\mathsf{G}} \backslash \mathsf{G} \\ \varsigma^2 =e}} \hat{\lambda}([\varsigma \vert \varsigma]).
\]
In particular, $\mathcal{Z}_{\hat{\mathsf{G}}}^{\hat{\lambda}}(\mathbb{RP}^2)$ vanishes if $\hat{\mathsf{G}}$ is not a split Real structure. On the other hand, $\mathcal{Z}_{\hat{\mathsf{G}}}^{\hat{\lambda}}(\mathbb{RP}^2)$ can be computed via the equality $\mathcal{Z}_{\hat{\mathsf{G}}}^{\hat{\lambda}}(\mathbb{RP}^2) = \langle Q \rangle_e$. Doing so and using equation \eqref{eq:frobSchurExpres} gives
\[
\mathcal{Z}_{\hat{\mathsf{G}}}^{\hat{\lambda}}(\mathbb{RP}^2)  = \frac{1}{\vert \mathsf{G} \vert} \sum_{V \in \textnormal{Irr}_{\mathbb{C}}^{\lambda}(\mathsf{G})} \nu^{(\hat{\mathsf{G}}, \hat{\lambda})}(V) \dim_{\mathbb{C}} V.
\]
In this way, we obtain a formula for the signed\footnote{The twisted $2$-cocycle condition implies that $\hat{\lambda}([\varsigma \vert \varsigma]) \in \{ \pm 1\}$} number of odd square roots of $e$ in terms of the Real twisted representation theory of $\mathsf{G}$.

As a final example, in terms of the equivalence \eqref{eq:kleinGroupoid}, we compute for the Klein bottle
\[
\mathcal{Z}_{\hat{\mathsf{G}}}^{\hat{\lambda}}(\mathbb{K}) = \frac{1}{\vert \mathsf{G} \vert} \sum_{(g, \varsigma) \in \hat{\mathsf{G}}^{(2)}} \hat{\lambda}([g \vert g^{-1}])^{-1} \frac{\hat{\lambda}([g \vert \varsigma])}{\hat{\lambda}([\varsigma \vert g^{-1}])}.
\]
While $\mathcal{Z}_{\hat{\mathsf{G}}}^{\hat{\lambda}}(\mathbb{K})$ itself admits a representation theoretic interpretation  (see Theorem \ref{thm:twistedFS} below), there is a rather nice interpretation of the one loop partition function
\[
\hat{\mathcal{Z}}_{\hat{\mathsf{G}}}^{\hat{\lambda},1 \mhyphen \textnormal{loop}} = \frac{1}{2} \left(
\mathcal{Z}_{\hat{\mathsf{G}}}^{\hat{\lambda}}(\mathbb{T}^2) + \mathcal{Z}_{\hat{\mathsf{G}}}^{\hat{\lambda}}(\mathbb{K})
\right).
\]

\begin{Prop}
\label{prop:2dRealVerlinde}
The equality
\[
\hat{\mathcal{Z}}_{\hat{\mathsf{G}}}^{\hat{\lambda},1\mhyphen \textnormal{loop}} = \textnormal{rk} \, KR^{0 + \hat{\lambda}}(B \mathsf{G})
\]
holds, where $KR^{0 + \hat{\lambda}}(B \mathsf{G})$ is the $\hat{\lambda}$-twisted $KR$-theory of the groupoid $B \mathsf{G}$, considered as a double cover of $B \hat{\mathsf{G}}$.
\end{Prop}

\begin{proof}
It was proven in \cite{mbyoung2018a} that $KR^{0 + \hat{\lambda}}(B \mathsf{G})$ is of rank
\[
\int_{\Lambda \Lambda_{\pi}^{\refl} B \hat{\mathsf{G}}} \uptau_{S^1} \uptau^{\refl}_{\pi} (\hat{\lambda}).
\]
The statement then follows from the fact that $\BBun_{\hat{\mathsf{G}}}^{\ori, 1 \mhyphen \textnormal{loop}}$ is a double cover of $\Lambda \Lambda_{\pi}^{\refl} B \hat{\mathsf{G}}$ together with the observation that the pullback of $\uptau_{S^1} \uptau^{\refl}_{\pi}$ to $\BBun_{\hat{\mathsf{G}}}^{\ori, 1 \mhyphen \textnormal{loop}}$ is equal to $\uptau_{\mathbb{T}^2} \sqcup \uptau^{\ori}_{\mathbb{K}}$. The latter observation follows from the explicit expressions for these $2$-cocycles.
\end{proof}

\begin{Rem}
For comparison, \cite[Corollary 13]{willerton2008} gives the equality $\mathcal{Z}^{\lambda}_{\mathsf{G}}(\mathbb{T}^2) = \textnormal{rk} \, K^{0 + \lambda}(B \mathsf{G})$.
\end{Rem}

In general, the partition function of a closed nonorientable surface $\Sigma$ is given by the Verlinde-type formula
\[
\mathcal{Z}_{\hat{\mathsf{G}}}^{\hat{\lambda}}(\Sigma) = \frac{1}{\vert \mathsf{G} \vert^{\chi(\Sigma)}} \sum_{V \in \textnormal{Irr}_{\mathbb{C}}^{\lambda}(\mathsf{G})} (\nu^{(\hat{\mathsf{G}},\hat{\lambda})}(V) \dim_{\mathbb{C}} V)^{\chi(\Sigma)},
\]
where $\chi(-)$ is the topological Euler characteristic. This formula can be proved  by decomposing $\Sigma$ into cups, caps, genus adding operators
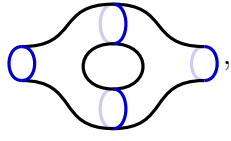
\[
\begin{tikzpicture}[very thick,scale=1.5,baseline=0,color=black]
\coordinate (p1) at (0,-0.575);
\coordinate (p2) at (0,-0.225);
\coordinate (p3) at (0,0.175);
\coordinate (p4) at (0,0.525);
\coordinate (p5) at (0.8,0.15);
\coordinate (p6) at (0.8,-0.15);
\draw (p2) .. controls +(0.35,0) and +(0.35,0) ..  (p3); 
\draw (p4) .. controls +(0.5,0) and +(-0.5,0) ..  (p5); 
\draw (p6) .. controls +(-0.5,0) and +(0.5,0) ..  (p1); 
\draw[very thick, blue!80!black,opacity=0.2] (p3) .. controls +(-0.15,0) and +(-0.15,0) ..  (p4); 
\draw[very thick, blue!80!black] (p5) .. controls +(0.15,0) and +(0.15,0) ..  (p6); 
\draw[very thick, blue!80!black,opacity=0.2] (p5) .. controls +(-0.15,0) and +(-0.15,0) ..  (p6); 
\coordinate (q1) at (0,-0.575);
\coordinate (q2) at (0,-0.225);
\coordinate (q3) at (0,0.175);
\coordinate (q4) at (0,0.525);
\coordinate (q5) at (-0.8,0.15);
\coordinate (q6) at (-0.8,-0.15);
\draw (q2) .. controls +(-0.35,0) and +(-0.35,0) ..  (q3); 
\draw (q4) .. controls +(-0.5,0) and +(0.5,0) ..  (q5); 
\draw (q6) .. controls +(0.5,0) and +(-0.5,0) ..  (q1); 
\draw[very thick, blue!80!black] (q1) .. controls +(0.15,0) and +(0.15,0) ..  (q2);
\draw[very thick, blue!80!black,opacity=0.2] (q1) .. controls +(-0.15,0) and +(-0.15,0) ..  (q2);
\draw[very thick, blue!80!black] (q3) .. controls +(0.15,0) and +(0.15,0) ..  (q4); 
\draw[very thick, blue!80!black] (q5) .. controls +(-0.15,0) and +(-0.15,0) ..  (q6); 
\draw[very thick, blue!80!black] (q5) .. controls +(0.15,0) and +(0.15,0) ..  (q6); 
\end{tikzpicture},
\]
as in the oriented case, and additionally crosscaps. Writing the genus adding operators in terms of $\{\mathfrak{p}_{V} \}_{V \in \textnormal{Irr}_{\mathbb{C}}^{\lambda}(\mathsf{G})}$ and the crosscap as equation \eqref{eq:frobSchurExpres} yields the claimed formula. 
Combined with equation \eqref{eq:DWPartitionFunction}, we obtain the following result.

\begin{Thm}
\label{thm:twistedFS}
If $\Sigma$ is a closed nonorientable surface, then
\[
\frac{1}{\vert \mathsf{G} \vert} \sum_{\rho \in \Hom^{\ori_{\Sigma}}_{\Grp}(\pi_1(\Sigma), \hat{\mathsf{G}})} \langle f_{\rho}^*\hat{\lambda}, [\Sigma] \rangle
=
\frac{1}{\vert \mathsf{G} \vert^{\chi(\Sigma)}} \sum_{V \in \textnormal{Irr}_{\mathbb{C}}^{\lambda}(\mathsf{G})} (\nu^{(\hat{\mathsf{G}},\hat{\lambda})}(V) \dim_{\mathbb{C}} V)^{\chi(\Sigma)}.
\]
\end{Thm}

Theorem \ref{thm:twistedFS} generalizes a number of earlier results. When both the Real structure and twisting are trivial, this is a theorem of Frobenius and Schur \cite{frobenius1906}; a proof using topological field theory was given in \cite{snyder2017}. More generally, when $\hat{\lambda}$ is in the image of
\[
\pi^*_{\mathsf{G}}: Z^2(B \mathsf{G}; \mathbb{Z}_2) \rightarrow Z^2(B (\mathsf{G} \times \mathbb{Z}_2); \mathsf{U}(1)_{\pi}),
\]
the above formula was proved in \cite[Theorem 4.1]{turaev2007}.

\begin{Rems}
\begin{enumerate}[label=(\roman*)]
\item By the cobordism hypothesis for two dimensional unoriented theories, in the form of \cite[Theorem 3.5.4]{schommer2011}, the triple $(\mathbb{C}^{\lambda}[\mathsf{G}], p^{\varsigma}, \Theta^{\varsigma})$ completely determines the theory $\mathcal{Z}_{\hat{\mathsf{G}}}^{\hat{\lambda}} : \Cob_{\langle 2,1,0 \rangle} \rightarrow \mathsf{Alg}_{\mathbb{C}}$.

\item Via Theorem \ref{thm:2dClassification}, the triple $(\Gamma_{\Lambda B \mathsf{G}}(\uptau(\lambda)^{-1}_{\mathbb{C}}), p, Q)$ determines the non-extended truncation of $\mathcal{Z}_{\hat{\mathsf{G}}}^{\hat{\lambda}}$. From this perspective, when the Real structure and twisting are trivial the non-extended theory was constructed in \cite[Theorem 4]{loktev2011}.

\item The theory $\mathcal{Z}_{\mathsf{G} \times \mathbb{Z}_2}^{\hat{\lambda}}$, with $\hat{\lambda}$ in the image of $\pi_{\mathsf{G}}^*$, was constructed using state sums in \cite{turaev2007}.

\item There is a generalization of the theories $\mathcal{Z}_{\hat{\mathsf{G}}}^{\hat{\lambda}}$ which clarifies the geometric meaning of the charge of the crosscap. Instead of $B \hat{\mathsf{G}}$, consider the action groupoid $X \git \hat{\mathsf{G}}$ associated to a finite $\hat{\mathsf{G}}$-set $X$. A cocycle $\hat{\lambda} \in Z^2(X \git \hat{\mathsf{G}} ; \mathsf{U}(1)_{\pi})$ determines an unoriented topological field theory $\mathcal{Z}_{X \git \hat{\mathsf{G}}}^{\hat{\lambda}}$; this can be shown either by generalizing the above arguments or by a direct calculation using part (i). In any case, attached to the circle is the commutative Frobenius algebra
\[
\mathcal{Z}_{X \git \hat{\mathsf{G}}}^{\hat{\lambda}} (S^1) \simeq \Gamma_{\Lambda (X \git \mathsf{G})}(\uptau_{S^1}(\lambda)_{\mathbb{C}}^{-1}).
\]
The involution $p$ is defined using $\uptau_{\pi}^{\refl}(\hat{\lambda})$. The charge of the crosscap is
\[
Q = \sum_{\varsigma \in \hat{\mathsf{G}} \backslash \mathsf{G}} \sum_{x \in X^{\varsigma}} \hat{\lambda}_x([\varsigma \vert \varsigma]) l_{(x,\varsigma^2)}
\]
where $X^{\varsigma} = \{ x \in X \mid \varsigma \cdot x = x\}$, and so is a sum over the fixed point set
\[
(X \git \mathsf{G})^{\hat{\mathsf{G}} \slash \mathsf{G}} = \bigsqcup_{\varsigma \in \hat{\mathsf{G}} \backslash \mathsf{G}} X^{\varsigma}.
\]
This fixed point set plays the role of the orientifold plane of $X \git \mathsf{G}$.
\end{enumerate}
\end{Rems}

\footnotesize

\bibliographystyle{plain}
\bibliography{mybib}

 
\end{document}